\documentclass[12pt]{amsart}      %{article} was 12pt latex e
\usepackage{amssymb}
\usepackage{eucal}
\usepackage{amsmath}
\usepackage{amscd}
\usepackage{multicol}
\usepackage[all]{xy}           %xypic macro for latex2.09
\usepackage{graphicx}
\usepackage{color}
\usepackage{colordvi}
\usepackage{xspace}
\usepackage{axodraw}
\usepackage{epsfig}
\usepackage{float}
\begin{document}  %for latex 2.09

\newcommand{\nc}{\newcommand}
\newcommand{\delete}[1]{}
\nc{\dfootnote}[1]{{}}          %{{}}
\nc{\ffootnote}[1]{\dfootnote{#1}}
%\nc{\mfootnote}[1]{{}}        % Use this to suppress footnotes
\nc{\mfootnote}[1]{\footnote{#1}} % Use this to show footnotes
%\nc{\ofootnote}[1]{{}}        % Use this to suppress footnotes
\nc{\ofootnote}[1]{\footnote{\tiny Older version: #1}} % Use this to show footnotes

\nc{\mlabel}[1]{\label{#1}}  % Use this to suppress names
\nc{\mcite}[1]{\cite{#1}}  % Use this to suppress names
\nc{\mref}[1]{\ref{#1}}  % Use this to suppress names

\delete{
\nc{\mlabel}[1]{\label{#1}  % Use the next two lines to show names
{\hfill \hspace{1cm}{\bf{{\ }\hfill(#1)}}}}
\nc{\mcite}[1]{\cite{#1}{{\bf{{\ }(#1)}}}}  % Use this lines to show names
\nc{\mref}[1]{\ref{#1}{{\bf{{\ }(#1)}}}}  % Use this lines to show names
}

\nc{\mbibitem}[1]{\bibitem{#1}} % Use this to show number
%\nc{\mbibitem}[1]{\bibitem[\bf #1]{#1}} % Use this to show name
\nc{\mkeep}[1]{\marginpar{{\bf #1}}} % Use this to show marginpar
%\nc{\mkeep}[1]{{}}      % Use this to suppress marginpar

%%%%%%%%%%%%%%%%%%%%%%%% Statements
\newtheorem{theorem}{Theorem}[section]
\newtheorem{prop}[theorem]{Proposition}
\newtheorem{defn}[theorem]{Definition}
\newtheorem{lemma}[theorem]{Lemma}
\newtheorem{coro}[theorem]{Corollary}
\newtheorem{prop-def}{Proposition-Definition}[section]
\newtheorem{claim}{Claim}[section]
\newtheorem{remark}[theorem]{Remark}
\newtheorem{propprop}{Proposed Proposition}[section]
\newtheorem{conjecture}{Conjecture}
\newtheorem{exam}[theorem]{Example}
\newtheorem{assumption}{Assumption}
\newtheorem{condition}[theorem]{Assumption}

\renewcommand{\labelenumi}{{\rm(\alph{enumi})}}
\renewcommand{\theenumi}{\alph{enumi}}

\nc{\tred}[1]{\textcolor{red}{#1}}
\nc{\tblue}[1]{\textcolor{blue}{#1}}
\nc{\tgreen}[1]{\textcolor{green}{#1}}

%%%%%%%%%%%%%%%%%%%%%%% symbols
\nc{\adec}{\check{;}}
\nc{\dftimes}{\widetilde{\otimes}} \nc{\dfl}{\succ}
\nc{\dfr}{\prec} \nc{\dfc}{\circ} \nc{\dfb}{\bullet}
\nc{\dft}{\star} \nc{\dfcf}{{\mathbf k}} \nc{\spr}{\cdot}
\nc{\disp}[1]{\displaystyle{#1}}
\nc{\bin}[2]{ (_{\stackrel{\scs{#1}}{\scs{#2}}})}  %binomial coeff
\nc{\binc}[2]{ \left (\!\! \begin{array}{c} \scs{#1}\\
    \scs{#2} \end{array}\!\! \right )}  %binomial coeff
\nc{\bincc}[2]{  \left ( {\scs{#1} \atop
    \vspace{-.5cm}\scs{#2}} \right )}  %binomial coeff
\nc{\sarray}[2]{\begin{array}{c}#1 \vspace{.1cm}\\ \hline
    \vspace{-.35cm} \\ #2 \end{array}}
\nc{\bs}{\bar{S}} \nc{\dcup}{\stackrel{\bullet}{\cup}}
\nc{\dbigcup}{\stackrel{\bullet}{\bigcup}} \nc{\etree}{\big |}
\nc{\la}{\longrightarrow} \nc{\fe}{\'{e}} \nc{\rar}{\rightarrow}
\nc{\dar}{\downarrow} \nc{\dap}[1]{\downarrow
\rlap{$\scriptstyle{#1}$}} \nc{\uap}[1]{\uparrow
\rlap{$\scriptstyle{#1}$}} \nc{\defeq}{\stackrel{\rm def}{=}}
\nc{\dis}[1]{\displaystyle{#1}}
\nc{\dotcup}{\,\displaystyle{\bigcup^\bullet}\ }
\nc{\sdotcup}{\tiny{\displaystyle{\bigcup^\bullet}\ }}
\nc{\hcm}{\ \hat{,}\ }
\nc{\hcirc}{\hat{\circ}} \nc{\hts}{\hat{\shpr}}
\nc{\lts}{\stackrel{\leftarrow}{\shpr}}
\nc{\rts}{\stackrel{\rightarrow}{\shpr}} \nc{\lleft}{[}
\nc{\lright}{]} \nc{\uni}[1]{\tilde{#1}} \nc{\wor}[1]{\check{#1}}
\nc{\free}[1]{\bar{#1}} \nc{\den}[1]{\check{#1}} \nc{\lrpa}{\wr}
\nc{\curlyl}{\left \{ \begin{array}{c} {} \\ {} \end{array}
    \right .  \!\!\!\!\!\!\!}
\nc{\curlyr}{ \!\!\!\!\!\!\!
    \left . \begin{array}{c} {} \\ {} \end{array}
    \right \} }
\nc{\leaf}{\ell}       % number of leafs
\nc{\longmid}{\left | \begin{array}{c} {} \\ {} \end{array}
    \right . \!\!\!\!\!\!\!}
\nc{\ot}{\otimes} \nc{\sot}{{\scriptstyle{\ot}}}
\nc{\otm}{\overline{\ot}}
\nc{\ora}[1]{\stackrel{#1}{\rar}}
\nc{\ola}[1]{\stackrel{#1}{\la}}%${\Bbb Z}$
\nc{\scs}[1]{\scriptstyle{#1}} \nc{\mrm}[1]{{\rm #1}}
\nc{\margin}[1]{\marginpar{\rm #1}}   %{\rm #1}}
\nc{\dirlim}{\displaystyle{\lim_{\longrightarrow}}\,}
\nc{\invlim}{\displaystyle{\lim_{\longleftarrow}}\,}
\nc{\mvp}{\vspace{0.5cm}} \nc{\svp}{\vspace{2cm}}
\nc{\vp}{\vspace{8cm}} \nc{\proofbegin}{\noindent{\bf Proof: }}
%\nc{\proofbegin}{\begin{proof}} % AMS command
\nc{\proofend}{$\blacksquare$ \vspace{0.5cm}}
%\nc{\proofend}{\end{proof}} %AMS command
%\nc{\intg}[1]{\lceil{#1}\rceil}  %old free int ring
%\nc{\sha}{\scs{\mbox{\cyr X}}} %used to be \cyr
\nc{\sha}{{\mbox{\cyr X}}}  %used to be \cyr
\nc{\ncsha}{{\mbox{\cyr X}^{\mathrm NC}}} \nc{\ncshao}{{\mbox{\cyr
X}^{\mathrm NC,\,0}}}
\nc{\shpr}{\diamond}    %Shuffle product
\nc{\shprm}{\,\overline{\diamond}\,}    %Shuffle product
\nc{\shpro}{\diamond^0}    %Shuffle product
\nc{\shprr}{\diamond^r}     %product on controlled trees
\nc{\shprp}{\shpr_v}
\nc{\shprl}{\shpr_\ell}
\nc{\shprw}{\shpr_w}
\nc{\shpra}{\overline{\diamond}^r}
\nc{\shpru}{\check{\diamond}} \nc{\catpr}{\diamond_l}
\nc{\rcatpr}{\diamond_r} \nc{\lapr}{\diamond_a}
\nc{\sqcupm}{\ot}
\nc{\lepr}{\diamond_e} \nc{\vep}{\varepsilon} \nc{\labs}{\mid\!}
\nc{\rabs}{\!\mid} \nc{\hsha}{\widehat{\sha}}
\nc{\lsha}{\stackrel{\leftarrow}{\sha}}
\nc{\rsha}{\stackrel{\rightarrow}{\sha}} \nc{\lc}{\lfloor}
\nc{\rc}{\rfloor}
\nc{\lm}{\,\slash}
\nc{\rtm}{\backslash\,}
\nc{\sqmon}[1]{\langle #1\rangle}
\nc{\forest}{\calf} \nc{\ass}[1]{\alpha({#1})}
\nc{\altx}{\Lambda_X} \nc{\vecT}{\vec{T}} \nc{\onetree}{\, \bullet\, }
\nc{\Ao}{\check{A}}
\nc{\seta}{\underline{\Ao}}
\nc{\deltaa}{\overline{\delta}}
\nc{\trho}{\tilde{\rho}}

%%%%%%%%%%%%%%%%%%%%% roman fonts, in alphabetic order
\nc{\mmbox}[1]{\mbox{\ #1\ }} \nc{\ann}{\mrm{ann}}
\nc{\Aut}{\mrm{Aut}} \nc{\can}{\mrm{can}}
\nc{\bread}{\mrm{b}}
\nc{\colim}{\mrm{colim}}
\nc{\Cont}{\mrm{Cont}} \nc{\rchar}{\mrm{char}}
\nc{\cok}{\mrm{coker}} \nc{\dtf}{{R-{\rm tf}}} \nc{\dtor}{{R-{\rm
tor}}}
\renewcommand{\det}{\mrm{det}}
\nc{\depth}{{\mrm d}}
\nc{\Div}{{\mrm Div}} \nc{\End}{\mrm{End}} \nc{\Ext}{\mrm{Ext}}
\nc{\Fil}{\mrm{Fil}} \nc{\Frob}{\mrm{Frob}} \nc{\Gal}{\mrm{Gal}}
\nc{\GL}{\mrm{GL}} \nc{\Hom}{\mrm{Hom}} \nc{\hsr}{\mrm{H}}
\nc{\hpol}{\mrm{HP}} \nc{\id}{\mrm{id}}
\nc{\mht}{\mrm{h}}
\nc{\im}{\mrm{im}}
\nc{\incl}{\mrm{incl}} \nc{\length}{\mrm{length}}
\nc{\LR}{\mrm{LR}} \nc{\mchar}{\rm char}
\nc{\mapped}{operated\xspace}
\nc{\Mapped}{Operated\xspace}
\nc{\NC}{\mrm{NC}}
\nc{\mpart}{\mrm{part}} \nc{\ql}{{\QQ_\ell}} \nc{\qp}{{\QQ_p}}
\nc{\rank}{\mrm{rank}} \nc{\rba}{\rm{RBA }} \nc{\rbas}{\rm{RBAs }}
\nc{\rbw}{\rm{RBW }} \nc{\rbws}{\rm{RBWs }} \nc{\rcot}{\mrm{cot}}
\nc{\rest}{\rm{controlled}\xspace}
\nc{\rdef}{\mrm{def}} \nc{\rdiv}{{\rm div}} \nc{\rtf}{{\rm tf}}
\nc{\rtor}{{\rm tor}} \nc{\res}{\mrm{res}}
\nc{\Set}{{\mathbf{Set}}}
\nc{\SL}{\mrm{SL}}
\nc{\Spec}{\mrm{Spec}} \nc{\tor}{\mrm{tor}} \nc{\Tr}{\mrm{Tr}}
\nc{\mtr}{\mrm{sk}}

%%%%%%%%%%%%%%%%%% bold face
\nc{\ab}{\mathbf{Ab}} \nc{\Alg}{\mathbf{Alg}}
\nc{\Algo}{\mathbf{Alg}^0} \nc{\Bax}{\mathbf{Bax}}
\nc{\Baxo}{\mathbf{Bax}^0} \nc{\RB}{\mathbf{RB}}
\nc{\RBo}{\mathbf{RB}^0} \nc{\BRB}{\mathbf{RB}}
\nc{\Dend}{\mathbf{DD}} \nc{\bfk}{{\bf k}} \nc{\bfone}{{\bf 1}}
\nc{\base}[1]{{a_{#1}}} \nc{\detail}{\marginpar{\bf More detail}
    \noindent{\bf Need more detail!}
    \svp}
\nc{\Diff}{\mathbf{Diff}} \nc{\gap}{\marginpar{\bf
Incomplete}\noindent{\bf Incomplete!!}
    \svp}
\nc{\FMod}{\mathbf{FMod}} \nc{\mset}{\mathbf{MSet}}
\nc{\rb}{\mathrm{RB}} \nc{\Int}{\mathbf{Int}}
\nc{\Mon}{\mathbf{Mon}}
\nc{\Map}{\mathrm{Map}}
%\nc{\remark}{\noindent{\bf Remark: }}
\nc{\remarks}{\noindent{\bf Remarks: }} \nc{\Rep}{\mathbf{Rep}}
\nc{\Rings}{\mathbf{Rings}} \nc{\Sets}{\mathbf{Sets}}
\nc{\DT}{\mathbf{DT}}

%%%%%%%%%%%%%%%%%%%Bbb fonts
\nc{\BA}{{\mathbb A}} \nc{\CC}{{\mathbb C}} \nc{\DD}{{\mathbb D}}
\nc{\EE}{{\mathbb E}} \nc{\FF}{{\mathbb F}} \nc{\GG}{{\mathbb G}}
\nc{\HH}{{\mathbb H}} \nc{\LL}{{\mathbb L}} \nc{\NN}{{\mathbb N}}
\nc{\QQ}{{\mathbb Q}} \nc{\RR}{{\mathbb R}} \nc{\TT}{{\mathbb T}}
\nc{\VV}{{\mathbb V}} \nc{\ZZ}{{\mathbb Z}}

%%%%%%%%%%%%%%%%%%% cal fonts

\nc{\cala}{{\mathcal A}} \nc{\calc}{{\mathcal C}}
\nc{\cald}{{\mathcal D}} \nc{\cale}{{\mathcal E}}
\nc{\calf}{{\mathcal F}} \nc{\calfr}{{{\mathcal F}^{\,r}}}
\nc{\calfo}{{\mathcal F}^0} \nc{\calfro}{{\mathcal F}^{\,r,0}}
\nc{\oF}{\overline{F}}  \nc{\calg}{{\mathcal G}}
\nc{\calh}{{\mathcal H}} \nc{\cali}{{\mathcal I}}
\nc{\calj}{{\mathcal J}} \nc{\call}{{\mathcal L}}
\nc{\calm}{{\mathcal M}}
\nc{\oM}{\overline{M}}
\nc{\caln}{{\mathcal N}}
\nc{\calo}{{\mathcal O}} \nc{\calp}{{\mathcal P}}
\nc{\calr}{{\mathcal R}} \nc{\calt}{{\mathcal T}}
\nc{\caltr}{{\mathcal T}^{\,r}}
\nc{\calu}{{\mathcal U}} \nc{\calv}{{\mathcal V}}
\nc{\calw}{{\mathcal W}} \nc{\calx}{{\mathcal X}}
\nc{\CA}{\mathcal{A}}

%%%%%%%%%%%%%%%%%%  frak fonts
\nc{\fraka}{{\mathfrak a}} \nc{\frakB}{{\mathfrak B}}
\nc{\frakb}{{\mathfrak b}} \nc{\frakd}{{\mathfrak d}}
\nc{\oD}{\overline{D}}
\nc{\frakD}{{\mathcal D}}
\nc{\frakF}{{\mathfrak F}} \nc{\frakg}{{\mathfrak g}}
\nc{\frakI}{{\mathcal I}}
\nc{\frakL}{{\mathcal L}}
\nc{\frakm}{{\mathfrak m}}
\nc{\ofrakm}{\bar{\frakm}}
\nc{\frakM}{{\mathcal M}}
\nc{\frakMo}{{\mathfrak M}^0} \nc{\frakp}{{\mathfrak p}}
\nc{\frakP}{{\mathcal P}}
\nc{\frakR}{{\mathcal R}}
\nc{\frakS}{{\mathcal S}} \nc{\frakSo}{{\mathfrak S}^0}
\nc{\fraks}{{\mathfrak s}} \nc{\os}{\overline{\fraks}}
\nc{\frakT}{{\mathfrak T}}
\nc{\oT}{\overline{T}}
\nc{\frakV}{{\mathcal V}}
\nc{\oV}{\overline{V}}
\nc{\frakW}{{\mathcal W}}
\nc{\frakw}{{\mathfrak w}}
\nc{\oW}{\overline{W}}
%\nc{\frakx}{{\mathfrak x}}
\nc{\frakX}{{\mathfrak X}} \nc{\frakXo}{{\mathfrak X}^0}
\nc{\frakx}{{\mathbf x}}
%\nc{\frakTxo}{{\frakTx}^0}
\nc{\frakTx}{\frakT}      %All rooted trees, correspond to \ncsha(X)
\nc{\frakTa}{\frakT^a}        % rooted trees for \ncsha(A)
\nc{\frakTxo}{\frakTx^0}   % rooted trees for \ncshao(X)
\nc{\caltao}{\calt^{a,0}}   % rooted trees for \ncshao(A)
\nc{\ox}{\overline{\frakx}} \nc{\fraky}{{\mathfrak y}}
\nc{\frakz}{{\mathfrak z}} \nc{\oX}{\overline{X}}

\font\cyr=wncyr10

\nc{\redtext}[1]{\textcolor{red}{#1}}

%%%%%%%%small-scale trees
\def\s1tree{\!\!\includegraphics[scale=0.41]{1tree.eps}}
%%%%trees with 3 leaves
\def\sa2tree{\!\!\includegraphics[scale=0.41]{2tree.eps}}
\def\sb3tree{\!\!\includegraphics[scale=0.41]{3tree.eps}}
\def\sc4tree{\!\!\includegraphics[scale=0.41]{4tree.eps}}

%%%%trees with 2 leaves
\def\1tree{\!\!\includegraphics[scale=0.51]{1tree.eps}}
%%%%trees with 3 leaves
\def\2tree{\!\!\includegraphics[scale=0.51]{2tree.eps}}
\def\3tree{\!\!\includegraphics[scale=0.51]{3tree.eps}}
\def\4tree{\!\!\includegraphics[scale=0.51]{4tree.eps}}
%%%%trees with 4 leaves
\def\5tree{\!\!\includegraphics[scale=0.51]{5tree.eps}}
\def\6tree{\!\!\includegraphics[scale=0.51]{6tree.eps}}
\def\7tree{\!\!\includegraphics[scale=0.51]{7tree.eps}}
\def\8tree{\!\!\includegraphics[scale=0.51]{8tree.eps}}
\def\9tree{\!\!\includegraphics[scale=0.51]{9tree.eps}}
\def\a1tree{\!\!\includegraphics[scale=0.51]{10tree.eps}}
\def\b1tree{\!\!\includegraphics[scale=0.51]{11tree.eps}}
\def\c1tree{\!\!\includegraphics[scale=0.51]{12tree.eps}}
\def\d1tree{\!\!\includegraphics[scale=0.51]{13tree.eps}}
\def\e1tree{\!\!\includegraphics[scale=0.51]{14tree.eps}}
\def\f1tree{\!\!\includegraphics[scale=0.51]{15tree.eps}}
%%%%%decorated trees
\def\dec1tree{\!\!\includegraphics[scale=0.51]{dectree1.eps}}
\def\xtree{\!\!\includegraphics[scale=0.41]{xtree.eps}}
\def\xyztree{\!\!\includegraphics[scale=0.41]{xyztree.eps}}
%\def\xtd31{\!\!\includegraphics[scale=1]{xtd31.eps}}

%\def\xthj44{\!\!\includegraphics[scale=1]{xthj44.eps}}

%%%%%%%%%%%%%%%%%%%%%%%%%%%%%%%%%%%%%%%%%%
%%%%%%%%%% paths and trees
\def\ccxtree{\includegraphics[scale=1]{ccxtree}}
\def\ccxtreedec{\includegraphics[scale=1]{ccxtreedec}}
\def\ccetree{\includegraphics[scale=1]{ccetree}}
\def\ccllcrcr{\includegraphics[scale=1]{ccllcrcr}}
\def\cclclcrcr{\includegraphics[scale=.5]{cclclcrcr}}
\def\ccllxryr{\includegraphics[scale=.5]{ccllxryr}}
\def\cclxlyrzr{\includegraphics[scale=1]{cclxlyrzr}}
\def\ccIII{\includegraphics[scale=1]{ccIII}}

\def\lxtree{\includegraphics[scale=1]{lxtree}}
\def\xdtree{\includegraphics[scale=1]{xdtree}}
\def\xytree{\includegraphics[scale=1]{xytree}}
\def\lxdtree{\includegraphics[scale=1]{lxdtree}}
\def\dtree{\includegraphics[scale=0.5]{dtree}}
\def\xyleft{\includegraphics[scale=.5]{xyleft}}
\def\xyright{\includegraphics[scale=.7]{xyright}}
\def\yzdtree{\includegraphics[scale=.7]{yzdtree}}
\def\xyztreea{\includegraphics[scale=.7]{xyztree1}}
\def\xyztreeb{\includegraphics[scale=.7]{xyztree2}}
\def\xyztreec{\includegraphics[scale=.7]{xyztree3}}
\def\tableps{\includegraphics[scale=1.4]{tableps}}

\def\decTree{\!\!\!\includegraphics[scale=1.41]{decTree.eps}}
\def\kyldec31{\!\!\includegraphics[scale=1.5]{kyldec31.eps}}

\def\yldec31{\!\!\includegraphics[scale=1.5]{ntreedec3.eps}}

\def\xyldec43{\!\!\includegraphics[scale=1.0]{mtreedec2.eps}}

\def\DecTreeProd1{\!\!\includegraphics[scale=1.5]{DecTreeProd1.eps}}
\def\TreeProd1{\!\!\includegraphics[scale=1.5]{TreeProd1.eps}}
\def\MPath{\!\!\includegraphics[scale=0.8]{MPath.eps}}
\def\DecMPath{\!\!\includegraphics[scale=0.8]{DecMPath.eps}}
\def\RiseMPath{\!\!\includegraphics[scale=0.8]{RiseMPath.eps}}
\def\DecomMPath{\!\!\includegraphics[scale=1.2]{DecomMPath.eps}}
\def\BijDecMPath{\!\!\includegraphics[scale=0.8]{BijDecMPath.eps}}
\def\MPathProd{\!\!\includegraphics[scale=1]{MPathProd.eps}}

%%%%%%%%%%%%%%%%%%%%%%%%%%%% leaf trees

\def\xlb2{{\scalebox{0.40}{ %%%%%%%%%%%%%%%%%%%%%%%%%%%%%%%%%\tb2
\begin{picture}(30,42)(38,-38)
\SetWidth{0.5} \SetColor{Black} \Vertex(45,-3){5.66}
\SetWidth{1.0} \Line(45,-3)(45,-33) \SetWidth{0.5}
\Vertex(45,-33){5.66}
\put(50,-50){\em\huge x}
\end{picture}}}}

\def\ylb2{{\scalebox{0.40}{ %%%%%%%%%%%%%%%%%%%%%%%%%%%%%%%%%\tb2
\begin{picture}(30,42)(38,-38)
\SetWidth{0.5} \SetColor{Black} \Vertex(45,-3){5.66}
\SetWidth{1.0} \Line(45,-3)(45,-33) \SetWidth{0.5}
\Vertex(45,-33){5.66}
\put(50,-50){\em\huge y}
\end{picture}}}}

\def\xyld31{{\scalebox{0.40}{ %%%%%%%%%%%%%%%%%%%%%%%%%%%%%%%%%\td31
\begin{picture}(42,42)(23,-38)
\SetColor{Black} \SetWidth{0.5} \Vertex(45,-3){5.66}
\Vertex(30,-33){5.66} \Vertex(60,-33){5.66} \SetWidth{1.0}
\Line(45,-3)(30,-33) \Line(60,-33)(45,-3)
\put(20,-58){\em\huge x}
\put(50,-58){\em\huge y}
\end{picture}}}}

\def\xylf41{{\scalebox{0.40}{ %%%%%%%%%%%%%%%%%%%%%%%%%%%%%%%%%\tf41
\begin{picture}(52,80)(38,-10)
\SetColor{Black} \SetWidth{0.5} \Vertex(45,27){5.66}
\Vertex(45,-3){5.66} \SetWidth{1.0} \Line(45,27)(45,-3)
\put(45,-25){\em\huge x}
\SetWidth{0.5} \Vertex(60,57){5.66} \SetWidth{1.0}
\Line(45,27)(60,57) \SetWidth{0.5} \Vertex(75,27){5.66}
\SetWidth{1.0} \Line(75,27)(60,57)
\put(75,7){\em\huge y}
\end{picture}}}}

\def\xylg42{{\scalebox{0.40}{ %%%%%%%%%%%%%%%%%%%%%%%%%%%%%%%%%\tg42
\begin{picture}(52,80)(8,-10)
\SetColor{Black} \SetWidth{0.5} \Vertex(45,27){5.66}
\Vertex(45,-3){5.66} \SetWidth{1.0} \Line(45,27)(45,-3)
\put(39,-25){\em\huge y}
\SetWidth{0.5} \Vertex(15,27){5.66} \Vertex(30,57){5.66}
\SetWidth{1.0} \Line(15,27)(30,57) \Line(45,27)(30,57)
\put(5,2){\em\huge x}
\end{picture}}}}

%%%%%%%%%%%%%%%%%%% trees and angularly decorated trees

\def\ta1{{\scalebox{0.25}{ %%%%%%%%%%%%%%%%%%%%%%%%%%%%%%%%%\ta1
\begin{picture}(12,12)(38,-38)
\SetWidth{0.5} \SetColor{Black} \Vertex(45,-33){5.66}
\end{picture}}}}

\def\tb2{{\scalebox{0.40}{ %%%%%%%%%%%%%%%%%%%%%%%%%%%%%%%%%\tb2
\begin{picture}(12,42)(38,-38)
\SetWidth{0.5} \SetColor{Black} \Vertex(45,-3){5.66}
\SetWidth{1.0} \Line(45,-3)(45,-33) \SetWidth{0.5}
\Vertex(45,-33){5.66}
\end{picture}}}}

\def\tc3{{\scalebox{0.25}{ %%%%%%%%%%%%%%%%%%%%%%%%%%%%%%%%%\tc3
\begin{picture}(12,72)(38,-38)
\SetWidth{0.5} \SetColor{Black} \Vertex(45,27){5.66}
\SetWidth{1.0} \Line(45,27)(45,-3) \SetWidth{0.5}
\Vertex(45,-33){5.66} \SetWidth{1.0} \Line(45,-3)(45,-33)
\SetWidth{0.5} \Vertex(45,-3){5.66}
\end{picture}}}}

\def\td31{{\scalebox{0.25}{ %%%%%%%%%%%%%%%%%%%%%%%%%%%%%%%%%\td31
\begin{picture}(42,42)(23,-38)
\SetWidth{0.5} \SetColor{Black} \Vertex(45,-3){5.66}
\Vertex(30,-33){5.66} \Vertex(60,-33){5.66} \SetWidth{1.0}
\Line(45,-3)(30,-33) \Line(60,-33)(45,-3)
\end{picture}}}}

\def\xtd31{{\scalebox{0.35}{ %%%%%%%%%%%%%%%%%%%%%%%%%%%%%%%%%\td31
\begin{picture}(70,42)(13,-35)
\SetWidth{0.5} \SetColor{Black} \Vertex(45,-3){5.66}
\Vertex(30,-33){5.66} \Vertex(60,-33){5.66} \SetWidth{1.0}
\Line(45,-3)(30,-33) \Line(60,-33)(45,-3)
\put(38,-38){\em \huge x}
\end{picture}}}}

\def\ytd31{{\scalebox{0.35}{ %%%%%%%%%%%%%%%%%%%%%%%%%%%%%%%%%\td31
\begin{picture}(70,42)(13,-35)
\SetWidth{0.5} \SetColor{Black} \Vertex(45,-3){5.66}
\Vertex(30,-33){5.66} \Vertex(60,-33){5.66} \SetWidth{1.0}
\Line(45,-3)(30,-33) \Line(60,-33)(45,-3)
\put(38,-38){\em \huge y}
\end{picture}}}}

\def\xldec41r{{\scalebox{0.35}{ %%%%%%%%%%%%%%%%%%%%%%%%%%%%%%%%%\xldec41r
\begin{picture}(70,42)(13,-45)
\SetColor{Black}
\SetWidth{0.5} \Vertex(45,-3){5.66}
\Vertex(30,-33){5.66} \Vertex(60,-33){5.66}
\Vertex(60,-63){5.66}
\SetWidth{1.0}
\Line(45,-3)(30,-33) \Line(60,-33)(45,-3)
\Line(60,-33)(60,-63)
\put(38,-38){\em \huge x}

\end{picture}}}}

\def\xyrlong{{\scalebox{0.35}{ %%%%%%%%%%%%%%%%%%%%%%%%%%%%%%%%%\td31
\begin{picture}(70,72)(13,-48)
\SetColor{Black}
\SetWidth{0.5} \Vertex(45,-3){5.66}
\Vertex(30,-33){5.66} \Vertex(60,-33){5.66} \SetWidth{1.0}
\Line(45,-3)(30,-33) \Line(60,-33)(45,-3)
\put(38,-38){\em\huge x}
\SetWidth{0.5}
\Vertex(45,-63){5.66} \Vertex(75,-63){5.66} \SetWidth{1.0}
\Line(60,-33)(45,-63) \Line(60,-33)(75,-63)
\put(55,-63){\em\huge y}
\end{picture}}}}

\def\xyllong{{\scalebox{0.35}{ %%%%%%%%%%%%%%%%%%%%%%%%%%%%%%%%%\td31
\begin{picture}(70,72)(13,-48)
\SetColor{Black}
\SetWidth{0.5} \Vertex(45,-3){5.66}
\Vertex(30,-33){5.66} \Vertex(60,-33){5.66} \SetWidth{1.0}
\Line(45,-3)(30,-33) \Line(60,-33)(45,-3)
\put(40,-33){\em\huge y}
\SetWidth{0.5}
\Vertex(15,-63){5.66} \Vertex(45,-63){5.66} \SetWidth{1.0}
\Line(30,-33)(15,-63) \Line(30,-33)(45,-63)
\put(25,-63){\em\huge x}
\end{picture}}}}

\def\xyldec43{{\scalebox{0.35}{ %%%%%%%%%%%%%%%%%%%%%%%%%%%%%%%%%\xyldec43
\begin{picture}(70,62)(13,-25)
\SetColor{Black}
\SetWidth{0.5} \Vertex(45,-3){5.66}
\Vertex(15,-33){5.66} \Vertex(45,-38){5.66}
\Vertex(75,-33){5.66}
\SetWidth{1.0}
\Line(45,-3)(15,-33) \Line(45,-3)(45,-38)
\Line(45,-3)(74,-33)
\put(25,-33){\em\huge x}
\put(50,-33){\em\huge y}
\end{picture}}}}

\def\te4{{\scalebox{0.25}{ %%%%%%%%%%%%%%%%%%%%%%%%%%%%%%%%%\te4
\begin{picture}(12,102)(38,-8)
\SetWidth{0.5} \SetColor{Black} \Vertex(45,57){5.66}
\Vertex(45,-3){5.66} \Vertex(45,27){5.66} \Vertex(45,87){5.66}
\SetWidth{1.0} \Line(45,57)(45,27) \Line(45,-3)(45,27)
\Line(45,57)(45,87)
\end{picture}}}}

\def\tf41{{\scalebox{0.25}{ %%%%%%%%%%%%%%%%%%%%%%%%%%%%%%%%%\tf41
\begin{picture}(42,72)(38,-8)
\SetWidth{0.5} \SetColor{Black} \Vertex(45,27){5.66}
\Vertex(45,-3){5.66} \SetWidth{1.0} \Line(45,27)(45,-3)
\SetWidth{0.5} \Vertex(60,57){5.66} \SetWidth{1.0}
\Line(45,27)(60,57) \SetWidth{0.5} \Vertex(75,27){5.66}
\SetWidth{1.0} \Line(75,27)(60,57)
\end{picture}}}}

\def\tg42{{\scalebox{0.25}{ %%%%%%%%%%%%%%%%%%%%%%%%%%%%%%%%%\tg42
\begin{picture}(42,72)(8,-8)
\SetWidth{0.5} \SetColor{Black} \Vertex(45,27){5.66}
\Vertex(45,-3){5.66} \SetWidth{1.0} \Line(45,27)(45,-3)
\SetWidth{0.5} \Vertex(15,27){5.66} \Vertex(30,57){5.66}
\SetWidth{1.0} \Line(15,27)(30,57) \Line(45,27)(30,57)
\end{picture}}}}

\def\th43{{\scalebox{0.25}{ %%%%%%%%%%%%%%%%%%%%%%%%%%%%%%%%%\th43
\begin{picture}(42,42)(8,-8)
\SetWidth{0.5} \SetColor{Black} \Vertex(45,-3){5.66}
\Vertex(15,-3){5.66} \Vertex(30,27){5.66} \SetWidth{1.0}
\Line(15,-3)(30,27) \Line(45,-3)(30,27) \Line(30,27)(30,-3)
\SetWidth{0.5} \Vertex(30,-3){5.66}
\end{picture}}}}

\def\thII43{{\scalebox{0.25}{ %%%%%%%%%%%%%%%%%%%%%%%%%%%%%%%%%\th43
\begin{picture}(72,57) (68,-128)
    \SetWidth{0.5}
    \SetColor{Black}
    \Vertex(105,-78){5.66}
    \SetWidth{1.5}
    \Line(105,-78)(75,-123)
    \Line(105,-78)(105,-123)
    \Line(105,-78)(135,-123)
    \SetWidth{0.5}
    \Vertex(75,-123){5.66}
    \Vertex(105,-123){5.66}
    \Vertex(135,-123){5.66}
  \end{picture}
  }}}

\def\thj44{{\scalebox{0.25}{ %%%%%%%%%%%%%%%%%%%%%%%%%%%%%%%%%\thj44
\begin{picture}(42,72)(8,-8)
\SetWidth{0.5} \SetColor{Black} \Vertex(30,57){5.66}
\SetWidth{1.0} \Line(30,57)(30,27) \SetWidth{0.5}
\Vertex(30,27){5.66} \SetWidth{1.0} \Line(45,-3)(30,27)
\SetWidth{0.5} \Vertex(45,-3){5.66} \Vertex(15,-3){5.66}
\SetWidth{1.0} \Line(15,-3)(30,27)
\end{picture}}}}

\def\xthj44{{\scalebox{0.35}{ %%%%%%%%%%%%%%%%%%%%%%%%%%%%%%%%%\thj44
\begin{picture}(42,72)(8,-8)
\SetWidth{0.5} \SetColor{Black} \Vertex(30,57){5.66}
\SetWidth{1.0} \Line(30,57)(30,27) \SetWidth{0.5}
\Vertex(30,27){5.66} \SetWidth{1.0} \Line(45,-3)(30,27)
\SetWidth{0.5} \Vertex(45,-3){5.66} \Vertex(15,-3){5.66}
\SetWidth{1.0} \Line(15,-3)(30,27)
\put(25,-3){\em\huge x}
\end{picture}}}}

\def\ti5{{\scalebox{0.25}{ %%%%%%%%%%%%%%%%%%%%%%%%%%%%%%%%%\ti5
\begin{picture}(12,132)(23,-8)
\SetWidth{0.5} \SetColor{Black} \Vertex(30,117){5.66}
\SetWidth{1.0} \Line(30,117)(30,87) \SetWidth{0.5}
\Vertex(30,87){5.66} \Vertex(30,57){5.66} \Vertex(30,27){5.66}
\Vertex(30,-3){5.66} \SetWidth{1.0} \Line(30,-3)(30,27)
\Line(30,27)(30,57) \Line(30,87)(30,57)
\end{picture}}}}

\def\tj51{{\scalebox{0.25}{ %%%%%%%%%%%%%%%%%%%%%%%%%%%%%%%%%\tj51
\begin{picture}(42,102)(53,-38)
\SetWidth{0.5} \SetColor{Black} \Vertex(61,27){4.24}
\SetWidth{1.0} \Line(75,57)(90,27) \Line(60,27)(75,57)
\SetWidth{0.5} \Vertex(90,-3){5.66} \Vertex(60,27){5.66}
\Vertex(75,57){5.66} \Vertex(90,-33){5.66} \SetWidth{1.0}
\Line(90,-33)(90,-3) \Line(90,-3)(90,27) \SetWidth{0.5}
\Vertex(90,27){5.66}
\end{picture}}}}

\def\tk52{{\scalebox{0.25}{ %%%%%%%%%%%%%%%%%%%%%%%%%%%%%%%%%\tk52
\begin{picture}(42,102)(23,-8)
\SetWidth{0.5} \SetColor{Black} \Vertex(60,57){5.66}
\Vertex(45,87){5.66} \SetWidth{1.0} \Line(45,87)(60,57)
\SetWidth{0.5} \Vertex(30,57){5.66} \SetWidth{1.0}
\Line(30,57)(45,87) \SetWidth{0.5} \Vertex(30,-3){5.66}
\SetWidth{1.0} \Line(30,-3)(30,27) \SetWidth{0.5}
\Vertex(30,27){5.66} \SetWidth{1.0} \Line(30,57)(30,27)
\end{picture}}}}

\def\tl53{{\scalebox{0.25}{ %%%%%%%%%%%%%%%%%%%%%%%%%%%%%%%%%\tl53
\begin{picture}(42,102)(8,-8)
\SetWidth{0.5} \SetColor{Black} \Vertex(30,57){5.66}
\Vertex(30,27){5.66} \SetWidth{1.0} \Line(30,57)(30,27)
\SetWidth{0.5} \Vertex(30,87){5.66} \SetWidth{1.0}
\Line(30,27)(45,-3) \SetWidth{0.5} \Vertex(15,-3){5.66}
\SetWidth{1.0} \Line(15,-3)(30,27) \Line(30,57)(30,87)
\SetWidth{0.5} \Vertex(45,-3){5.66}
\end{picture}}}}

\def\tm54{{\scalebox{0.25}{ %%%%%%%%%%%%%%%%%%%%%%%%%%%%%%%%%\tm54
\begin{picture}(42,72)(8,-38)
\SetWidth{0.5} \SetColor{Black} \Vertex(30,-3){5.66}
\SetWidth{1.0} \Line(30,27)(30,-3) \Line(30,-3)(45,-33)
\SetWidth{0.5} \Vertex(15,-33){5.66} \SetWidth{1.0}
\Line(15,-33)(30,-3) \SetWidth{0.5} \Vertex(45,-33){5.66}
\SetWidth{1.0} \Line(30,-33)(30,-3) \SetWidth{0.5}
\Vertex(30,-33){5.66} \Vertex(30,27){5.66}
\end{picture}}}}

\def\tn55{{\scalebox{0.25}{ %%%%%%%%%%%%%%%%%%%%%%%%%%%%%%%%%\tn55
\begin{picture}(42,72)(8,-38)
\SetWidth{0.5} \SetColor{Black} \Vertex(15,-33){5.66}
\Vertex(45,-33){5.66} \Vertex(30,27){5.66} \SetWidth{1.0}
\Line(45,-33)(45,-3) \SetWidth{0.5} \Vertex(45,-3){5.66}
\Vertex(15,-3){5.66} \SetWidth{1.0} \Line(30,27)(45,-3)
\Line(15,-3)(30,27) \Line(15,-3)(15,-33)
\end{picture}}}}

\def\tp56{{\scalebox{0.25}{ %%%%%%%%%%%%%%%%%%%%%%%%%%%%%%%%%\tp56
\begin{picture}(66,111)(0,0)
\SetWidth{0.5} \SetColor{Black} \Vertex(30,66){5.66}
\Vertex(45,36){5.66} \SetWidth{1.0} \Line(30,66)(45,36)
\Line(15,36)(30,66) \SetWidth{0.5} \Vertex(30,6){5.66}
\Vertex(60,6){5.66} \SetWidth{1.0} \Line(60,6)(45,36)
\SetWidth{0.5}
\SetWidth{1.0} \Line(45,36)(30,6) \SetWidth{0.5}
\Vertex(15,36){5.66}
\end{picture}}}}

\def\tq57{{\scalebox{0.25}{ %%%%%%%%%%%%%%%%%%%%%%%%%%%%%%%%%\tq57
\begin{picture}(81,111)(0,0)
\SetWidth{0.5} \SetColor{Black} \Vertex(45,36){5.66}
\Vertex(30,6){5.66} \Vertex(60,6){5.66} \SetWidth{1.0}
\Line(60,6)(45,36) \SetWidth{0.5}
\SetWidth{1.0} \Line(45,36)(30,6) \SetWidth{0.5}
\Vertex(75,36){5.66} \SetWidth{1.0} \Line(45,36)(60,66)
\Line(60,66)(75,36) \SetWidth{0.5} \Vertex(60,66){5.66}
\end{picture}}}}

\def\tr58{{\scalebox{0.25}{ %%%%%%%%%%%%%%%%%%%%%%%%%%%%%%%%%\tr58
\begin{picture}(81,111)(0,0)
\SetWidth{0.5} \SetColor{Black} \Vertex(60,6){5.66}
\Vertex(75,36){5.66} \SetWidth{1.0} \Line(60,66)(75,36)
\SetWidth{0.5} \Vertex(60,66){5.66}
\SetWidth{1.0} \Line(60,36)(60,66) \Line(60,6)(60,36)
\SetWidth{0.5} \Vertex(60,36){5.66} \Vertex(45,36){5.66}
\SetWidth{1.0} \Line(60,66)(45,36)
\end{picture}}}}

\def\ts59{{\scalebox{0.25}{ %%%%%%%%%%%%%%%%%%%%%%%%%%%%%%%%%\ts59
\begin{picture}(81,111)(0,0)
\SetWidth{0.5} \SetColor{Black}
\Vertex(75,36){5.66} \SetWidth{1.0} \Line(60,66)(75,36)
\SetWidth{0.5} \Vertex(60,66){5.66}
\SetWidth{1.0} \Line(60,36)(60,66) \SetWidth{0.5}
\Vertex(60,36){5.66} \Vertex(45,36){5.66} \SetWidth{1.0}
\Line(60,66)(45,36) \Line(75,6)(75,36) \SetWidth{0.5}
\Vertex(75,6){5.66}
\end{picture}}}}

\def\tt591{{\scalebox{0.25}{ %%%%%%%%%%%%%%%%%%%%%%%%%%%%%%%%%\tt591
\begin{picture}(81,111)(0,0)
\SetWidth{0.5} \SetColor{Black}
\Vertex(75,36){5.66} \SetWidth{1.0} \Line(60,66)(75,36)
\SetWidth{0.5} \Vertex(60,66){5.66}
\SetWidth{1.0} \Line(60,36)(60,66) \SetWidth{0.5}
\Vertex(60,36){5.66} \Vertex(45,36){5.66} \SetWidth{1.0}
\Line(60,66)(45,36) \SetWidth{0.5} \Vertex(45,6){5.66}
\SetWidth{1.0} \Line(45,6)(45,36)
\end{picture}}}}

%%%%%%%%%%%%%%%%%%%%%%%%%%%%%%%%%%%%%%%%%%%%%%%%%%%%%%%%%%%%%%%Decorated trees%

\def\ydec31{\!\!\includegraphics[scale=0.5]{ydec31.eps}}

\def\kyldec31{\!\!\includegraphics[scale=0.5]{kyldec31.eps}}

\def\yldec31{\!\!\includegraphics[scale=0.5]{ntreedec3.eps}}
%\def\yldec31{\!\!\includegraphics[scale=0.5]{yldec31.eps}}

%\def\xldec41r{\!\!\includegraphics[scale=0.5]{xtreedec1.eps}}
%\def\xldec41r{\!\!\includegraphics[scale=0.5]{xldec41r.eps}}

%\def\xyldec43{\!\!\includegraphics[scale=0.5]{mtreedec2.eps}}
%\def\xyldec43{\!\!\includegraphics[scale=0.5]{xyldec43.eps}}

%\def\treeUNO{\includegraphics[scale=0.4]{treeUNO.eps}}
%\def\treeDUO{\includegraphics[scale=0.4]{treeDUO.eps}}
%\def\treeTRES{\includegraphics[scale=0.5]{treeTRES.eps}}

%%%%%%%%%%%%%%%%%%%%%%%%%%%%%%%%%%%%%%%%%%%%%%%%%%%%%%%%%%%%%%%%%%%%%%%%%%%%

\def\lta1{{\scalebox{0.25}{ %%%%%%%%%%%%%%%%%%%%%%%%%%%%%%%%%\lta1
\begin{picture}(0,45)(60,-15)
\SetWidth{1.5} \SetColor{Black} \Line(60,30)(60,-15)
\end{picture}}}}

\def\ltb2{{\scalebox{0.25}{ %%%%%%%%%%%%%%%%%%%%%%%%%%%%%%%%%\ltb2
\begin{picture}(12,45)(53,-15)
\SetWidth{1.5} \SetColor{Black} \Line(60,30)(60,-15)
\SetWidth{0.5} \Vertex(60,0){5.66}
\end{picture}}}}

\def\ltc3{{\scalebox{0.25}{ %%%%%%%%%%%%%%%%%%%%%%%%%%%%%%%%%\ltc3
\begin{picture}(12,75)(53,-15)
\SetWidth{0.5} \SetColor{Black} \Vertex(60,30){5.66}
\SetWidth{1.5} \Line(60,60)(60,-15) \SetWidth{0.5}
\Vertex(60,0){5.66}
\end{picture}}}}

\def\ltd31{{\scalebox{0.25}{ %%%%%%%%%%%%%%%%%%%%%%%%%%%%%%%%%\ltd31
\begin{picture}(75,90)(0,0)
\SetWidth{0.5} \SetColor{Black}
\Vertex(60,15){5.66} \SetWidth{1.5} \Line(45,45)(60,15)
\Line(60,15)(75,45) \Line(60,15)(60,0)
\end{picture}}}}

\delete{
\begin{picture}(75,90)(0,0)
\SetWidth{0.5} \SetColor{Black} \% \Vertex(60,15){5.66}
\SetWidth{1.5} \Line(45,45)(60,15) \Line(60,15)(75,45)
\Line(60,15)(60,0)
\end{picture}}

\def\lte4{{\scalebox{0.25}{ %%%%%%%%%%%%%%%%%%%%%%%%%%%%%%%%%\lte4
\begin{picture}(66,120)(0,0)
\SetWidth{0.5} \SetColor{Black}
\Vertex(60,45){5.66} \Vertex(60,75){5.66} \SetWidth{1.5}
\Line(60,105)(60,0) \SetWidth{0.5} \Vertex(60,15){5.66}
\end{picture}}}}

\def\ltf41l{{\scalebox{0.25}{ %%%%%%%%%%%%%%%%%%%%%%%%%%%%%%%%%\ltf41l
\begin{picture}(75,120)(0,0)
\SetWidth{0.5} \SetColor{Black}
\Vertex(60,15){5.66} \SetWidth{1.5} \Line(60,0)(60,15)
\Line(60,15)(45,45) \Line(60,15)(75,45) \SetWidth{0.5}
\Vertex(45,45){5.66} \SetWidth{1.5} \Line(45,45)(45,75)
\end{picture}}}}

\def\ltg41r{{\scalebox{0.25}{ %%%%%%%%%%%%%%%%%%%%%%%%%%%%%%%%%\ltg41r
\begin{picture}(81,120)(0,0)
\SetWidth{0.5} \SetColor{Black}
\Vertex(60,15){5.66} \SetWidth{1.5} \Line(75,45)(75,75)
\Line(60,0)(60,15) \Line(60,15)(45,45) \Line(60,15)(75,45)
\SetWidth{0.5} \Vertex(75,45){5.66}
\end{picture}}}}

\delete{
\begin{picture}(81,120)(0,0)
\SetWidth{0.5} \SetColor{Black}
\Vertex(60,15){5.66} \SetWidth{1.5} \Line(75,45)(75,75)
\Line(60,0)(60,15) \Line(60,15)(45,45) \Line(60,15)(75,45)
\SetWidth{0.5} \Vertex(75,45){5.66}
\end{picture}}

\def\lth42{{\scalebox{0.25}{ %%%%%%%%%%%%%%%%%%%%%%%%%%%%%%%%%\lth42
\begin{picture}(75,150)(0,0)
\SetWidth{0.5} \SetColor{Black}
\Vertex(60,45){5.66} \SetWidth{1.5} \Line(60,30)(60,45)
\Line(60,45)(45,75) \Line(60,45)(75,75) \SetWidth{0.5}
\Vertex(60,15){5.66} \SetWidth{1.5} \Line(60,0)(60,30)
\end{picture}}}}

\def\lti43{{\scalebox{0.25}{ %%%%%%%%%%%%%%%%%%%%%%%%%%%%%%%%%\lti43
\begin{picture}(75,150)(0,0)
\SetWidth{0.5} \SetColor{Black}
\SetWidth{1.5} \Line(60,30)(60,45) \SetWidth{0.5}
\Vertex(60,15){5.66} \SetWidth{1.5} \Line(60,0)(60,30)
\Line(60,15)(75,45) \Line(60,15)(45,45)
\end{picture}}}}

%%%%%%%%%%%%%%%%%%%%%%%%%%%%%%%%%%%%%%%%%%%%%%%%%%%

%%%%%%%%%%%%%%%%%%%%%%%%%%%%%%%%%%%%%%%%%%%55
%%%%%%%%% Motzkin paths

\def\ma1{{\scalebox{0.35}{ %%%%%%%%%%%%%%%%%%%%%%%%%%%%%%%%%\ta1
\begin{picture}(45,12)(25,-38)
\SetColor{Black}
\SetWidth{0.5} \Vertex(45,-33){5.66}
\end{picture}}}}

\def\mb2{{\scalebox{0.35}{ %%%%%%%%%%%%%%%%%%%%%%%%%%%%%%%%%\tb2
\begin{picture}(55,32)(5,-38)
\SetColor{Black}
\SetWidth{0.5} \Vertex(15,-33){5.66}
\SetWidth{1.0} \Line(15,-33)(45,-33)
\SetWidth{0.5} \Vertex(45,-33){5.66}
\end{picture}}}}

\def\xmb2{{\scalebox{0.35}{ %%%%%%%%%%%%%%%%%%%%%%%%%%%%%%%%%\tb2
\begin{picture}(55,32)(5,-38)
\SetColor{Black}
\SetWidth{0.5} \Vertex(15,-33){5.66}
\SetWidth{1.0} \Line(15,-33)(45,-33)
\put(25,-23){\em\huge x}
\SetWidth{0.5} \Vertex(45,-33){5.66}
\end{picture}}}}

\def\xpmb2{{\scalebox{0.35}{ %%%%%%%%%%%%%%%%%%%%%%%%%%%%%%%%%\tb2
\begin{picture}(55,32)(5,-38)
\SetColor{Black}
\SetWidth{0.5} \Vertex(15,-33){5.66}
\SetWidth{1.0} \Line(15,-33)(45,-33)
\put(25,-23){\em\huge $x'$}
\SetWidth{0.5} \Vertex(45,-33){5.66}
\end{picture}}}}

\def\motz1{{\scalebox{0.35}{ %%%%%%%%%%%%%%%%%%%%%%%%%%%%%%%%%\tb2
\begin{picture}(55,32)(5,-38)
\SetColor{Black}
\SetWidth{0.5} \Vertex(15,-33){5.66}
\SetWidth{1.0} \Line(15,-33)(45,-33)
\put(25,-23){\em\huge x_{i_1}}
\SetWidth{0.5} \Vertex(45,-33){5.66}
\end{picture}}}}

\def\motzb{{\scalebox{0.35}{ %%%%%%%%%%%%%%%%%%%%%%%%%%%%%%%%%\tb2
\begin{picture}(55,32)(5,-38)
\SetColor{Black}
\SetWidth{0.5} \Vertex(15,-33){5.66}
\SetWidth{1.0} \Line(15,-33)(45,-33)
\put(25,-23){\em\huge x_{i_{b-1}}}
\SetWidth{0.5} \Vertex(45,-33){5.66}
\end{picture}}}}

\def\ymb2{{\scalebox{0.35}{ %%%%%%%%%%%%%%%%%%%%%%%%%%%%%%%%%\tb2
\begin{picture}(55,32)(5,-38)
\SetColor{Black}
\SetWidth{0.5} \Vertex(15,-33){5.66}
\SetWidth{1.0} \Line(15,-33)(45,-33)
\put(25,-23){\em\huge y}
\SetWidth{0.5} \Vertex(45,-33){5.66}
\end{picture}}}}

% \xmb2 \xymc3 \xmf32 \xmh43

\def\mc3{{\scalebox{0.35}{ %%%%%%%%%%%%%%%%%%%%%%%%%%%%%%%%%\tc3
\begin{picture}(85,42)(5,-38)
\SetColor{Black}
\SetWidth{0.5} \Vertex(15,-33){5.66}
\SetWidth{1.0} \Line(15,-33)(45,-33)
\SetWidth{0.5} \Vertex(45,-33){5.66}
\SetWidth{1.0} \Line(45,-33)(75,-33)
\SetWidth{0.5} \Vertex(75,-33){5.66}
\end{picture}}}}

\def\xymc3{{\scalebox{0.35}{ %%%%%%%%%%%%%%%%%%%%%%%%%%%%%%%%%\tc3
\begin{picture}(85,42)(5,-38)
\SetColor{Black}
\SetWidth{0.5} \Vertex(15,-33){5.66}
\SetWidth{1.0} \Line(15,-33)(45,-33)
\put(25,-23){\em\huge x}
\SetWidth{0.5} \Vertex(45,-33){5.66}
\SetWidth{1.0} \Line(45,-33)(75,-33)
\put(55,-23){\em \huge y}
\SetWidth{0.5} \Vertex(75,-33){5.66}
\end{picture}}}}

\def\md31{{\scalebox{0.35}{ %%%%%%%%%%%%%%%%%%%%%%%%%%%%%%%%%\td31
\begin{picture}(100,42)(20,-38)
\SetColor{Black}
\SetWidth{0.5} \Vertex(60,-3){5.66}
\Vertex(30,-33){5.66} \Vertex(90,-33){5.66} \SetWidth{1.0}
\Line(60,-3)(30,-33) \Line(90,-33)(60,-3)
\end{picture}}}}

\def\amd31{{\scalebox{0.35}{ %%%%%%%%%%%%%%%%%%%%%%%%%%%%%%%%%\td31
\begin{picture}(100,42)(20,-38)
\SetColor{Black}
\SetWidth{0.5} \Vertex(60,-3){5.66}
\Vertex(30,-33){5.66} \Vertex(90,-33){5.66} \SetWidth{1.0}
\Line(60,-3)(30,-33) \put(30,-18){\em\huge $\omega$}
\Line(90,-33)(60,-3) \put(85,-18){\em\huge $\omega$}
\end{picture}}}}

\def\bmd31{{\scalebox{0.35}{ %%%%%%%%%%%%%%%%%%%%%%%%%%%%%%%%%\td31
\begin{picture}(100,42)(20,-38)
\SetColor{Black}
\SetWidth{0.5} \Vertex(60,-3){5.66}
\Vertex(30,-33){5.66} \Vertex(90,-33){5.66} \SetWidth{1.0}
\Line(60,-3)(30,-33) \put(30,-18){\em\huge $\beta$}
\Line(90,-33)(60,-3) \put(85,-18){\em\huge $\beta$}
\end{picture}}}}

\def\me4{{\scalebox{0.35}{ %%%%%%%%%%%%%%%%%%%%%%%%%%%%%%%%%\te4
\begin{picture}(110,42)(0,-38)
\SetColor{Black}
\SetWidth{0.5} \Vertex(10,-33){5.66}
\SetWidth{1.0} \Line(10,-33)(40,-33)
\SetWidth{0.5} \Vertex(40,-33){5.66}
\SetWidth{1.0} \Line(40,-33)(70,-33)
\SetWidth{0.5} \Vertex(70,-33){5.66}
\SetWidth{1.0} \Line(70,-33)(100,-33)
\SetWidth{0.5} \Vertex(100,-33){5.66}
\end{picture}}}}

\def\mf41{{\scalebox{0.35}{ %%%%%%%%%%%%%%%%%%%%%%%%%%%%%%%%%\tf41
\begin{picture}(110,42)(0,-38)
\SetColor{Black}
\SetWidth{0.5} \Vertex(10,-33){5.66}
\SetWidth{1.0} \Line(10,-33)(40,-33)
\SetWidth{0.5} \Vertex(40,-33){5.66}
\Vertex(70,-3){5.66} \Vertex(100,-33){5.66}
\SetWidth{1.0}
\Line(40,-33)(70,-3) \Line(100,-33)(70,-3)
\end{picture}}}}

\def\xmf41{{\scalebox{0.35}{ %%%%%%%%%%%%%%%%%%%%%%%%%%%%%%%%%\tf41
\begin{picture}(110,42)(0,-38)
\SetColor{Black}
\SetWidth{0.5} \Vertex(10,-33){5.66}
\SetWidth{1.0} \Line(10,-33)(40,-33)
\put(20,-23){\em\huge x}
\SetWidth{0.5} \Vertex(40,-33){5.66}
\Vertex(70,-3){5.66} \Vertex(100,-33){5.66}
\SetWidth{1.0}
\Line(40,-33)(70,-3) \Line(100,-33)(70,-3)
\end{picture}}}}

\def\axmf41{{\scalebox{0.35}{ %%%%%%%%%%%%%%%%%%%%%%%%%%%%%%%%%\tf41
\begin{picture}(110,42)(0,-38)
\SetColor{Black}
\SetWidth{0.5} \Vertex(10,-33){5.66}
\SetWidth{1.0} \Line(10,-33)(40,-33)
\put(20,-23){\em\huge x}
\SetWidth{0.5} \Vertex(40,-33){5.66}
\Vertex(70,-3){5.66} \Vertex(100,-33){5.66}
\SetWidth{1.0}
\Line(40,-33)(70,-3) \Line(100,-33)(70,-3)
\put(40,-18){\em \huge $\omega$}
\put(90,-18){\em \huge $\omega$}
\end{picture}}}}

\def\mg42{{\scalebox{0.35}{ %%%%%%%%%%%%%%%%%%%%%%%%%%%%%%%%%\tg42
\begin{picture}(110,42)(0,-38)
\SetColor{Black}
\SetWidth{0.5} \Vertex(10,-33){5.66}
\SetWidth{1.0} \Line(70,-33)(100,-33)
\SetWidth{0.5} \Vertex(70,-33){5.66}
\Vertex(40,-3){5.66} \Vertex(100,-33){5.66}
\SetWidth{1.0}
\Line(10,-33)(40,-3) \Line(70,-33)(40,-3)
\end{picture}}}}

\def\mh43{{\scalebox{0.35}{ %%%%%%%%%%%%%%%%%%%%%%%%%%%%%%%%%\th43
\begin{picture}(110,42)(0,-38)
\SetColor{Black}
\SetWidth{0.5} \Vertex(10,-33){5.66}
\SetWidth{1.0} \Line(10,-33)(40,-3)
\SetWidth{0.5} \Vertex(40,-3){5.66}
\SetWidth{1.0} \Line(40,-3)(70,-3)
\SetWidth{0.5} \Vertex(70,-3){5.66}
\SetWidth{1.0} \Line(70,-3)(100,-33)
\SetWidth{0.5} \Vertex(100,-33){5.66}
\end{picture}}}}

\def\xmh43{{\scalebox{0.35}{ %%%%%%%%%%%%%%%%%%%%%%%%%%%%%%%%%\th43
\begin{picture}(110,42)(0,-38)
\SetColor{Black}
\SetWidth{0.5} \Vertex(10,-33){5.66}
\SetWidth{1.0} \Line(10,-33)(40,-3)
\SetWidth{0.5} \Vertex(40,-3){5.66}
\SetWidth{1.0} \Line(40,-3)(70,-3)
\put(50,3){\em\huge x}
\SetWidth{0.5} \Vertex(70,-3){5.66}
\SetWidth{1.0} \Line(70,-3)(100,-33)
\SetWidth{0.5} \Vertex(100,-33){5.66}
\end{picture}}}}

\def\ymh43{{\scalebox{0.35}{ %%%%%%%%%%%%%%%%%%%%%%%%%%%%%%%%%\th43
\begin{picture}(110,42)(0,-38)
\SetColor{Black}
\SetWidth{0.5} \Vertex(10,-33){5.66}
\SetWidth{1.0} \Line(10,-33)(40,-3)
\SetWidth{0.5} \Vertex(40,-3){5.66}
\SetWidth{1.0} \Line(40,-3)(70,-3)
\put(50,3){\em\huge y}
\SetWidth{0.5} \Vertex(70,-3){5.66}
\SetWidth{1.0} \Line(70,-3)(100,-33)
\SetWidth{0.5} \Vertex(100,-33){5.66}
\end{picture}}}}

\def\axmh43{{\scalebox{0.35}{ %%%%%%%%%%%%%%%%%%%%%%%%%%%%%%%%%\th43
\begin{picture}(110,42)(0,-38)
\SetColor{Black}
\SetWidth{0.5} \Vertex(10,-33){5.66}
\SetWidth{1.0} \Line(10,-33)(40,-3) \put(13,-17){\em \huge $\omega$}
\SetWidth{0.5} \Vertex(40,-3){5.66}
\SetWidth{1.0} \Line(40,-3)(70,-3)
\put(50,3){\em\huge x}
\SetWidth{0.5} \Vertex(70,-3){5.66}
\SetWidth{1.0} \Line(70,-3)(100,-33) \put(90,-17){\em\huge $\omega$}
\SetWidth{0.5} \Vertex(100,-33){5.66}
\end{picture}}}}

\def\mi5{{\scalebox{0.35}{ %%%%%%%%%%%%%%%%%%%%%%%%%%%%%%%%%\th43
\begin{picture}(140,42)(0,-38)
\SetColor{Black}
\SetWidth{0.5} \Vertex(10,-33){5.66}
\SetWidth{1.0} \Line(10,-33)(40,-33)
\SetWidth{0.5} \Vertex(40,-33){5.66}
\SetWidth{1.0} \Line(40,-33)(70,-33)
\SetWidth{0.5} \Vertex(70,-33){5.66}
\SetWidth{1.0} \Line(70,-33)(100,-33)
\SetWidth{0.5} \Vertex(100,-33){5.66}
\SetWidth{1.0} \Line(100,-33)(130,-33)
\SetWidth{0.5} \Vertex(130,-33){5.66}
\end{picture}}}}

\def\mj51{{\scalebox{0.35}{ %%%%%%%%%%%%%%%%%%%%%%%%%%%%%%%%%\th43
\begin{picture}(140,42)(0,-38)
\SetColor{Black}
\SetWidth{0.5} \Vertex(10,-33){5.66}
\SetWidth{1.0} \Line(10,-33)(40,-3)
\SetWidth{0.5} \Vertex(40,-3){5.66}
\SetWidth{1.0} \Line(40,-3)(70,-33)
\SetWidth{0.5} \Vertex(70,-33){5.66}
\SetWidth{1.0} \Line(70,-33)(100,-33)
\SetWidth{0.5} \Vertex(100,-33){5.66}
\SetWidth{1.0} \Line(100,-33)(130,-33)
\SetWidth{0.5} \Vertex(130,-33){5.66}
\end{picture}}}}

\def\mk52{{\scalebox{0.35}{ %%%%%%%%%%%%%%%%%%%%%%%%%%%%%%%%%\th43
\begin{picture}(140,42)(0,-38)
\SetColor{Black}
\SetWidth{0.5} \Vertex(10,-33){5.66}
\SetWidth{1.0} \Line(10,-33)(40,-33)
\SetWidth{0.5} \Vertex(40,-33){5.66}
\SetWidth{1.0} \Line(40,-33)(70,-3)
\SetWidth{0.5} \Vertex(70,-3){5.66}
\SetWidth{1.0} \Line(70,-3)(100,-33)
\SetWidth{0.5} \Vertex(100,-33){5.66}
\SetWidth{1.0} \Line(100,-33)(130,-33)
\SetWidth{0.5} \Vertex(130,-33){5.66}
\end{picture}}}}

\def\ml53{{\scalebox{0.35}{ %%%%%%%%%%%%%%%%%%%%%%%%%%%%%%%%%\th43
\begin{picture}(140,42)(0,-38)
\SetColor{Black}
\SetWidth{0.5} \Vertex(10,-33){5.66}
\SetWidth{1.0} \Line(10,-33)(40,-33)
\SetWidth{0.5} \Vertex(40,-33){5.66}
\SetWidth{1.0} \Line(40,-33)(70,-33)
\SetWidth{0.5} \Vertex(70,-33){5.66}
\SetWidth{1.0} \Line(70,-33)(100,-3)
\SetWidth{0.5} \Vertex(100,-3){5.66}
\SetWidth{1.0} \Line(100,-3)(130,-33)
\SetWidth{0.5} \Vertex(130,-33){5.66}
\end{picture}}}}

\def\mm54{{\scalebox{0.35}{ %%%%%%%%%%%%%%%%%%%%%%%%%%%%%%%%%\th43
\begin{picture}(140,42)(0,-38)
\SetColor{Black}
\SetWidth{0.5} \Vertex(10,-33){5.66}
\SetWidth{1.0} \Line(10,-33)(40,-3)
\SetWidth{0.5} \Vertex(40,-3){5.66}
\SetWidth{1.0} \Line(40,-3)(70,-3)
\SetWidth{0.5} \Vertex(70,-3){5.66}
\SetWidth{1.0} \Line(70,-3)(100,-33)
\SetWidth{0.5} \Vertex(100,-33){5.66}
\SetWidth{1.0} \Line(100,-33)(130,-33)
\SetWidth{0.5} \Vertex(130,-33){5.66}
\end{picture}}}}

\def\xymm54{{\scalebox{0.35}{ %%%%%%%%%%%%%%%%%%%%%%%%%%%%%%%%%\th43
\begin{picture}(140,42)(0,-38)
\SetColor{Black}
\SetWidth{0.5} \Vertex(10,-33){5.66}
\SetWidth{1.0} \Line(10,-33)(40,-3)
\SetWidth{0.5} \Vertex(40,-3){5.66}
\SetWidth{1.0} \Line(40,-3)(70,-3)
\put(50,3){\em\huge x}
\SetWidth{0.5} \Vertex(70,-3){5.66}
\SetWidth{1.0} \Line(70,-3)(100,-33)
\SetWidth{0.5} \Vertex(100,-33){5.66}
\SetWidth{1.0} \Line(100,-33)(130,-33)
\put(110,-27){\em\huge y}
\SetWidth{0.5} \Vertex(130,-33){5.66}
\end{picture}}}}

\def\mn55{{\scalebox{0.35}{ %%%%%%%%%%%%%%%%%%%%%%%%%%%%%%%%%\th43
\begin{picture}(140,42)(0,-38)
\SetColor{Black}
\SetWidth{0.5} \Vertex(10,-33){5.66}
\SetWidth{1.0} \Line(10,-33)(40,-33)
\SetWidth{0.5} \Vertex(40,-33){5.66}
\SetWidth{1.0} \Line(40,-33)(70,-3)
\SetWidth{0.5} \Vertex(70,-3){5.66}
\SetWidth{1.0} \Line(70,-3)(100,-3)
\SetWidth{0.5} \Vertex(100,-3){5.66}
\SetWidth{1.0} \Line(100,-3)(130,-33)
\SetWidth{0.5} \Vertex(130,-33){5.66}
\end{picture}}}}

\def\xymn55{{\scalebox{0.35}{ %%%%%%%%%%%%%%%%%%%%%%%%%%%%%%%%%\th43
\begin{picture}(140,42)(0,-38)
\SetColor{Black}
\SetWidth{0.5} \Vertex(10,-33){5.66}
\SetWidth{1.0} \Line(10,-33)(40,-33)
\put(20,-27){\em\huge x}
\SetWidth{0.5} \Vertex(40,-33){5.66}
\SetWidth{1.0} \Line(40,-33)(70,-3)
\SetWidth{0.5} \Vertex(70,-3){5.66}
\SetWidth{1.0} \Line(70,-3)(100,-3)
\put(80,3){\em\huge y}
\SetWidth{0.5} \Vertex(100,-3){5.66}
\SetWidth{1.0} \Line(100,-3)(130,-33)
\SetWidth{0.5} \Vertex(130,-33){5.66}
\end{picture}}}}

\def\mo56{{\scalebox{0.35}{ %%%%%%%%%%%%%%%%%%%%%%%%%%%%%%%%%\th43
\begin{picture}(140,42)(0,-38)
\SetColor{Black}
\SetWidth{0.5} \Vertex(10,-33){5.66}
\SetWidth{1.0} \Line(10,-33)(40,-3)
\SetWidth{0.5} \Vertex(40,-3){5.66}
\SetWidth{1.0} \Line(40,-3)(70,-3)
\SetWidth{0.5} \Vertex(70,-3){5.66}
\SetWidth{1.0} \Line(70,-3)(100,-3)
\SetWidth{0.5} \Vertex(100,-3){5.66}
\SetWidth{1.0} \Line(100,-3)(130,-33)
\SetWidth{0.5} \Vertex(130,-33){5.66}
\end{picture}}}}

\def\xymo56{{\scalebox{0.35}{ %%%%%%%%%%%%%%%%%%%%%%%%%%%%%%%%%\th43
\begin{picture}(140,42)(0,-38)
\SetColor{Black}
\SetWidth{0.5} \Vertex(10,-33){5.66}
\SetWidth{1.0} \Line(10,-33)(40,-3)
\SetWidth{0.5} \Vertex(40,-3){5.66}
\SetWidth{1.0} \Line(40,-3)(70,-3)
\put(50,3){\em\huge x}
\SetWidth{0.5} \Vertex(70,-3){5.66}
\SetWidth{1.0} \Line(70,-3)(100,-3)
\put(80,3){\em\huge y}
\SetWidth{0.5} \Vertex(100,-3){5.66}
\SetWidth{1.0} \Line(100,-3)(130,-33)
\SetWidth{0.5} \Vertex(130,-33){5.66}
\end{picture}}}}

\def\mpp57{{\scalebox{0.35}{ %%%%%%%%%%%%%%%%%%%%%%%%%%%%%%%%%\th43
\begin{picture}(140,42)(0,-38)
\SetColor{Black}
\SetWidth{0.5} \Vertex(10,-33){5.66}
\SetWidth{1.0} \Line(10,-33)(40,-3)
\SetWidth{0.5} \Vertex(40,-3){5.66}
\SetWidth{1.0} \Line(40,-3)(70,-33)
\SetWidth{0.5} \Vertex(70,-33){5.66}
\SetWidth{1.0} \Line(70,-33)(100,-3)
\SetWidth{0.5} \Vertex(100,-3){5.66}
\SetWidth{1.0} \Line(100,-3)(130,-33)
\SetWidth{0.5} \Vertex(130,-33){5.66}
\end{picture}}}}

\def\abmpp57{{\scalebox{0.35}{ %%%%%%%%%%%%%%%%%%%%%%%%%%%%%%%%%\th43
\begin{picture}(140,62)(0,-38)
\SetColor{Black}
\SetWidth{0.5} \Vertex(10,-33){5.66}
\SetWidth{1.0} \Line(10,-33)(40,-3) \put(12,-13){\em \huge $\omega$}
\SetWidth{0.5} \Vertex(40,-3){5.66}
\SetWidth{1.0} \Line(40,-3)(70,-33) \put(50,-13){\em\huge $\omega$}
\SetWidth{0.5} \Vertex(70,-33){5.66}
\SetWidth{1.0} \Line(70,-33)(100,-3) \put(75,-13){\em\huge $\beta$}
\SetWidth{0.5} \Vertex(100,-3){5.66}
\SetWidth{1.0} \Line(100,-3)(130,-33) \put(120,-13){\em\huge $\beta$}
\SetWidth{0.5} \Vertex(130,-33){5.66}
\end{picture}}}}

\def\mq58{{\scalebox{0.35}{ %%%%%%%%%%%%%%%%%%%%%%%%%%%%%%%%%\th43
\begin{picture}(140,62)(0,-38)
\SetColor{Black}
\SetWidth{0.5} \Vertex(10,-33){5.66}
\SetWidth{1.0} \Line(10,-33)(40,-3)
\SetWidth{0.5} \Vertex(40,-3){5.66}
\SetWidth{1.0} \Line(40,-3)(70,23)
\SetWidth{0.5} \Vertex(70,23){5.66}
\SetWidth{1.0} \Line(70,23)(100,-3)
\SetWidth{0.5} \Vertex(100,-3){5.66}
\SetWidth{1.0} \Line(100,-3)(130,-33)
\SetWidth{0.5} \Vertex(130,-33){5.66}
\end{picture}}}}

\def\abmq58{{\scalebox{0.35}{ %%%%%%%%%%%%%%%%%%%%%%%%%%%%%%%%%\th43
\begin{picture}(140,62)(0,-38)
\SetColor{Black}
\SetWidth{0.5} \Vertex(10,-33){5.66}
\SetWidth{1.0} \Line(10,-33)(40,-3)
\put(10,-13){\em\huge $\omega$}
\SetWidth{0.5} \Vertex(40,-3){5.66}
\SetWidth{1.0} \Line(40,-3)(70,23)
\put(40,13){\em\huge $\beta$}
\SetWidth{0.5} \Vertex(70,23){5.66}
\SetWidth{1.0} \Line(70,23)(100,-3) \put(90,13){\em\huge$\beta$}
\SetWidth{0.5} \Vertex(100,-3){5.66}
\SetWidth{1.0} \Line(100,-3)(130,-33)
\put(120,-13){\em\huge$\omega$}
\SetWidth{0.5} \Vertex(130,-33){5.66}
\end{picture}}}}

\def\uuxdyd{{\scalebox{0.35}{ %%%%%%%%%%%%%%%%%%%%%%%%%%%%%%%%%\th43
\begin{picture}(200,62)(-30,-68)
\SetColor{Black}
\SetWidth{0.5} \Vertex(10,-33){5.66}
\SetWidth{1.0} \Line(10,-33)(40,-3)
\SetWidth{0.5} \Vertex(40,-3){5.66}
\SetWidth{1.0} \Line(40,-3)(70,-3)
\put(50,3){\em\huge x}
\SetWidth{0.5} \Vertex(70,-3){5.66}
\SetWidth{1.0} \Line(70,-3)(100,-33)
\SetWidth{0.5} \Vertex(100,-33){5.66}
\SetWidth{1.0} \Line(100,-33)(130,-33)
\put(110,-27){\em\huge y}
\SetWidth{0.5} \Vertex(130,-33){5.66}
\Vertex(-20,-63){5.66} \Vertex(160,-63){5.66}
\SetWidth{1.0} \Line(-20,-63)(10,-33)
\Line(130,-33)(160,-63)
\end{picture}}}}

\def\uxuydd{{\scalebox{0.35}{ %%%%%%%%%%%%%%%%%%%%%%%%%%%%%%%%%\th43
\begin{picture}(200,62)(-30,-68)
\SetColor{Black}
\SetWidth{0.5} \Vertex(10,-33){5.66}
\SetWidth{1.0} \Line(10,-33)(40,-33)
\put(20,-27){\em\huge x}
\SetWidth{0.5} \Vertex(40,-33){5.66}
\SetWidth{1.0} \Line(40,-33)(70,-3)
\SetWidth{0.5} \Vertex(70,-3){5.66}
\SetWidth{1.0} \Line(70,-3)(100,-3)
\put(80,3){\em\huge y}
\SetWidth{0.5} \Vertex(100,-3){5.66}
\SetWidth{1.0} \Line(100,-3)(130,-33)
\SetWidth{0.5} \Vertex(130,-33){5.66}
\Vertex(-20,-63){5.66} \Vertex(160,-63){5.66}
\SetWidth{1.0} \Line(-20,-63)(10,-33)
\Line(130,-33)(160,-63)
\end{picture}}}}

\def\bigdecm{{\scalebox{0.5}{ %%%%%%%%%%%%%%%%%%%%%%%%%%%%%%%%%\bigdecm
\begin{picture}(140,82)(250,-10)
\SetColor{Black}
\DashLine(10,-33)(610,-33){10}
%\DashLine(10,-3)(610,-3){10}
%\DashLine(10,23)(610,23){10}
%\DashLine(10,53)(610,53){10}

\SetColor{Black}
\SetWidth{0.5} \Vertex(10,-33){5.66}
\SetWidth{1.0} \Line(10,-33)(40,-3)
\put(10,-13){\em\huge $\alpha$}
\SetWidth{0.5} \Vertex(40,-3){5.66}
\SetWidth{1.0} \Line(40,-3)(70,23)
\put(40,13){\em\huge $\beta$}
\SetWidth{0.5} \Vertex(70,23){5.66}
\SetWidth{1.0} \Line(70,23)(100,23)
\put(80,33){\em\huge $a$}
\SetWidth{0.5} \Vertex(100,23){5.66}
\SetWidth{1.0} \Line(100,23)(130,53)
\put(100,43){\em\huge $\gamma$}
\SetWidth{0.5} \Vertex(130,53){5.66}
\SetWidth{1.0} \Line(130,53)(160,53)
\put(140,63){\em\huge $b$}
\SetWidth{0.5} \Vertex(160,53){5.66}
\SetWidth{1.0} \Line(160,53)(190,53)
\put(170,63){\em\huge $c$}
\SetWidth{0.5} \Vertex(190,53){5.66}
\SetWidth{1.0} \Line(190,53)(220,23)
\put(205,43){\em\huge $\gamma$}
\SetWidth{0.5} \Vertex(220,23){5.66}
\SetWidth{1.0} \Line(220,23)(250,23)
\put(230,33){\em\huge $d$}
\SetWidth{0.5} \Vertex(250,23){5.66}
\SetWidth{1.0} \Line(250,23)(280,-3)
\put(265,13){\em\huge $\beta$}
\SetWidth{0.5} \Vertex(280,-3){5.66}
\SetWidth{1.0} \Line(280,-3)(310,-3)
\put(290,03){\em\huge $e$}
\SetWidth{0.5} \Vertex(310,-3){5.66}
\SetWidth{1.0} \Line(310,-3)(340,23)
\put(315,13){\em\huge $\delta$}
\SetWidth{0.5} \Vertex(340,23){5.66}
\SetWidth{1.0} \Line(340,23)(370,23)
\put(350,33){\em\huge $f$}
\SetWidth{0.5} \Vertex(370,23){5.66}
\SetWidth{1.0} \Line(370,23)(400,53)
\put(378,38){\em\huge $\sigma$}
\SetWidth{0.5} \Vertex(400,53){5.66}
\SetWidth{1.0} \Line(400,53)(430,53)
\put(410,63){\em\huge $g$}
\SetWidth{0.5} \Vertex(430,53){5.66}
\SetWidth{1.0} \Line(430,53)(460,23)
\put(450,38){\em\huge $\sigma$}
\SetWidth{0.5} \Vertex(460,23){5.66}
\SetWidth{1.0} \Line(460,23)(490,53)
\put(468,38){\em\huge $\tau$}
\SetWidth{0.5} \Vertex(490,53){5.66}
\SetWidth{1.0} \Line(490,53)(520,53)
\put(500,63){\em\huge $h$}
\SetWidth{0.5} \Vertex(520,53){5.66}
\SetWidth{1.0} \Line(520,53)(550,23)
\put(540,38){\em\huge $\tau$}
\SetWidth{0.5} \Vertex(550,23){5.66}
\SetWidth{1.0} \Line(550,23)(580,-3)
\put(570,8){\em\huge $\delta$}
\SetWidth{0.5} \Vertex(580,-3){5.66}
\SetWidth{1.0} \Line(580,-3)(610,-33)
\put(600,-18){\em\huge $\alpha$}
\SetWidth{0.5} \Vertex(610,-33){5.66}

\end{picture}}}}

\def\bigdecl{{\scalebox{0.5}{ %%%%%%%%%%%%%%%%%%%%%%%%%%%%%%%%%\bigdecl
\begin{picture}(140,82)(250,-10)
\SetColor{Black}
\DashLine(10,-33)(610,-33){10}
%\DashLine(10,-3)(610,-3){10}
%\DashLine(10,23)(610,23){10}
%\DashLine(10,53)(610,53){10}

\SetColor{Black}
\SetWidth{0.5} \Vertex(10,-33){5.66}
\SetWidth{1.0} \Line(10,-33)(40,-3)
\SetWidth{0.5} \Vertex(40,-3){5.66}
\SetWidth{1.0} \Line(40,-3)(70,23)
\SetWidth{0.5} \Vertex(70,23){5.66}
\SetWidth{1.0} \Line(70,23)(100,23)
\put(80,33){\em\huge $a$}
\SetWidth{0.5} \Vertex(100,23){5.66}
\SetWidth{1.0} \Line(100,23)(130,53)
\SetWidth{0.5} \Vertex(130,53){5.66}
\SetWidth{1.0} \Line(130,53)(160,53)
\put(140,63){\em\huge $b$}
\SetWidth{0.5} \Vertex(160,53){5.66}
\SetWidth{1.0} \Line(160,53)(190,53)
\put(170,63){\em\huge $c$}
\SetWidth{0.5} \Vertex(190,53){5.66}
\SetWidth{1.0} \Line(190,53)(220,23)
\SetWidth{0.5} \Vertex(220,23){5.66}
\SetWidth{1.0} \Line(220,23)(250,23)
\put(230,33){\em\huge $d$}
\SetWidth{0.5} \Vertex(250,23){5.66}
\SetWidth{1.0} \Line(250,23)(280,-3)
\SetWidth{0.5} \Vertex(280,-3){5.66}
\SetWidth{1.0} \Line(280,-3)(310,-3)
\put(290,03){\em\huge $e$}
\SetWidth{0.5} \Vertex(310,-3){5.66}
\SetWidth{1.0} \Line(310,-3)(340,23)
\SetWidth{0.5} \Vertex(340,23){5.66}
\SetWidth{1.0} \Line(340,23)(370,53)
\SetWidth{0.5} \Vertex(370,53){5.66}
\SetWidth{1.0} \Line(370,53)(400,53)
\put(380,58){\em\huge $f$}
\SetWidth{0.5} \Vertex(400,53){5.66}
\SetWidth{1.0} \Line(400,53)(430,23)
\SetWidth{0.5} \Vertex(430,23){5.66}
\SetWidth{1.0} \Line(430,23)(460,23)
\put(440,33){\em\huge $g$}
\SetWidth{0.5} \Vertex(460,23){5.66}
\SetWidth{1.0} \Line(460,23)(490,53)
\SetWidth{0.5} \Vertex(490,53){5.66}
\SetWidth{1.0} \Line(490,53)(520,53)
\put(500,63){\em\huge $h$}
\SetWidth{0.5} \Vertex(520,53){5.66}
\SetWidth{1.0} \Line(520,53)(550,23)
\SetWidth{0.5} \Vertex(550,23){5.66}
\SetWidth{1.0} \Line(550,23)(580,-3)
\SetWidth{0.5} \Vertex(580,-3){5.66}
\SetWidth{1.0} \Line(580,-3)(610,-33)
\SetWidth{0.5} \Vertex(610,-33){5.66}

\end{picture}}}}

\def\bigdect{{\scalebox{0.4}{ %%%%%%%%%%%%%%%%%%%%%%%%%%%%%%%%%\bigdect
\begin{picture}(140,120)(0,-60)
\SetColor{Black}
%from alpha
\SetWidth{0.5} \Vertex(70,60){5.66}
\put(48,60){\em\huge$\alpha$}
\SetWidth{1.0} \Line(70,60)(0,20)
\SetWidth{0.5} \Vertex(0,20){5.66}
\put(-15,25){\em\huge$\beta$}
\SetWidth{1.0} \Line(70,60)(70,20)
\SetWidth{0.5} \Vertex(70,20){5.66}
\put(50,20){\em\huge$e$}
\SetWidth{1.0} \Line(70,60)(140,20)
\SetWidth{0.5} \Vertex(140,20){5.66}
\put(150,25){\em\huge$\delta$}

%from beta
\SetWidth{1.0} \Line(0,20)(-50,-20)
\SetWidth{0.5} \Vertex(-50,-20){5.66}
\put(-70,-20){\em\huge $a$}
\SetWidth{1.0} \Line(0,20)(0,-20)
\SetWidth{0.5} \Vertex(0,-20){5.66}
\put(-20,-20){\em\huge$\gamma$}
\SetWidth{1.0} \Line(0,20)(50,-20)
\SetWidth{0.5} \Vertex(50,-20){5.66}
\put(50,-38){\em\huge$d$}

% from gamma
\SetWidth{1.0} \Line(0,-20)(-30,-50)
\SetWidth{0.5} \Vertex(-30,-50){5.66}
\put(-45,-68){\em\huge$b$}
\SetWidth{1.0} \Line(0,-20)(30,-50)
\SetWidth{0.5} \Vertex(30,-50){5.66}
\put(25,-68){\em\huge$c$}

% from delta
\SetWidth{1.0} \Line(140,20)(100,-10)
\SetWidth{0.5} \Vertex(100,-10){5.66}
\put(80,-20){\em\huge$f$}
\SetWidth{1.0} \Line(140,20)(140,-20)
\SetWidth{0.5} \Vertex(140,-20){5.66}
\put(150,-30){\em\huge$\sigma$}
\SetWidth{1.0} \Line(140,-20)(140,-60)
\SetWidth{0.5} \Vertex(140,-60){5.66}
\put(150,-70){\em\huge$g$}
\SetWidth{1.0} \Line(140,20)(180,-10)
\SetWidth{0.5} \Vertex(180,-10){5.66}
\put(190,-30){\em\huge$\tau$}
\SetWidth{1.0} \Line(180,-10)(180,-60)
\SetWidth{0.5} \Vertex(180,-60){5.66}
\put(190,-70){\em\huge$h$}

\end{picture}}}}

\def\bigdectl{{\scalebox{0.4}{ %%%%%%%%%%%%%%%%%%%%%%%%%%%%%%%%%\bigdectl
%%%%%%%%%%% leaf decorated big tree
\begin{picture}(140,120)(0,-60)
\SetColor{Black}
%from alpha
\SetWidth{0.5} \Vertex(70,60){5.66}
%\put(48,60){\em\huge$\alpha$}
\SetWidth{1.0} \Line(70,60)(0,20)
\SetWidth{0.5} \Vertex(0,20){5.66}
%\put(-15,25){\em\huge$\beta$}
\SetWidth{1.0} \Line(70,60)(70,20)
\SetWidth{0.5} \Vertex(70,20){5.66}
\put(50,20){\em\huge$e$}
\SetWidth{1.0} \Line(70,60)(140,20)
\SetWidth{0.5} \Vertex(140,20){5.66}
%\put(150,25){\em\huge$\delta$}

%from beta
\SetWidth{1.0} \Line(0,20)(-50,-20)
\SetWidth{0.5} \Vertex(-50,-20){5.66}
\put(-70,-20){\em\huge $a$}
\SetWidth{1.0} \Line(0,20)(0,-20)
\SetWidth{0.5} \Vertex(0,-20){5.66}
%\put(-20,-20){\em\huge$\gamma$}
\SetWidth{1.0} \Line(0,20)(50,-20)
\SetWidth{0.5} \Vertex(50,-20){5.66}
\put(50,-38){\em\huge$d$}

% from gamma
\SetWidth{1.0} \Line(0,-20)(-30,-50)
\SetWidth{0.5} \Vertex(-30,-50){5.66}
\put(-45,-68){\em\huge$b$}
\SetWidth{1.0} \Line(0,-20)(30,-50)
\SetWidth{0.5} \Vertex(30,-50){5.66}
\put(25,-68){\em\huge$c$}

% from delta
\SetWidth{1.0} \Line(140,20)(100,-10)
\SetWidth{0.5} \Vertex(100,-10){5.66}
\put(80,-20){\em\huge$f$}
\SetWidth{1.0} \Line(140,20)(140,-20)
\SetWidth{0.5} \Vertex(140,-20){5.66}
%\put(150,-30){\em\huge$\sigma$}
\SetWidth{1.0} \Line(140,-20)(140,-60)
\SetWidth{0.5} \Vertex(140,-60){5.66}
\put(150,-70){\em\huge$g$}
\SetWidth{1.0} \Line(140,20)(180,-10)
\SetWidth{0.5} \Vertex(180,-10){5.66}
%\put(190,-30){\em\huge$\tau$}
\SetWidth{1.0} \Line(180,-10)(180,-60)
\SetWidth{0.5} \Vertex(180,-60){5.66}
\put(190,-70){\em\huge$h$}
\end{picture}}}}

\def\bigdectls{{\scalebox{0.4}{ %%%%%%%%%%%%%%%%%%%%%%%%%%%%%%%%%\bigdectls
%%%%%% leaf decorated leaf spaced
\begin{picture}(180,120)(0,-60)
\SetColor{Black}
%from alpha
\SetWidth{0.5} \Vertex(70,60){5.66}
%\put(48,60){\em\huge$\alpha$}
\SetWidth{1.0} \Line(70,60)(0,20)
\SetWidth{0.5} \Vertex(0,20){5.66}
%\put(-15,25){\em\huge$\beta$}
\SetWidth{1.0} \Line(70,60)(70,20)
\SetWidth{0.5} \Vertex(70,20){5.66}
\put(50,20){\em\huge$e$}
\SetWidth{1.0} \Line(70,60)(140,20)
\SetWidth{0.5} \Vertex(140,20){5.66}
%\put(150,25){\em\huge$\delta$}

%from beta
\SetWidth{1.0} \Line(0,20)(-50,-20)
\SetWidth{0.5} \Vertex(-50,-20){5.66}
\put(-70,-20){\em\huge $a$}
\SetWidth{1.0} \Line(0,20)(0,-20)
\SetWidth{0.5} \Vertex(0,-20){5.66}
%\put(-20,-20){\em\huge$\gamma$}
\SetWidth{1.0} \Line(0,20)(50,-20)
\SetWidth{0.5} \Vertex(50,-20){5.66}
\put(50,-38){\em\huge$d$}

% from gamma
\SetWidth{1.0} \Line(0,-20)(-30,-50)
\SetWidth{0.5} \Vertex(-30,-50){5.66}
\put(-45,-68){\em\huge$b$}
\SetWidth{1.0} \Line(0,-20)(30,-50)
\SetWidth{0.5} \Vertex(30,-50){5.66}
\put(25,-68){\em\huge$c$}

% from delta
\SetWidth{1.0} \Line(140,20)(100,-10)
\SetWidth{0.5} \Vertex(100,-10){5.66}
\SetWidth{1.0} \Line(100,-10)(100,-60)
\SetWidth{0.5} \Vertex(100,-60){5.66}
\put(150,-40){\em\huge$g$}
\SetWidth{1.0} \Line(140,20)(140,-20)
\SetWidth{0.5} \Vertex(140,-20){5.66}
%\put(150,-30){\em\huge$\sigma$}
\put(80,-70){\em\huge$f$}
\SetWidth{1.0} \Line(140,20)(180,-10)
\SetWidth{0.5} \Vertex(180,-10){5.66}
%\put(190,-30){\em\huge$\tau$}
\SetWidth{1.0} \Line(180,-10)(180,-60)
\SetWidth{0.5} \Vertex(180,-60){5.66}
\put(190,-70){\em\huge$h$}
\end{picture}}}}

\def\bigdecta{{\scalebox{0.4}{ %%%%%%%%%%%%%%%%%%%%%%%%%%%%%%%%%\bigdectls
%%%%%% angularly decorated
\begin{picture}(180,120)(-40,-60)
\SetColor{Black}
%from alpha
\SetWidth{0.5} \Vertex(70,60){5.66}
%\put(48,60){\em\huge$\alpha$}
\SetWidth{1.0} \Line(70,60)(0,20)
\SetWidth{0.5} \Vertex(0,20){5.66}
%\put(-15,25){\em\huge$\beta$}
%\SetWidth{1.0} \Line(70,60)(70,20)
%\SetWidth{0.5} \Vertex(70,20){5.66}
\put(60,30){\em\huge$e$}
\SetWidth{1.0} \Line(70,60)(140,20)
\SetWidth{0.5} \Vertex(140,20){5.66}
%\put(150,25){\em\huge$\delta$}

%from beta
\SetWidth{1.0} \Line(0,20)(-50,-20)
\SetWidth{0.5} \Vertex(-50,-20){5.66}
\put(-20,-10){\em\huge $a$}
\SetWidth{1.0} \Line(0,20)(0,-20)
\SetWidth{0.5} \Vertex(0,-20){5.66}
%\put(-20,-20){\em\huge$\gamma$}
\SetWidth{1.0} \Line(0,20)(50,-20)
\SetWidth{0.5} \Vertex(50,-20){5.66}
\put(10,-20){\em\huge$d$}

% from gamma
\SetWidth{1.0} \Line(0,-20)(-30,-50)
\Line(0,-20)(0,-50)
\SetWidth{0.5} \Vertex(-30,-50){5.66}
\Vertex(0,-50){5.66}
\put(-22,-68){\em\huge$b$}
\SetWidth{1.0} \Line(0,-20)(30,-50)
\SetWidth{0.5} \Vertex(30,-50){5.66}
\put(8,-68){\em\huge$c$}

% from delta
\SetWidth{1.0} \Line(140,20)(100,-20)
\SetWidth{0.5} \Vertex(100,-20){5.66}
\SetWidth{1.0} \Line(100,-20)(70,-50)
\Line(100,-20)(130,-50)
\SetWidth{0.5} \Vertex(70,-50){5.66}
\Vertex(130,-50){5.66}
\put(90,-60){\em\huge$f$}
%\SetWidth{1.0} \Line(140,20)(140,-20)
%\SetWidth{0.5} \Vertex(140,-20){5.66}
%\put(150,-30){\em\huge$\sigma$}
\put(140,-10){\em\huge$g$}
\SetWidth{1.0} \Line(140,20)(180,-10)
\SetWidth{0.5} \Vertex(180,-10){5.66}
%\put(190,-30){\em\huge$\tau$}
\SetWidth{1.0} \Line(180,-10)(150,-50)
\Line(180,-10)(210,-50)
\SetWidth{0.5} \Vertex(150,-50){5.66}
\Vertex(210,-50){5.66}
\put(175,-60){\em\huge$h$}
\end{picture}}}}

\def\motzstand{{\scalebox{0.35}{ %%%%%%%%%%%%%%%%%%%%%%%%%%%%%%%%%\motzstand
\begin{picture}(300,62)(-20,-68)
\SetColor{Black}
\SetWidth{0.5} \Vertex(-50,-63){5.66}
\SetWidth{1.0} \Line(-50,-63)(-20,-63)
\Vertex(-20,-63){5.66}
\put(-40,-55){\huge $x_1$}
\SetWidth{1.0} \Line(-20,-63)(10,-33)
\SetWidth{0.5} \Vertex(10,-33){5.66}
\SetWidth{1.0} \Line(10,-33)(40,-33)
\put(15,-22){\huge $x_2$}
\SetWidth{0.5} \Vertex(40,-33){5.66}
\SetWidth{1.0} \Line(40,-33)(70,-63)
\SetWidth{0.5} \Vertex(70,-63){5.66}
\SetWidth{1.0} \Line(70,-63)(100,-63)
\put(75,-52){\huge $x_3$}
\SetWidth{0.5} \Vertex(100,-63){5.66}
\SetWidth{1.0} \Line(100,-63)(130,-33)
\SetWidth{0.5} \Vertex(130,-33){5.66}
\SetWidth{1.0} \Line(130,-33)(160,-33)
\put(135,-22){\huge $x_4$}
\SetWidth{0.5} \Vertex(160,-33){5.66}
\SetWidth{1.0} \Line(160,-33)(190,-63)
\SetWidth{0.5} \Vertex(190,-63){5.66}
\SetWidth{1.0} \Line(190,-63)(220,-63)
\put(195,-52){\huge $x_5$}
\SetWidth{0.5} \Vertex(220,-63){5.66}
\SetWidth{1.0} \Line(220,-63)(250,-63)
\put(225,-52){\huge $x_6$}
\SetWidth{0.5} \Vertex(250,-63){5.66}
\end{picture}}}}

\def\rhsmotzstand{{\scalebox{0.35}{ %%%%%%%%%%%%%%%%%%%%%%%%%%%%%%%%%\motzstand
\begin{picture}(700,62)(-120,-68)
\SetColor{Black}
\SetWidth{0.5} \Vertex(-100,-63){5.66}
\put(-75,-60){\huge $\circ$}
\SetWidth{0.5} \Vertex(-50,-63){5.66}
\SetWidth{1.0} \Line(-50,-63)(-20,-63)
\put(-40,-55){\huge $x_1$}
\Vertex(-20,-63){5.66}
\put(5,-60){\huge $\circ$}
\SetWidth{0.5} \Vertex(30,-63){5.66}

\SetWidth{1.0} \Line(30,-63)(60,-33)
\SetWidth{0.5} \Vertex(60,-33){5.66}
\SetWidth{1.0} \Line(60,-33)(90,-33)
\put(65,-22){\huge $x_2$}
\SetWidth{0.5} \Vertex(90,-33){5.66}
\SetWidth{1.0} \Line(90,-33)(120,-63)
\SetWidth{0.5} \Vertex(120,-63){5.66}
\put(135,-60){\huge $\circ$}

\SetWidth{0.5} \Vertex(170,-63){5.66}
\SetWidth{1.0} \Line(170,-63)(200,-63)
\put(175,-52){\huge $x_3$}
\SetWidth{0.5} \Vertex(200,-63){5.66}
\put(225,-60){\huge $\circ$}

\SetWidth{0.5} \Vertex(250,-63){5.66}
\SetWidth{1.0} \Line(250,-63)(280,-33)
\SetWidth{0.5} \Vertex(280,-33){5.66}
\SetWidth{1.0} \Line(280,-33)(310,-33)
\put(290,-22){\huge $x_4$}
\SetWidth{0.5} \Vertex(310,-33){5.66}
\SetWidth{1.0} \Line(310,-33)(340,-63)
\SetWidth{0.5} \Vertex(340,-63){5.66}
\put(365,-60){\huge $\circ$}

\SetWidth{0.5} \Vertex(390,-63){5.66}
\SetWidth{1.0} \Line(390,-63)(420,-63)
\put(395,-52){\huge $x_5$}
\SetWidth{0.5} \Vertex(420,-63){5.66}
\put(445,-60){\huge $\circ$}

\SetWidth{0.5} \Vertex(470,-63){5.66}
\put(495,-60){\huge $\circ$}

\SetWidth{0.5} \Vertex(520,-63){5.66}
\SetWidth{1.0} \Line(520,-63)(550,-63)
\put(525,-52){\huge $x_6$}
\SetWidth{0.5} \Vertex(550,-63){5.66}
\put(575,-60){\huge $\circ$}

\SetWidth{0.5} \Vertex(600,-63){5.66}
\end{picture}}}}

%%%%%%%%%%%%%%%%%%%%%%%%%%%%%%%%%%%%%%%%%%%%%%%%%%%

%%%%%%%%%%%%%%%%%%%%% leave trees
%%%%%%%%%%%%%%%%%%%%%%%%%%%%%%%%%%%%%%%%%%%%%%%%%%%%%%%%%%%

\def\lxtree{\includegraphics[scale=.7]{lxtree}}
\def\xdtree{\includegraphics[scale=.7]{xdtree}}
\def\xdtreend{\includegraphics[scale=.2]{xdtreend}} % xdtree no decoration
\def\xytree{\includegraphics[scale=.7]{xytree}}
\def\xytreend{\includegraphics[scale=.2]{xytreend}} % xytree no decoration

\def\xdtreed{\includegraphics[scale=.3]{xdtreed}}
\def\xytreed{\includegraphics[scale=.3]{xytreed}}
\def\lxdtree{\includegraphics[scale=.7]{lxdtree}}
\def\dtree{\includegraphics[scale=0.5]{dtree}}
\def\xyleft{\includegraphics[scale=.7]{xyleft}}
\def\xyright{\includegraphics[scale=.7]{xyright}}
\def\yzdtree{\includegraphics[scale=.7]{yzdtree}}
\def\xyztreea{\includegraphics[scale=.7]{xyztree1}}
\def\xyztreeb{\includegraphics[scale=.7]{xyztree2}}
\def\xyztreec{\includegraphics[scale=.7]{xyztree3}}
\def\tableps{\includegraphics[scale=1.4]{tableps}}

\def\ccxtree{\includegraphics[scale=1]{ccxtree}}
\def\ccxtreedec{\includegraphics[scale=1]{ccxtreedec}}
\def\ccetree{\includegraphics[scale=1]{ccetree}}
\def\ccllcrcr{\includegraphics[scale=1]{ccllcrcr}}
\def\cclclcrcr{\includegraphics[scale=.5]{cclclcrcr}}
\def\ccllxryr{\includegraphics[scale=.5]{ccllxryr}}
\def\cclxlyrzr{\includegraphics[scale=1]{cclxlyrzr}}
\def\ccIII{\includegraphics[scale=1]{ccIII}}

%%%%%%%%%%%%%%%%%%%%%%%%%%%%%%%%%%%%%%%%%%%%%%%%%%%%%%%%%%%%%%%%%%

\title{Operated semigroups, Motzkin paths and rooted trees}
\author{Li Guo}
\address{Department of Mathematics and Computer Science,
         Rutgers University,
         Newark, NJ 07102}
\email{liguo@newark.rutgers.edu}

%\date{\today}

%\begin{document}

\begin{abstract}
Combinatorial objects such as rooted trees that carry a recursive structure have found important applications recently in both mathematics and physics. We put such structures in an algebraic framework of operated semigroups. This framework provides the concept of operated semigroups with intuitive and convenient combinatorial descriptions, and at the same time endows the familiar combinatorial objects with a precise algebraic interpretation. As an application, we obtain constructions of free Rota-Baxter algebras in terms of Motzkin paths and rooted trees.
\end{abstract}

% MSC: 16W99, 05C05, 08A50

\maketitle

\tableofcontents

\noindent
{\bf Key words: } \mapped semigroups, \mapped algebras, free algebras, bracketed words, planar rooted trees, Motzkin paths, Dyck paths, Rota-Baxter algebras.

\setcounter{section}{0}
{\ }
%\vspace{-1cm}

\section{Introduction}
\label{intro}
\subsection{Motivation}
This paper explores the relationship between two subjects that have been studied
separately until recently. One subject considers algebraic structures, such
as semigroups and associative algebras, with an operator
acting on them. Such structures include differential algebras,
difference algebras and Rota--Baxter algebras (See Example~\mref{ex:mapalg} for the definitions). Another subject studies objects, often combinatorial in nature, that have
underlying recursive structures, such as rooted trees and Motzkin paths.
Since the 1990s, algebraic structures on rooted trees have been studied in the work of Connes and Kreimer~\mcite{C-K0,C-K1} on the renormalization of quantum field theory, Grossman and Larson~\mcite{G-L} on data structures and of Loday and Ronco~\mcite{Lo1,L-R1} on operads. The grafting operator on trees plays an important role in their works.
More recently~\mcite{A-M,E-G0}, free Rota--Baxter algebras were constructed using planar rooted trees with special decorations.

In this paper, we relate these two subjects through the concepts of \mapped semigroups, \mapped monoids and \mapped algebras. We establish
that free \mapped semigroups have natural combinatorial interpretation in terms of Motzkin paths and planar rooted forests. This freeness characterization of rooted forests and Motzkin paths gives an algebraic explanation of the fundamental roles played by these combinatorial objects and their related numerical sequences such as the Catalan numbers and Motzkin numbers~\mcite{St}.
This characterization should be useful in further algebraic studies of these combinatorial objects. This connection also
endows the concept of \mapped algebras and semigroups with familiar
combinatorial contents, giving significance to these \mapped algebraic structures beyond the abstract generalization. As a consequence, we obtain several constructions of free Rota-Baxter algebras which can be adopted to free objects in the other related algebraic structures.
\medskip

\subsection{Definitions and examples}
\begin{defn}{\rm
An {\bf \mapped semigroup} (or a {\bf semigroup with an operator}) is a semigroup $U$ together with an operator
$\alpha: U\to U$. $\alpha$ is called the {\bf distinguished operator} on
$U$.
A morphism from an \mapped
semigroup $(U,\alpha)$ to an \mapped semigroup $(V,\beta)$ is a
semigroup homomorphism $f :U\to V$ such that $f \circ \alpha=
\beta \circ f,$ that is, such that the following diagram
commutes.
$$
 \xymatrix{
            U\ar[rr]^\alpha \ar[d]^f && U \ar[d]_f \\
            V\ar[rr]^\beta           && V}
$$

More generally, let $\Omega$ be a set. An {\bf $\Omega$-\mapped semigroup} is a semigroup $U$ together with a set of operators
$\alpha_\omega: U\to U,\ \omega\in \Omega$.
In other words, an $\Omega$-\mapped semigroup is a pair $(U,\alpha)$ with a semigroup $U$ and a map $\alpha: \Omega \to \Map(U,U),
\alpha(\omega) = \alpha_\omega$. Here $\Map(U,U)$ is the set of maps from $U$ to $U$.
A morphism from an $\Omega$-\mapped
semigroup $(U,\{\alpha_\omega, \omega\in \Omega\})$ to an $\Omega$-\mapped semigroup $(V,\{\beta_\omega, \omega\in \Omega\})$ is a
semigroup homomorphism $f :U\to V$ such that $f \circ \alpha_\omega=
\beta_\omega \circ f$ for $\omega \in \Omega.$
}
\mlabel{de:mapset}
\end{defn}
\begin{remark}{\rm
When a semigroup is replaced by a monoid we obtain the concept of an ($\Omega$-) {\bf \mapped monoid}. Let $\bfk$ be a commutative ring. We similarly define the concepts of an ($\Omega$-){\bf \mapped $\bfk$-algebra} or
($\Omega$-){\bf \mapped nonunitary $\bfk$-algebra}.
}
\end{remark}

\begin{exam}{\rm
Here are some examples of \mapped $\bfk$-algebras with one operator.
\begin{enumerate}
\item
A semigroup is an \mapped semigroup when the
distinguished operator is taken to be the identity;
\item A differential algebra~\mcite{Kol,S-P} is an associative algebra $A$ with
a linear operator $d:A\to A$ such that
$$ d(xy)=d(x)y+xd(y), \forall\, x,y\, \in A;$$
\item
A difference algebras~\mcite{RCo,S-P1} is an associative algebra $A$ with an algebra endomorphism on $A$;
 \item
Let $\lambda$ be fixed in the ground ring $\bfk$. A Rota-Baxter algebras~\mcite{Ba,Ca,C-K1,E-G4,G-K1,Ro1,Ro2} (of weight $\lambda$) is defined to be an associative algebra $A$ with a linear operator $P$ such that
\begin{equation}
P(x)P(y) = P(xP(y))+ P(P(x)y)+\lambda P(xy),\ \forall\,
x,y\, \in\, A;
\mlabel{eq:rba}
\end{equation}
\item A differential algebras of weight $\lambda$~\mcite{G-K3} is defined to be an algebra $A$ together with a linear operator $d: A\to A$ such that
$$ d(xy)=d(x)y+xd(y)+\lambda\, d(x)d(y), \forall\, x,y\, \in A.$$
\end{enumerate}
\mlabel{ex:mapalg}
Other examples of \mapped algebras can be found in~\mcite{Ro2}.

Here are some $\Omega$-\mapped algebras with multiple operators.
\begin{enumerate}
\item A {\bf $\Delta$-differential algebras}~\mcite{Kol} is an algebra with multiple differential operators $\delta\in \Delta$ that commute with each other;
\item A {\bf $\sigma\delta$-algebra}~\mcite{Si} is an algebra with a commuting pair of a difference operator $\sigma$ and a differential operator $\delta$;
\item A {\bf differential Rota-Baxter algebra} of weight $\lambda$~\mcite{G-K3} is an algebra with a differential operator $d$ of weight $\lambda$ and a Rota-Baxter algebra $P$ of weight $\lambda$, such that $d\circ P=\id.$ This last relation is a natural generalization of the First Fundamental Theorem of Calculus when $d$ is taken to be the usual derivation and $P$ is the integral operator $P[f](x)=\int_a^x f(t)\, dt$;
\item As an variation, we can consider an algebra with a differential operator and a Rota-Baxter operator of weight $-1$~\mcite{G-Z} that commute with each other. It arises naturally in the study of multiple zeta values by renormalization methods;
\item
A {\bf Rota-Baxter family} on an algebra $R$ is a collection of linear operators $P_\omega$ on $R$ with $\omega$ in a semigroup $\Omega$, such that
$$ P_\alpha(x)P_\beta(y)=P_{\alpha\beta}(P_\alpha(x)y) +P_{\alpha\beta}(xP_{\beta}(y))+\lambda P_{\alpha\beta}(xy), \ \forall\, x, y\in A, \alpha, \beta\in \Omega.$$
It arises naturally  in renormalization of quantum field theory~\cite[Prop. 9.1]{EGP}.
\end{enumerate}
}
\end{exam}

%\subsection{Outline of the paper}

Our main goal here is to give combinatorial constructions of free objects in the category of $\Omega$-\mapped semigroups and $\Omega$-\mapped monoids.
They naturally give free objects in the category of \mapped algebras.
These free objects are obtained as the adjoint functor of the forgetful functor from the category
of \mapped semigroups and \mapped monoids to the category of sets in the usual way. More precisely,

\begin{defn}{\rm
A {\bf free \mapped semigroup} on a set $X$ is an \mapped semigroup
$(U_X,\alpha_X)$ together with a map $j_X:X\to U_X$ with the
property that, for any \mapped semigroup $(V,\beta)$ and
any map $f:X\to V$, there is a unique morphism
$\free{f}:(U_X,\alpha_X)\to (V,\beta)$ of \mapped semigroups such
that $ f=\free{f}\circ j_X.$ In other words, the following diagram commutes.
$$
 \xymatrix{
            X \ar[rr]^{j_X} \ar[drr]_{f} && U_X \ar[d]^{\free{f}} \\
                                         && V
            }
$$
Let $\cdot$ be the binary operation on the semigroup $U_X$, we also use the quadruple $(U_X, \cdot, \alpha_X, j_X)$ to denote the free \mapped semigroup on $X$, except when $X=\emptyset$, when we drop $j_X$ and use the triple
$(U_X,\cdot, \alpha_X)$.
}
\mlabel{de:freemapped}
\end{defn}
We similarly define the concepts of free \mapped monoids, and free \mapped unitary and nonunitary $\bfk$-algebras. We also similarly define the more general concept of free $\Omega$-\mapped monoids. See Theorem~\mref{thm:motzdec} for the precise definition.

\subsection{Outline of the paper}
In Section~\mref{sec:paths}, free \mapped semigroups and free \mapped monoids are constructed in terms of Motzkin paths (Corollary~\mref{co:Motz}). In fact, we construct free $\Omega$-\mapped semigroups and free $\Omega$-\mapped monoids in terms of a natural generalization of Motzkin paths (Theorem~\mref{thm:motzdec}). Through the Motzkin words, we relate the Motzkin paths with the recursively constructed bracketed words in Section~\mref{ss:words} (Theorem~\mref{thm:motzword}).
This in turn allows us to construct free \mapped semigroups and free \mapped monoids in terms of bracketed words (Corollary~\mref{co:wordfree}).
In Section~\mref{sec:trees}, free $\Omega$-\mapped semigroups are constructed in terms of vertex decorated planar rooted forests, through a natural isomorphism from the free \mapped semigroup of peak-free Motzkin paths to the \mapped semigroup of the vertex decorated planar rooted forests (Theorem~\mref{thm:treefree}).
\smallskip

One can regard these results on the free objects as a first step in the study of \mapped semigroups and \mapped algebras that generalizes the extensive work on semigroups~\mcite{E-N,Gr,HLV,SGIF}.  But we also have in mind the more practical purpose of using these free objects to underly the algebraic structures on the combinatorial objects of planar rooted trees and Motzkin paths, and to study Rota-Baxter algebras and related structures.

Given the recent progresses on Rota-Baxter algebra in both theoretical and application aspects~\mcite{Ag3,A-M,C-K0,E-G4,E-G0,E-G-K3,E-G-M,Gu5,G-K1,
G-K3,G-S,G-Z}, it is desirable to obtain convenient constructions of free Rota-Baxter algebras. To this end,
in Section~\mref{sec:bij}, we put together bijections and inclusions among bracketed words, Motzkin path, vertex decorated forests and angular decorated forests, as well as their various subsets (Theorem~\mref{thm:diag}). These maps preserves the structure of \mapped semigroups. In Section~\mref{sec:rba}, we use these bijections and the construction of free Rota-Baxter algebra in terms of angularly decorated rooted forests~\mcite{E-G0} to construct free Rota-Baxter algebras in terms of Motzkin paths, bracketed words and leaf decorated forests (Corollary~\mref{co:motzfree} -- \mref{co:ltreefree}).
\medskip

\noindent
{\bf Notations: } We will use $\NN$ to denote the set of non-negative integers. By an algebra we mean an associative unitary algebra unless otherwise specified. For a commutative ring $\bfk$ and a set $Y$, we use $\bfk\, Y$ to denote the free $\bfk$-module with basis $Y$. When $Y$ is a monoid (resp. semigroup), $\bfk\, Y$ carries the natural $\bfk$-algebra (resp. nonunitary $\bfk$-algebra) structure. We use $\sdotcup$ to denote a disjoint union.
\medskip

\noindent
{\bf Acknowledgements: }
This work is supported in part by NSF grant DMS-0505643.

\section{Free \mapped semigroups and monoids in terms of Motzkin paths}
\mlabel{sec:paths}
We show that free \mapped semigroups and monoids have a natural construction by Motzkin paths.

Recall~\mcite{D-S,Do-S} that a {\bf Motzkin path} is  a lattice path in $\NN^2$ from $(0,0)$
to $(n,0)$ whose permitted steps are
an up diagonal step (or {\bf up step} for short) $(1,1)$, a down diagonal step (or {\bf down step}) $(1,-1)$ and a horizontal step (or {\bf level step}) $(1,0)$.
The first few Motzkin paths are
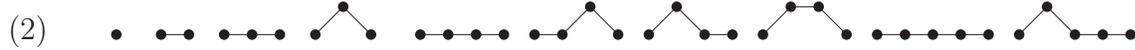
\begin{equation}
 \ma1 \;
 \mb2   \;
 \mc3   \;
 \md31  \;
 \me4   \;
 \mf41  \;
 \mg42 \;
 \mh43 \;
 \mi5 \;
 \mj51 \;
 \mlabel{eq:mpath}
\end{equation}

The height of a Motzkin path is simply the maximum of the height of the points on the path. Let $\frakP$ be the set of Motzkin paths. For Motzkin paths $\frakm$ and $\frakm'$, define $\frakm\circ \frakm$, called the {\bf link product} of $\frakm$ and $\frakm'$, to be the Motzkin path obtained by joining the last vertex of $\frakm$ with the first vertex of $\frakm'$. For example,
$$
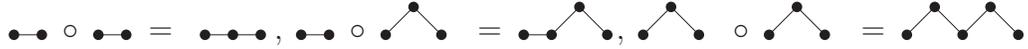
\mb2 \circ \mb2 = \ \mc3 ,
\mb2 \circ \md31 =  \mf41,
\md31 \circ \md31 =  \mpp57
$$
The link product is obviously associative with the {\bf trivial Motzkin path} $\onetree$ as the identity.
Let $\frakI$ be the set of {\bf indecomposable} (also called {\bf prime}) Motzkin paths,
consisting of Motzkin paths that touch the $x$-axis only at the two end vertices.
It is clear that a Motzkin path is indecomposable if and only if it is not the link product of two non-trivial Motzkin paths.

Next for a Motzkin path $\frakm$, denote $\lm \frakm \rtm$
to be the Motzkin path obtained by raising $\frakm$ on the left end by an up step
and on the right end by a down step. For example,
$$ \lm \ma1\, \rtm = \md31,\ \lm \ \ \mb2 \rtm =  \mh43\ ,
\lm \md31 \rtm =  \mq58
$$
This defines an operator $\lm\ \rtm$ on $\frakP$, called the {\bf raising operator}. Thus $\frakP$, with the link product $\circ$ and the raising operator $\lm\ \rtm$, is an \mapped monoid.

Let $\frakD$ denote set of {\bf Dyck paths} which are defined to be paths that do not have any level steps.
Then $(\frakD,\circ, \lm\ \rtm)$ is an \mapped submonoid
of the \mapped monoid $(\frakP, \circ, \lm\ \rtm)$.

A Motzkin path is called {\bf peak-free} if it is not $\ma1$ and does not have an up step followed immediately by a down step. For example,
$ \mb2\; \mc3\;  \me4\; \mh43$ and $\; \mi5$
in the list (\mref{eq:mpath}) are peak-free while the rest
are not. Let $\frakL$ denote the set of peak-free Motzkin paths. Then $(\frakL, \circ, \lm\ \rtm)$ is an \mapped subsemigroup (but not submonoid) of $(\frakP,\circ, \lm\ \rtm)$.

Let $X$ be a set. An {\bf $X$-decorated (or colored) Motzkin path}~\mcite{CSY,D-S} is a Motzkin path whose level steps are decorated
(colored) by elements in $X$. Some examples are
$$
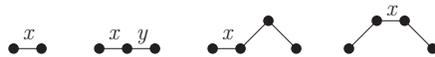
\xmb2 \quad \xymc3 \quad \xmf41 \quad \xmh43
$$
Let $\frakP(X)$ be the set of $X$-decorated Motzkin paths and let $\frakL(X)$ be the set of peak-free $X$-decorated Motzkin paths.
Note that Motzkin paths with no decorations can be identified with $X$-decorated Motzkin paths where $X$ is a singleton.

We next generalize the concept of Motzkin paths to allow decorations on the up and down steps.
A {\bf matching pair} of steps in a Motzkin path consists of an up step and the first down step to the right of this up step with the same height. To put it another way, a matching pair of steps is an up step and a down step to its right such that the path between (and excluding) these two steps is a Motzkin path.

Let $\Omega$ be a set. By an {\bf $(X,\Omega)$-decorated} or a {\bf fully decorated} Motzkin path we mean a Motzkin path where each matching pair of steps is decorated by an element of $\Omega$, and where each level step is decorated by an element of $X$. For example,
$$
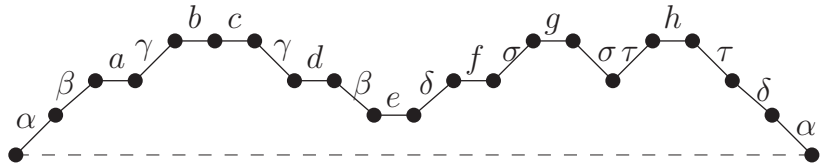
\bigdecm
\vspace{0.2cm}
$$
is an $(X,\Omega)$-decorated Motzkin path
with $a,b,c,d,e,f,g,h\in X$ and $\alpha,\beta,\gamma,\delta,\sigma,\tau\in \Omega$.

The set of $(X,\Omega)$-decorated Motzkin paths is denoted by $\frakP(X,\Omega)$. We similarly define {\bf $(X,\Omega)$-decorated peak-free Motzkin paths} $\frakL(X,\Omega)$ and {\bf $\Omega$-decorated Dyck paths} $\frakD(\Omega)$. The last notation makes sense since a Dyck path does not have any level steps and thus does not involve decorations by $X$.

The link product of two $(X,\Omega)$-decorated Motzkin paths is defined in the same way as for Motzkin paths. Further, for each $\omega\in \Omega$ and an $(X,\Omega)$-decorated Motzkin path $\frakm$, we define $ \lm_\omega\; \frakm\; \rtm_\omega$ to be the Motzkin path obtained by raising $\frakm$ on the left end by an up step on the right end by a down step, both decorated by $\omega$. For example, we have
$$
\lm_\omega \ma1 \rtm_\omega= \amd31,\
\lm_\omega \xmb2 \rtm_\omega= \axmh43, \quad
\lm_\omega \bmd31 \rtm_\omega= \abmq58
$$
Thus for each $\omega\in \Omega$, we obtain a map $\lm_\omega \ \rtm_\omega$ on each of the semigroups or monoids $\frakP(X,\Omega)$, $\frakL(X,\Omega)$ and $\frakD(\Omega)$, making it into an $\Omega$-\mapped semigroup or an $\Omega$-\mapped monoid.

The concepts of height and indecomposability in $\frakP(X,\Omega)$ are defined in the same way as in $\frakP$.
For $n\geq 0$, let $\frakP_n(X,\Omega)$ be the submonoid of elements of $\frakP(X,\Omega)$ of height $\leq n$. Also define $\frakP_{-1}(X,\Omega)=\emptyset$.
Define
$\frakL_n(X,\Omega)=\frakP_n(X,\Omega)\cap \frakL(X,\Omega)$ and
$\frakD_n(\Omega)=\frakP_n(X,\Omega)\cap \frakD(\Omega)$.

\begin{theorem} Let $\Omega$ and $X$ be non-empty sets.
Let $j_X: X\to \frakL(X,\Omega) \subseteq \frakP(X,\Omega)$ be defined by $j_X(x)=\xmb2,\, x\in X$.
\begin{enumerate}
\item
The quadruple $\big(\frakP(X,\Omega),\circ, \{\lm_\omega\ \rtm_\omega\;\big|\; \omega\in \Omega\}, j_X\big)$ is the free $\Omega$-\mapped monoid on $X$. More precisely,
for any $\Omega$-\mapped monoid $(H,\{\alpha_\omega\,\big| \omega\in \Omega\})$ consisting of a monoid $H$ and maps $\alpha_\omega: H \to H$ for $ \omega\in \Omega$, there is a unique morphism
$\free{f}:(\frakP(X,\Omega), \{\lm_\omega\ \rtm_\omega\, \big|\, \omega\in \Omega\})\to (H, \{\alpha_\omega\,\big|\, \omega\in \Omega\})$ of \mapped monoids such
that $ f=\free{f}\circ j_X.$
\mlabel{it:motmondec}
\item
The quadruple $(\frakL(X,\Omega),\circ, \{\lm_\omega\ \rtm_\omega\;\big| \; \omega\in \Omega\}, j_X)$ is the free $\Omega$-\mapped semigroup on $X$.
\mlabel{it:motsgdec}
\item
The triple $(\frakD(\Omega),\circ, \{\lm_\omega\ \rtm_\omega\;\big|\; \omega\in \Omega\})$ is the free $\Omega$-\mapped monoid on the empty set.
\mlabel{it:motmon0dec}
\end{enumerate}
\mlabel{thm:motzdec}
\end{theorem}

\begin{proof}
We only need to prove (\mref{it:motmondec}). The proof of the other parts are similar.

Let $(H,\{\alpha_\omega\,\big|\, \omega\in \Omega\})$ be an $\Omega$-\mapped monoid with a monoid $H$ and maps $\alpha_\omega: H\to H, \omega\in \Omega$. Let $f:X\to H$ be a set map. We will
use induction on $n$ to construct a unique sequence of monoid homomorphisms
\begin{equation}
\free{f}_n: \frakP_n(X,\Omega) \to H, n \geq 0,
\mlabel{eq:motzind}
\end{equation}
with the following properties.
\begin{enumerate}
\item
$\free{f}_n \big|_{\frakP_{n-1}(X,\Omega)} = \free{f}_{n-1}.$

\mlabel{it:rest}
\item
$\free{f}_n \circ (\lm_\omega\ \rtm_\omega)  = \alpha_{\omega} \circ \free{f}_{n-1}$ on $\frakP_{n-1}(X,\Omega)$ for each $\omega\in \Omega$.
\mlabel{it:ocomm}
\end{enumerate}
In other words, the following diagrams commute.
$$\xymatrix{\frakP_{n-1}(X,\Omega) \ar@{_{(}->}[d] \ar^{\free{f}_{n-1}}[rr]& & H \\
\frakP_n(X,\Omega) \ar_{\free{f}_n}[urr]
}\qquad
\xymatrix{\frakP_{n-1}(X,\Omega) \ar^{\free{f}_{n-1}}[rr]
    \ar_{\lm_\omega\ \rtm_\omega}[d] && H \ar[d]^{\alpha_\omega} \\
    \frakP_n(X,\Omega) \ar^{\free{f}_n}[rr] && H
}
$$

When $n=0$, $\frakP_0(X,\Omega)$ is the monoid of paths from $(0,0)$ to $(m,0)$, $m\geq 0$, consisting of only level steps which are decorated by elements of $X$. Thus $\frakP_0(X,\Omega)$ is the free monoid generated by $\{ \xmb2\big | x\in X\}$.
Then the map $f: X\to H$ extends uniquely to a monoid homomorphism $\free{f}_0: \frakP_0(X,\Omega) \to H$ such
that $\free{f}_0 \circ j_X = f$. $\free{f}_0$ trivially satisfies properties (\mref{it:rest}) and (\mref{it:ocomm}) since $\frakP_{-1}(X,\Omega)=\emptyset$ by convention.

For given $k\geq 0$, assume that there is a unique map $\free{f}_k: \frakP_k(X,\Omega) \to H$ satisfying the properties (\mref{it:rest}) and (\mref{it:ocomm}).
Note that $\frakP_{k+1}(X,\Omega)$ is the free monoid generated by $\frakI_{k+1}(X,\Omega)$, the set of indecomposable Motzkin paths of height $\leq k+1$, and note that an indecomposable Motzkin path of height $k+1$ is of the form $\lm_\omega\,\ofrakm \rtm_\omega$ for an $\omega\in \Omega$ and an $\ofrakm\in \frakP_k(X,\Omega)$ of height $k$. This is because an $\frakm\in \frakI_{k+1}(X,\Omega)$ touches the $x$-axis only at the beginning and the end of the path. So the first step must be a rise step and the last step must be a fall step, decorated by the same $\omega\in \Omega$. Further, if the first step and the last step are removed, we still have a Motzkin path $\ofrakm$ of height $k$ and $\frakm=\lm_\omega\,\ofrakm\rtm_\omega$.

Thus we have the disjoint union
$$
\frakI_{k+1}(X,\Omega)=\frakI_k(X,\Omega) \dotcup
\Big(\dotcup_{\omega\in \Omega}\ \lm_\omega\, \big(\frakL_k(X,\Omega)\backslash \frakL_{k-1}(X,\Omega)\big)\rtm_\omega\Big). $$
Define
$$ f_{k+1}: \frakI_{k+1}(X,\Omega) \to H$$
by requiring
$$
f_{k+1}(\frakm)= \left \{\begin{array}{ll}
    \free{f}_k(\frakm), & \frakm\in \frakI_k(X,\Omega),\\
    \alpha_\omega(\free{f}_k(\ofrakm)), &
        \frakm=\lm_\omega\ofrakm\rtm_\omega \in
    \lm_\omega (\frakL_k(X,\Omega)\backslash \frakL_{k-1}(X,\Omega))\rtm_\omega,\ \omega\in \Omega.
    \end{array} \right .
$$
Then extend $f_{k+1}$ to the free monoid $\frakP_{k+1}(X,\Omega)$ on $\frakI_{k+1}(X,\Omega)$ by multiplicity and obtain
$$ \free{f}_{k+1}: \frakP_{k+1}(X,\Omega) \to H.$$
By the construction of $f_{k+1}$, $\free{f}_{k+1}$ satisfies properties (\mref{it:rest}) and (\mref{it:ocomm}), and it is the unique such monoid homomorphism.

By Property (\mref{it:rest}), the sequence $\{\free{f}_n, n\geq 0\}$, forms a direct system of monoid homomorphisms and thus gives the direct limit
\begin{equation}
 \free{f}= \dirlim \free{f}_n: \frakP(X,\Omega) \to H
\mlabel{eq:mdir}
\end{equation}
which is naturally a monoid homomorphism.
By Property (\mref{it:ocomm}), we further have
$\free{f}\circ (\lm_\omega\ \rtm_\omega) = \alpha_{\omega}\circ \free{f},\
    \forall\, \omega\in \Omega.$
Thus $\free{f}$ is a homomorphism of $\Omega$-\mapped monoids such that $\free{f}\circ j_X= f$.

Furthermore, if $\free{f}': \frakP(X,\Omega)\to H$ is another homomorphism of $\Omega$-\mapped monoids such that $\free{f}' \circ j_X=f$.
Let $\free{f}'_n= \free{f}'\big|_{\frakP_n(X,\Omega)}$.
Then we have
$$ \free{f}' \circ j_X=f =\free{f}\circ j_X$$
and hence $\free{f}'_0 = \free{f}_0$ since $\frakP_0(X,\Omega)$ is the free monoid generated by $j_X(X)$. Further, $\{\free{f}'_n,\ n\geq 0\}$ also satisfies Property (\mref{it:rest}) by its construction, and satisfies Property (\mref{it:ocomm}) since $\free{f}'$ is a homomorphism of $\Omega$-\mapped monoids. But by our inductive construction of $\{\free{f}_n,\ n\geq 0\}$, such $\free{f}_n$ are unique. Thus we have $\free{f}'_n=\free{f}_n,\ n\geq 0$, and therefore
$\free{f}=\free{f}'$. This proves the uniqueness of $\free{f}$.
\end{proof}

By taking $\Omega$ to be a singleton in Theorem~\mref{thm:motzdec}, we have:
\begin{coro}
\begin{enumerate}
\item
Let $X$ be a non-empty set. The quadruple $(\frakP(X),\circ, \lm\ \rtm, j_X)$ is the free \mapped monoid on $X$.
In particular, $(\frakP, \circ, \lm\ \rtm)$ is the free \mapped monoid on one generator.
\mlabel{it:motmon}
\item
Let $X$ be a non-empty set. The quadruple $(\frakL(X),\circ, \lm\ \rtm, j_X)$ is the free \mapped semigroup on $X$.
In particular, $(\frakL, \circ, \lm\ \rtm)$ is the free \mapped semigroup on one generator.
\mlabel{it:motsg}
\item
The triple $(\frakD,\circ, \lm\ \rtm)$ is the free \mapped monoid on the empty set.
\mlabel{it:motmon0}
\end{enumerate}
\mlabel{co:Motz}
\end{coro}
%Note that the free \mapped semigroup on the empty set is
%the empty set.
%
Recall that for a semigroup (resp. monoid) $Y$, we use $\bfk\, Y$ to denote the corresponding nonunitary (resp. unitary) $\bfk$-algebra. We then have
\begin{coro}
Let $\Omega$ and $X$ be non-empty sets. Let $j_X$ be as defined in Theorem~\mref{thm:motzdec}.
\begin{enumerate}
\item
The quadruple $(\bfk\frakP(X,\Omega),\circ, \{\lm_\omega\ \rtm_\omega\;\big|\; \omega\in \Omega\}, j_X)$ is the free $\Omega$-\mapped $\bfk$-algebra on $X$.
\mlabel{it:motalgdec}
\item
The quadruple $(\bfk \frakL(X,\Omega),\circ, \{\lm_\omega\ \rtm_\omega\;\big| \; \omega\in \Omega\}, j_X)$ is the free $\Omega$-\mapped nonunitary $\bfk$-algebra on $X$.
\mlabel{it:motrngdec}
\item
The quadruple $(\frakD(\Omega),\circ, \{\lm_\omega\ \rtm_\omega\;\big|\; \omega\in \Omega\}, j_X)$ is the free $\Omega$-\mapped $\bfk$-algebra on the empty set.
\mlabel{it:motalg0dec}
\end{enumerate}
\mlabel{co:motzdec}
\end{coro}

\begin{proof}
(\mref{it:motalgdec}).
The forgetful functor from the category $\Omega$-${\bf OAlg}$ of $\Omega$-\mapped algebras to the category $\Set$ of sets is the composition of the forgetful functor from $\Omega$-${\bf OAlg}$ to the category $\Omega$-${\bf OMon}$ of \mapped monoids and the forgetful functor from $\Omega$-${\bf OMon}$ to $\Set$. As is well-known (for example from Theorem 1 in page 101 of~\mcite{Ma}), the adjoint functor of a composed functor is the composition of the adjoint functors. This proves (\mref{it:motalgdec}).

The proofs of the others parts are the same.
\end{proof}

\section{Free \mapped semigroups and monoids in terms of bracketed words}
\mlabel{ss:words}

\subsection{Motzkin words}
We recall the following definition~\mcite{Al,Fl,ST}.
\begin{defn} {\rm
A word from the alphabet set $X\cup \{\lm, \rtm\}$ (often denoted by $X\cup \{U, D\}$) is called an {\bf $X$-decorated Motzkin word} if it has the properties that
\begin{enumerate}
\item
the number of $\lm$ in the word equals the number of $\rtm$ in the word;
\item
counting from the left, the number of occurrence of $\lm$ is always greater or equal to the number of occurrence of $\rtm$.
\end{enumerate}
}
\mlabel{de:motzword}
\end{defn}
Thus an $X$-decorated Motzkin word is an element in the free monoid $M\big(X\cup \{\lm, \rtm\}\big)$ on the set
$X\cup \{\lm, \rtm\}$ with above two properties.
$X$-decorated Motzkin words are used to code Motzkin paths so that every up (resp. down) step in a Motzkin path corresponds to the symbol $\lm$ (resp. $\rtm$) and every level step decorated by $x\in X$ corresponds to $x$.
Under this coding, the set of Dyck paths corresponds to the set of legal bracketings~\mcite{BM,Be}, consisting of words from the alphabet set $\{\lm, \rtm\}$ with the above two properties.

We now generalize the concept of decorated Motzkin words.
Consider the free monoid
\begin{equation}
M_{X,\Omega}=M\big(X\cup \{ \lm_\omega\,\big|\, \omega\in \Omega\} \cup \{\rtm_\omega\, \big| \, \omega\in \Omega\}\big)
\mlabel{eq:mxo}
\end{equation}
on the set $X\cup \{ \lm_\omega\,\big|\, \omega\in \Omega\} \cup \{\rtm_\omega\, \big| \, \omega\in \Omega\}$. For a given $\omega\in \Omega$,
define
$$P_\omega: M_{X,\Omega}\to M_{X,\Omega},\quad
P_\omega(\frakm)=\lm_\omega \frakm \rtm_\omega, \frakm\in M_{X,\Omega}.$$
Then $M_{X,\Omega}$ is an $\Omega$-operated monoid.
Define $\frakW(X,\Omega)$ to be the $\Omega$-operated submonoid of $M_{X,\Omega}$ generated by $X$. Elements of $\frakW(X,\Omega)$ are called {\bf $(X,\Omega)$-decorated Motzkin words.}

We next show that, as in the case of Motzkin words, $(X,\Omega)$-decorated Motzkin words code $(X,\Omega)$-decorated Motzkin paths.

\begin{prop}
The $\Omega$-\mapped monoids $\frakP(X,\Omega)$ and $\frakW(X,\Omega)$ are isomorphic. Consequently, $\frakW(X,\Omega)$ is the free $\Omega$-operated monoid on $X$.
\mlabel{pp:dmotzword}
\end{prop}
\begin{proof}
By the freeness of $\frakP(X,\Omega)$, there is a unique $\Omega$-operated monoid homomorphism
\begin{equation}
 \phi_{\frakP,\frakW}: \frakP(X,\Omega) \to \frakW(X,\Omega)
 \mlabel{eq:freePW}
\end{equation}
such that $\phi_{\frakP,\frakW}(\xmb2) = x, x\in X$.
This homomorphism is surjective since $\frakW(X,\Omega)$ is generated by $X$ as an $\Omega$-operated monoid.

To prove the injectivity of $\phi_{\frakP,\frakW}$, we describe $\phi_{\frakP,\frakW}$ explicitly. Intuitively,
under $\phi_{\frakP,\frakW}$,
\begin{equation}
 \left \{ \begin{array}{l}
\mbox{up step in a Motzkin path decorated by } \omega\in \Omega \leftrightarrow \lm_\omega, \\
\mbox{down step in a Motzkin path decorated by } \omega\in\Omega \leftrightarrow \rtm_\omega, \\
\mbox{level step in a Motzkin path decorated by } x\in X \leftrightarrow x.
\end{array} \right .
\mlabel{eq:freePW2}
\end{equation}

To be more precise, we note that an undecorated Motzkin path from $(0,0)$ to $(0,n)$, as a piecewise linear function $f:[0,n]\to \RR$, is
uniquely determined by $\vec{f}=(f_1,\cdots,f_n)$ where, for each $1\leq i\leq n$,
\begin{align}
f_i: [i-1,i]\to \RR,\quad & f_i(t)=f(i-1)+\left\{
\begin{array}{l} U(t-i+1) \\ D(t-i+1), \\ L(t-i+1)
\end{array} \right . \quad t\in [i-1,i], \mlabel{eq:piece} \\
{\rm where\ }
& \left \{ \begin{array}{l} U(s)=s \\ D(s)=-s, \\ L(s)=0 \end{array} \right .  s\in [0,1].
\notag
\end{align}

More generally, an $(X,\Omega)$-decorated Motzkin path from $(0,0)$ to $(0,n)$ is a piecewise linear function $f:[0,n]\to \RR$ in Eq.~(\mref{eq:piece}) with each linear piece decorated by a certain element of $X\cup \Omega$. Thus an $(X,\Omega)$-decorated Motzkin path is uniquely determined by $\vec{F}=(F_1,\cdots,F_n)$ with $F_i=(f_i,d_i)$ where $f_i$ is as in Eq.~(\mref{eq:piece}) and $d_i\in X\cup \Omega$.
Then we have
\begin{align}
& \phi_{\frakP,\frakW}(F_1,\cdots,F_n)= w_1\cdots w_n, {\rm\ where }
\\
&w_i = \left\{ \begin{array}{ll}
    \lm_\omega, & {\rm if\ } F_i=(f_i,d_i) {\rm\ with\ } f_i=f(i-1)+U(t-i+1) {\rm\ and\ } d_i=\omega\in \Omega, \\
    \rtm_\omega, & {\rm if\ } F_i=(f_i,d_i) {\rm\ with\ } f_i=f(i-1)+D(t-i+1) {\rm\ and\ } d_i=\omega\in\Omega, \\
    x, & {\rm if\ } F_i=(f_i,d_i) {\rm\ with\ } f_i=f(i-1)+L(t-i+1) {\rm\ and\ } d_i=x\in X.
\end{array} \right .
\notag
\end{align}

Now suppose two Motzkin paths $\frakm$ and $\frakm'$ are distinct and are given by
$(F_1,\cdots,F_n)$ and $(F'_1,\cdots,F'_m)$ respectively. If $n\neq m$,  then the two words $\phi_{\frakP,\frakW}(\frakm)=w_1\cdots w_n$ and $\phi_{\frakP,\frakW}(\frakm')=w'_1\cdots w'_m$ have different length and hence are distinct. If $n=m$, then there is a $k\in \{1,\cdots,n\}$ such that
$F_i=F'_i$ for $1\leq i\leq k-1$ and $F_k\neq F'_k$. This means that in $F_k=(f_k,d_k)$ and $F'_k=(f'_k,d'_k)$ either $f_k\neq f'_k$ or $f_k=f'_k$ but $d_k\neq d'_k$.
In either case we have $w_k\neq w'_k$ and hence $\phi_{\frakP,\frakW}(\frakm)\neq \phi_{\frakP,\frakW}(\frakm')$. This proves the injectivity of $\phi_{\frakP,\frakW}$.
\end{proof}

\begin{remark}{\rm
Through the bijection $\phi_{\frakP,\frakW}$, we can use the definition of a $(X,\Omega)$-Motzkin path to characterize an $(X,\Omega)$-Motzkin word to be a word $\frakw\in M_{X,\Omega}$ such that
\begin{enumerate}
\item
ignoring the $\Omega$-decoration of $\frakw$, we have an $X$-decorated Motzkin word;
\item
for any letter $\lm$ in $\frakw$ decorated an $\omega\in\Omega$, its conjugate $\rtm$ is also decorated by the same $\omega$.
\end{enumerate}
Here for each $\lm$ in $\frakw$, the {\bf conjugate} of $\lm$ is the $\rtm$ in $\frakw$ to the right of this $\lm$ such that the subword of $\frakw$ between (and excluding) these $\lm$ and $\rtm$ is a $X$-decorated Motzkin word. The existence and uniqueness of the conjugate follow from the matching down step of an up step in the matching pair of Motzkin paths.
}
\mlabel{rk:pathword}
\end{remark}

\subsection{Bracketed words}

We use the following recursion to give an external construction of $(X,\Omega)$-decorated Motzkin words and hence of the free \mapped semigroup and free \mapped monoid over $X$.

For any set $Y$, let $S(Y)$ denote the free semigroup generated by $Y$,
let $M(Y)$ denote the free monoid generated by $Y$. For a fixed $\omega\in \Omega$, let
$\lc_\omega Y\rc_\omega$ be the set $\{ \lc_\omega y\rc_\omega \big | y\in Y\}$ which is in bijection with $Y$, but disjoint from $Y$. Also assume the sets $\lc_\omega Y\rc_\omega$ to be disjoint with each other as $\omega\in \Omega$ varies.

We now inductively define a direct system
$\{\frakS_n=\frakS_n(X,\Omega), i_{n,n+1}:\frakS_n\to \frakS_{n+1} \}_{n\in \NN}$ of free semigroups
and a direct system
$\{\frakM_n=\frakM_n(X,\Omega), \uni{i}_{n,n+1}: \frakM_n\to \frakM_{n+1} \}_{n\in \NN}$ of free monoids,
both with inclusions as the transition maps, and such that, for each $\omega\in \Omega$,
\begin{equation}
\lc_\omega \frakS_n \rc_\omega \subseteq \frakS_{n+1},\quad
    \lc_\omega \frakM_n \rc_\omega \subseteq \frakM_{n+1},\ n\in \NN.
\mlabel{eq:sgtran}
\end{equation}

We do this by first letting $\frakS_0=S(X)$ and $\frakM_0=M(X)=S(X)\cup \{\bfone\}$, and then
define
$$\frakS_1=S\big(X\cup (\cup_{\omega\in\Omega}\lc_\omega\frakS_0\rc_\omega)\big)
=S\big(X\cup (\cup_{\omega\in\Omega} \lc_\omega S(X)\rc_\omega)\big), \quad
\frakM_1=M\big(X\cup (\cup_{\omega\in \Omega}\lc_\omega \frakM_0\rc_\omega)\big)$$
with $i_{0,1}$ and $\uni{i}_{0,1}$ being the inclusions
\begin{eqnarray*}
 i_{0,1}:&& \frakS_0=S(X) \hookrightarrow \frakS_1=S\big(X\cup (\cup_{\omega\in \Omega}\lc_\omega \frakS_0\rc_\omega)\big),
 \\
 \uni{i}_{0,1}:&& \frakM_0=M(X) \hookrightarrow
    \frakM_1=M\big(X\cup (\cup_{\omega\in \Omega}\lc_\omega \frakM_0\rc_\omega)\big).
\end{eqnarray*}
Clearly, $\lc_\omega \frakS_0 \rc_\omega\subseteq \frakS_1$ and $\lc_\omega \frakM_0\rc_\omega \subseteq \frakM_1$ for each $\omega\in \Omega$.
%Note that since $\lc_\omega \frakM_0\rc_\omega \cap %\frakM_0 =\emptyset$,  $\lc_\omega \bfone\rc_\omega$ is %not the
%identity. Thus $\frakM_1\neq \frakS_1\cup \{\bfone\}$.

Inductively assume that $\frakS_{n-1}$ and $\frakM_{n-1}$ have
been defined for $n\geq 2$,
with the inclusions
\begin{equation}
  i_{n-2,n-1}: \frakS_{n-2} \hookrightarrow \frakS_{n-1} \mmbox{and}
  \uni{i}_{n-2,n-1}: \frakM_{n-2} \to \frakM_{n-1}.
\mlabel{eq:maptrans}
\end{equation}
We then define
\begin{equation}
 \frakS_n:=S\big(X\cup (\cup_{\omega\in \Omega} \lc_\omega \frakS_{n-1}\rc_\omega)\big) \mmbox{and}
 \frakM_n:=M\big(X\cup (\cup_{\omega\in\Omega} \lc_\omega\frakM_{n-1}\rc_\omega)\big ).
 \mlabel{eq:frakm}
 \end{equation}
The inclusions in Eq.~(\mref{eq:maptrans}) give the inclusions
$$
   \lc_\omega\frakS_{n-2}\rc_\omega \hookrightarrow  \lc_\omega \frakS_{n-1} \rc_\omega \mmbox{and}
   \lc_\omega \frakM_{n-2}\rc_\omega \hookrightarrow
   \lc_\omega \frakM_{n-1} \rc_\omega,
$$
yielding inclusions of free semigroups and free monoids
\allowdisplaybreaks{
\begin{eqnarray*}
i_{n-1,n}: \frakS_{n-1} &=& S\big(X\cup (\cup_{\omega\in\Omega} \lc_\omega\frakS_{n-2}\rc_\omega)\big)\hookrightarrow
     S\big(X\cup (\cup_{\omega\in\Omega}\lc_\omega\frakS_{n-1}\rc_\omega) \big) =\frakS_{n},\\
\uni{i}_{n-1,n}: \frakM_{n-1} &=& M\big(X\cup (\cup_{\omega\in \Omega}  \lc_\omega\frakM_{n-2}\rc_\omega)\big)\hookrightarrow
    M\big(X\cup (\cup_{\omega\in\Omega} \lc_\omega \frakM_{n-1}\rc_\omega)\big) =\frakM_{n}.
\end{eqnarray*}}
By Eq.~(\mref{eq:frakm}),
$\lc_\omega \frakS_{n-1} \rc_\omega \subseteq \frakS_n$ and
$\lc_\omega \frakM_{n-1} \rc_\omega \subseteq \frakM_n$.
This completes the inductive construction of the direct systems.
Define the direct limit of semigroups
\begin{equation}
 \frakS(X,\Omega)=\dirlim \frakS_n = \bigcup_{n\geq 0} \frakS_n
\mlabel{eq:uniword}
\end{equation}
whose elements are called {\bf nonunitary bracketed words} and the direct limit of monoids
\begin{equation}
 \frakM(X,\Omega)=\dirlim \frakM_n =\bigcup_{n\geq 0} \frakM_n
\mlabel{eq:nuword}
\end{equation}
whose elements are called {\bf unitary bracketed words}.
Then by Eq.~(\mref{eq:sgtran}), $\lc_\omega \frakS(X,\Omega) \rc_\omega \subseteq \frakS(X,\Omega)$ and $\lc_\omega \frakM(X,\Omega) \rc_\omega \subseteq \frakM(X,\Omega)$ for $\omega\in \Omega$. Thus
$\frakS(X,\Omega)$ (resp. $\frakM(X,\Omega)$) carries an $\Omega$-\mapped semigroup (resp. monoid) structure.

\begin{theorem}
\begin{enumerate}
\item
The $\Omega$-\mapped monoids $\frakM(X,\Omega)$ and $\frakW(X,\Omega)$ are naturally isomorphic.
\mlabel{it:wordword}
\item
The $\Omega$-\mapped monoids $\frakM(X,\Omega)$ and $\frakP(X,\Omega)$ are naturally isomorphic.
\mlabel{it:wordpath}
\item
The $\Omega$-mapped semigroups $\frakS(X,\Omega)$ and $\frakL(X,\Omega)$ are naturally isomorphic.
\mlabel{it:wordpathsg}
\end{enumerate}
\mlabel{thm:motzword}
\end{theorem}

\begin{proof}
(\mref{it:wordword})
By Theorem~\mref{thm:motzdec} and Proposition~
\mref{pp:dmotzword}, $(\frakW(X), \circ, \lm\ \rtm)$ is a free $\Omega$-\mapped monoid on $X$. Thus there is a unique homomorphism of $\Omega$-\mapped monoids
\begin{equation}
\phi_{\frakW,\frakM}: \frakW(X,\Omega) \to \frakM(X,\Omega)
\mlabel{eq:freeWM}
\end{equation}
such that $\phi_{\frakW,\frakM}(x)=x.$

Let $\frakM'$ be the $\Omega$-\mapped submonoid of $\frakM(X,\Omega)$ generated by $X$. Then an inductive argument shows that $\frakM_n\subseteq \frakM'$ for all $n\geq 0$. Thus the $\Omega$-\mapped monoid $\frakM(X,\Omega)$ is generated by $X$ and therefore $\phi_{\frakW,\frakM}$ is surjective.

To prove that $\phi_{\frakW,\frakM}$ is injective, we only need to define a homomorphism of $\Omega$-\mapped monoids \begin{equation}
\phi_{\frakM,\frakW}: \frakM(X,\Omega) \to \frakW(X,\Omega)
\mlabel{eq:freeMW}
\end{equation}
such that $\phi_{\frakM,\frakW}\circ \phi_{\frakW,\frakM}=\id_{\frakW(X,\Omega)}$. For this, we define
$$ \phi_{\frakM,\frakW,n} : \frakM_n \to \frakW(X,\Omega)$$
by induction on $n\geq 0$.
When $n=0$, $\frakM_0$ is the free monoid on $X$. Thus there is a unique monoid isomorphism
$\phi_{\frakM,\frakW,0}: \frakM_0 \to \frakW(X,\Omega)$ sending $x$ to $x$, $x\in X$.

Suppose a monoid homomorphism $\phi_{\frakM,\frakW,n}: \frakM_n \to \frakW(X,\Omega)$ has been defined for $n\geq 0$. Then for each $\omega\in \Omega$, we obtain a map
\begin{equation} \lc_\omega \phi_{\frakM,\frakW,n} \rc_\omega: \lc_\omega \frakM_n \rc_\omega \to \frakW(X,\Omega), \quad \lc_\omega \frakm \rc_\omega \mapsto \lm_\omega \phi_{\frakM,\frakW,n}(\frakm) \rtm_\omega.
\mlabel{eq:pathword}
\end{equation}
Thus we obtain a homomorphism of monoids
$$ \phi_{\frakM,\frakW,n+1}: \frakM_{n+1}=M\big(X\cup (\cup_{\omega\in \Omega} \lc_\omega \frakM_n \rc_\omega)\big)
\to \frakW(X,\Omega) $$
with $\phi_{\frakM,\frakW,n+1}\big|_{\frakM_n}=
\phi_{\frakM,\frakW,n}$.
Taking the direct limit, we obtain a monoid homomorphism.
\begin{equation}
\phi_{\frakM,\frakW}=\dirlim \phi_{\frakM,\frakW,n}: \frakM(X,\Omega)=\dirlim \frakM_n \to \frakW(X,\Omega).
\mlabel{eq:wordpath}
\end{equation}
By Eq.~(\mref{eq:pathword}), for each $\omega\in \Omega$, the bracket operator $\lc_\omega\ \rc_\omega$ on $\frakM(X,\Omega)$ is compatible with the operator $P_\omega$ on $\frakW(X,\Omega)$. Thus $\phi_{\frakM,\frakW}$ is an $\Omega$-\mapped monoid homomorphism.
Further, since $\phi_{\frakM,\frakW}\circ \phi_{\frakW,\frakM}$ is the identity on $X$, by the universal property of $\frakW(X,\Omega)$, we must have
$\phi_{\frakM,\frakW}\circ \phi_{\frakW,\frakM}=\id_{\frakW(X,\Omega)}$.
\medskip

\noindent
(\mref{it:wordpath}). The isomorphism is
\begin{equation}
\phi_{\frakM,\frakP}=\phi_{\frakP,\frakW}^{-1}\circ \phi_{\frakM,\frakW}
\mlabel{eq:freeMP}
\end{equation}
for $\phi_{\frakP,\frakW}$ in Eq.~(\mref{eq:freePW}) and $\phi_{\frakM,\frakW}$ in Eq.~(\mref{eq:freeMW}).
\medskip

\noindent
(\mref{it:wordpathsg})
Since $\frakS(X,\Omega)$ (resp. $\frakL(X,\Omega)$) is the $\Omega$-\mapped semigroup of $\frakM(X,\Omega)$ (resp. $\frakP(X,\Omega)$) generated by $X$ (resp. $\{\xmb2\,\big|\, x\in X\}$), the $\Omega$-\mapped monoid isomorphism $\phi_{\frakM,\frakP}$ in Eq.~(\mref{eq:freeMP}) restricts to an isomorphism
$\phi_{\frakS,\frakL}: \frakS(X,\Omega)\to \frakL(X,\Omega)$ of $\Omega$-\mapped semigroups.
\end{proof}

\begin{remark}{\rm
As we can see in the proof of Theorem~\mref{thm:motzword}, the isomorphism $\phi_{\frakM,\frakW}: \frakM(X,\Omega) \to \frakW(X,\Omega)$ is almost like the identity map. For example,
\begin{eqnarray*}
\phi_{\frakM,\frakW}(x \lc_\omega y \rc_\omega) &=& \phi_{\frakM,\frakW}(x) \phi_{\frakM,\frakW}(\lc_\omega y\rc_\omega)
= x \lm_\omega y \rtm_\omega.
\end{eqnarray*}
The difference is that in $x\lc_\omega y\rc_\omega$, $\lc_\omega y\rc_\omega$ is a new symbol, while in $x\lm_\omega y\rtm_\omega$, $\lm_\omega y\rtm_\omega$ is a word consisting of the three symbols $\lm_\omega, y$ and $\rtm_\omega$.
Thus we can identify $\frakM(X,\Omega)$ with the $\frakW(X,\Omega)$, allowing us to use $\frakM(X,\Omega)$ to give a recursive description of $\frakW(X,\Omega)$ and use $\frakW(X,\Omega)$ to give an explicit description of $\frakM(X,\Omega)$.
\mlabel{rk:words}
}
\end{remark}

By Theorem~\mref{thm:motzdec} and Theorem~\mref{thm:motzword}, we have
\begin{coro}
Let $j_X$ denote the natural embeddings $X \to \frakM(X,\Omega)$ or $X \to \frakS(X,\Omega)$.
\begin{enumerate}
\item
    With the word concatenation product, the triple $(\frakM(X,\Omega),\lc\; \rc, j_X)$ is the free $\Omega$-\mapped monoid on $X$.
    \mlabel{it:mapmon}
\item
    With the word concatenation product, the triple $(\frakS(X,\Omega),\lc\; \rc, j_X)$ is the free $\Omega$-\mapped semigroup on $X$.
    \mlabel{it:mapsg}
\end{enumerate}
\mlabel{co:wordfree}
\end{coro}
We will use $j_X$ here and later to denote the natural embeddings of $X$ to various \mapped monoids and \mapped semigroups. The meaning will be clear from the context.
By the same proof as for Corollary~\mref{co:motzdec}, we further have:
\begin{coro}
Let $j_X$ denote the natural embeddings from $X$ to $\bfk\,\frakM(X)$ or $\bfk\,\frakS(X)$.
\begin{enumerate}
\item
    The triple $(\bfk\,\frakM(X),\lc\; \rc, j_X)$ is the free \mapped (unitary) $\bfk$-algebra on $X$.
    \mlabel{it:mapalgw}
\item
    The triple $(\bfk\,\frakS(X),\lc\; \rc, j_X)$ is the free \mapped nonunitary $\bfk$-algebra on $X$.
    \mlabel{it:mapnualgw}
\end{enumerate}
\mlabel{co:freetm}
\end{coro}

%%%%%%%%%%%%%%%%%%%%%%%%%%%%%%%%%%%%%%%%%%%%%%%
\section{Free \mapped semigroups in terms of planar rooted trees}
\mlabel{sec:trees}

We follow the notations and terminologies in~\cite{Di,We}. A free tree is an undirected graph that is connected and
contains no cycles. A {\bf rooted tree} is a free tree in which a particular vertex has been distinguished as the {\bf root}. A {\bf planar rooted tree} (or ordered tree) is a rooted tree with a fixed embedding into the plane.
The {\bf depth} $\depth(T)$ of a rooted tree $T$ is the length of the longest path from its root to its leaves.

Let $\calt$ be the set of planar rooted trees.
A {\bf planar rooted forest} is a noncommutative concatenation of planar rooted trees, denoted by
$T_1\sqcup \cdots \sqcup T_b$ or simply $T_1\, \cdots\, T_b$, with $T_1,\cdots, T_b\in\calt$.
The number $b=\bread(F)$ is called the {\bf breadth} of $F$.
The {\bf depth} $\depth(F)$ of $F$ is the maximum of the depths of the trees $T_i, 1\leq i\leq b$.
Let $\calf$ be the set of {planar rooted forests}. Then
$\calf$ is the free semigroup generated by $\calt$ with the tree concatenation product.
\begin{remark}{\rm
For the rest of this paper, a tree or forest means a planar rooted tree or a planar rooted forest unless otherwise specified.
}
\end{remark}
Let $\lc T_1\, \cdots \, T_b\rc$ denote the usual {\bf grafting} of the trees $T_1,\cdots,T_b$
by adding a new root together with an edge from the new root to the root of each of the trees $T_1,\cdots, T_b$.

For two non-empty sets $X$ and $\Omega$, let $\calf(X,\Omega)$ be the set of planar rooted forests whose leaf vertices are decorated by elements of $X$ and non-leaf vertices are decorated by elements of $\Omega$.
The only vertex of the tree $\onetree$ is taken to be a leaf vertex. For example,
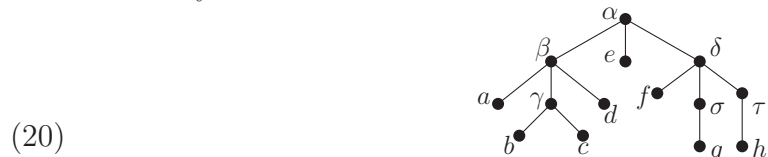
\begin{equation}
 \bigdect
 \mlabel{eq:bigdect}
\end{equation}
with $a,b,c,d,e,f,g,h\in X$ and $\alpha,\beta,\gamma,\delta,\sigma,\tau\in \Omega$.
As special cases, we have $\calf(X,X)$ of planar rooted forests whose vertices are decorated by elements of $X$. When $X=\Omega=\{x\}$,
$\calf(X,X)$ is identified with the planar rooted forests without decorations.

As in the above case of planar rooted forests without decorations, $\calf(X,\Omega)$, with the concatenation product, is a semigroup. Further, for $\omega\in \Omega$ and $F=T_1\,\cdots\,T_b\in \calf(X,\Omega)$, let $\lc_\omega\, F\,\rc_\omega$ be the grafting of $T_1,\cdots,T_b$ with the new root decorated by $\omega$.  Then with the grafting operators $\lc_\omega\ \rc_\omega, \omega\in \Omega$, $\calf(X,\Omega)$ is an $\Omega$-\mapped semigroup.

We describe the recursive structure on $\calf(X,\Omega)$ in algebraic terms.
For any subset $Y$ of $\calf(X,\Omega)$, let $\sqmon{Y}$ be the sub-semigroup of $\calf(X,\Omega)$ generated by $Y$.
Let $\calf_0(X,\Omega)=\sqmon{\{\bullet_x\ \big|\, x\in X\}}$, consisting of forests composed of trees $\onetree_x, x\in X$.
These are also the forests decorated by $X$ of depth zero.
Then recursively define
\begin{equation}
 \calf_n(X,\Omega)= \sqmon{X\cup (\cup_{\omega\in \Omega}\lc_\omega \calf_{n-1}(X,\Omega)\rc_\omega)}.
 \mlabel{eq:treex}
\end{equation}
It is clear that $\calf_n(X,\Omega)$ is the set of $(X,\Omega)$-decorated forests with depth less or equal to $n$ and
\begin{equation}
 \calf(X,\Omega) =\cup_{n\geq 0} \calf_n(X,\Omega)=\dirlim \calf_n(X,\Omega).
 \mlabel{lem:treex}
 \end{equation}

\begin{theorem}
Let $X$ and $\Omega$ be non-empty sets.
\begin{enumerate}
\item
$\calf(X,\Omega)$ is the free $\Omega$-\mapped semigroup on $X$.
\mlabel{it:treesg}
\item
\bfk $\calf(X,\Omega)$ is the free $\Omega$-\mapped non-unitary algebra on $X$.
\mlabel{it:treealg}
\end{enumerate}
\mlabel{thm:treefree}
\end{theorem}

\begin{proof}
(\mref{it:treesg})
It can be proved directly following the proof of Theorem~\mref{thm:motzdec}.(\mref{it:motsgdec}), using the recursive structure on $\calf(X,\Omega)$ in Eq.~(\mref{eq:treex}). Just replace $\frakL(X,\Omega)$ and $\frakL_n(X,\Omega), n\geq 0$ by $\calf(X,\Omega)$ and $\calf_n(M,\Omega)$. We leave the details to the interested reader and turn to an indirect proof by showing that the $\Omega$-\mapped semigroup $\calf(X,\Omega)$ is isomorphic to the $\Omega$-\mapped semigroup $\frakL(X,\Omega)$. Hence $\calf(X,\Omega)$ is free by Theorem~\mref{thm:motzdec}.(\mref{it:motsgdec}).

We obtain such an isomorphism by starting with the natural set map
$$ f: X\to\calf(X,\Omega),\ x\mapsto \onetree_x,\ x\,\in\, X.$$
Then by Theorem~\mref{thm:motzdec}.(\mref{it:motsgdec}), there is a unique homomorphism
\begin{equation}
\phi_{\frakL,\calf}: \frakL(X,\Omega) \to \calf(X,\Omega)
\mlabel{eq:pathtree}
\end{equation}
of $\Omega$-\mapped semigroups such that
$\phi_{\frakL,\calf}(\xmb2 )= \onetree_x, x\in X.$
We only need to show that $\phi_{\frakL,\calf}$ is bijective.

By an inductive argument, the $\Omega$-\mapped subsemigroup of $\calf(X,\Omega)$ generated by $X$ contains $\calf_n(X,\Omega)$ for all $n\geq 0$. Thus by Eq.~(\mref{eq:treex}), $\calf(X,\Omega)$ is generated by $X$ as an $\Omega$-operated semigroup. Thus $\phi_{\frakL,\calf}$ is surjective. To prove that $\phi_{\frakL,\calf}$ is injective, we construct a homomorphism of $\Omega$-\mapped semigroups
\begin{equation}
\phi_{\calf,\frakL}: \calf(X,\Omega) \to \frakL(X,\Omega)
\mlabel{eq:freeFL}
\end{equation}
such that $\phi_{\calf,\frakL}\circ \phi_{\frakL,\calf}=\id_{\frakL(X,\Omega)}$.
This follows by multiplicity from a map
\begin{equation}
\phi_{\calf,\frakL}: \calt(X,\Omega) \to \frakL(X,\Omega) \cap
    \frakI(X,\Omega).
\mlabel{eq:treepathex}
\end{equation}
This  explicitly defined combinatorial bijection might be of interest on its own right. Trees and Motzkin paths have been related in previous works such as~\mcite{D-S}.

To define $\phi_{\calf,\frakL}$ in Eq.~(\mref{eq:treepathex}), we first combine the well-known
processes of preorder and postorder of traversing a planar rooted tree to define the process of biorder.
The {\bf vertex biorder list} of a tree $T\in \calt(X,\Omega)$ is defined as follows.
\begin{enumerate}
\item
If $T$ has only one vertex, then that vertex is the vertex biorder list of $T$;
\item
If $T$ has more than one vertices, then the root vertex of $T$ has branches $T_1,\cdots,T_k$, $k\geq 1$, listed from the left to the right. Then the vertex biorder list of $T$ is the root of $T$, followed by the vertex biorder list of $T_1$, $\cdots$, followed by the vertex biorder list of $T_k$, {\em followed by the root of $T$}.
\end{enumerate}
We use the adjective vertex with biorder to distinguish it from the edge biorder list to be introduced in Remark~\mref{rk:anglepath}.
For example, the vertex biorder list of the tree in Eq.~(\mref{eq:bigdect}) is
\begin{equation}
\alpha \beta a \gamma b c \gamma d \beta e \delta f \sigma g \sigma \tau h \tau \delta \alpha
\mlabel{eq:treewordex}
\end{equation}

It is clear that a vertex appears exactly once in the list if it is a leaf and exactly twice if it is not a leaf. The only vertex of $\onetree$ is taken to be a leaf. Thus we can record $\lm_\omega$ (instead of $\omega$) if a non-leaf vertex decorated by $\omega$ is listed for the first time and $\rtm_\omega$ (instead of $\omega$) if this vertex is listed for the second time.
This gives a word in $M_{X,\Omega}$ defined in Eq.~(\mref{eq:mxo}) and satisfies the required properties of an $(X,\Omega)$-Motzkin word described in Remark~\mref{rk:pathword} and therefore gives an $(X,\Omega)$-Motzkin word.
For example, the Motzkin word from Eq.~(\mref{eq:treewordex}) is
\begin{equation}
\lm_\alpha \lm_\beta\; a \lm_\gamma\; b\; c \rtm_\gamma d \rtm_\beta\; e \lm_\delta f \lm_\sigma\; g \rtm_\sigma \lm_\tau\; h \rtm_\tau \rtm_\delta \rtm_\alpha
\mlabel{eq:treeword2}
\end{equation}

Then through the bijection $\phi_{\frakP,\frakW}$ in Eq.~(\mref{eq:freePW2}), such an $(X,\Omega)$-decorated Motzkin word gives an $(X,\Omega)$-decorated Motzkin path.
As an example, the Motzkin word in Eq.~(\mref{eq:treeword2}) above corresponds to the  Motzkin path
\begin{equation}
 \bigdecm
 \vspace{.5cm}
\mlabel{eq:bigdecm}
\end{equation}
This completes the construction of $\phi_{\calf,\frakL}$ in Eq.~(\mref{eq:treepathex}) and hence in Eq.~(\mref{eq:freeFL}).

It is clear that this correspondence sends the concatenation of rooted forests to the link of Motzkin paths and sends the grafting operator of rooted forests to the raising operator of peak-free Motzkin paths. Since $(\phi_{\calf,\frakL}\circ \phi_{\frakL,\calf})(\xmb2)=\xmb2$ for $x\in X$, by the freeness of $\frakL(X,\Omega)$, we have $\phi_{\calf,\frakL}\circ \phi_{\frakL,\calf}=\id_{\frakL(X,\Omega)}$, as needed.

\medskip

\noindent
(\mref{it:treealg}) The proof follows in the same way as the proof of Corollary~\mref{co:freetm}.
\end{proof}

\section{Some natural bijections}
\mlabel{sec:bij}
As noted in Example~\mref{ex:mapalg}, well-known algebras with operators such as differential algebras, difference algebras and Rota-Baxter algebras are \mapped algebra with additional conditions on their operators. Thus the free objects in these categories are quotients of free \mapped algebras.
In this context, results in previous sections show that free objects in these categories are quotients of the \mapped algebras of Motzkin paths or planar rooted forests. For some of these categories it is possible to find a canonical subset of the Motzkin paths or planar rooted forests that projects to a basis in these quotients.
This is the case for Rota-Baxter algebras. It was shown in \mcite{A-M, E-G0} that free Rota-Baxter algebras on a set can be constructed from a subset of planar rooted forests with decorations on the angles. We now give similar constructions in terms of Motzkin paths and leaf decorated trees, as well as in terms of bracketed words which relates to the construction in~\mcite{E-G4}. For certain applications, these new constructions are probably more natural then the construction in terms of rooted trees with decorations on the angles.
We remark that there are several variations of constructions of free Rota-Baxter algebras, depending on whether they are regarded as the adjoint functors of the forgetful functors from the category of Rota-Baxter algebras to the category of algebras, or to the category of modules, or to the category of sets. We will construct the free Rota-Baxter algebra on a set as was considered in~\mcite{A-M,E-G0}. The constructions of the other cases are similar.

\subsection{Angularly decorated forests}
\mlabel{ss:adf}
For later references, we review the concept of
angularly decorated planar rooted forests. See~\mcite{E-G0} for further details.

Let $X$ be a non-empty set.
Let $F\in \calf$ with $\leaf=\leaf(F)$ leaves.
Let $X^F$ denote the set of pairs
$(F;\vec{x})$ where $\vec{x}$ is in $X^{(\ell(F)-1)}$.
We use the convention that $X^{\onetree}=\{(\onetree;1)\}$.
Let $X^\calf = \dotcup_{F\in \calf} X^F$.
We call $(F;\vec{x})$ an {\bf angularly decorated forest} since it can be identified with the forest $F$ together with an ordered decoration by $\vec{x}$ on the angles of $F$.
For example, we have
$$
    \big( \tg42\ ;\ x     \big) = \xldec41r, \quad
    \big( \thII43\ ;\ (x, y) \big) = \xyldec43, \quad
    \big( \ta1\ \ \td31\ ;\ (x, y) \big)=  \ta1\ {x}\!\! \begin{array}{l}\\[-.3cm]
    \ytd31 \end{array}.
$$
$\ta1\ {x}\!\! \begin{array}{l}\\[-.3cm]
    \ytd31 \end{array}$ is denoted by $\ta1 \sqcup_{x} \!\!\! \begin{array}{l}\\[-.3cm]
    \ytd31 \end{array}$ in~\cite{E-G0}.

Let $(F;\vec{x})\in X^F$.
Let $F=T_1 \cdots  T_b$ be the decomposition of $F$
into trees. We consider the corresponding decomposition of the decorated forest. If $b=1$, then $F$ is a tree and $(F;\vec{x})$
has no further decompositions. If $b>1$,
denote $\leaf_i=\leaf(T_i), 1\leq i\leq b$.
Then
$$(T_1;(x_1,\cdots, x_{\leaf_1-1})),\
(T_2; (x_{\leaf_1+1}, \cdots, x_{\leaf_1+\leaf_2-1})),
\cdots,
(T_b; (x_{\leaf_1+\cdots+\leaf_{b-1}+1}, \cdots, x_{\leaf_1+\cdots+\leaf_b}))
$$
are well-defined angularly decorated trees when  $\leaf(T_i)>1$.
If $\leaf(T_i)=1$, then $x_{\leaf_{i-1}+\leaf_i-1}=x_{\leaf_{i-1}}$
and we use
the convention $(T_i;x_{\leaf_{i-1}+\leaf_i-1})=(T_i;\bfone)$.
Thus we have,
\begin{eqnarray*}
(F;(x_1,\cdots, x_{\leaf-1}))&=&
(T_1;(x_1, \cdots, x_{\leaf_1-1})){x_{\leaf_1}}
(T_2; (x_{\leaf_1+1},\cdots, x_{\leaf_1+\leaf_2-1}))
    {x_{\leaf_1+\leaf_2}}
\\
&&\cdots {x_{\leaf_1+\cdots+\leaf_{b-1}}}
(T_b; (x_{\leaf_1+\cdots+\leaf_{b-1}+1}, \cdots, x_{\leaf_1+\cdots+\leaf_b})).
\end{eqnarray*}
We call this the {\bf standard decomposition} of $(F;\vec{x})$ and
abbreviate it as
\begin{eqnarray}
(F;\vec{x})&=&(T_1;\vec{x}_1){x_{i_1}}
(T_2;\vec{x}_2) {x_{i_2}} \cdots {x_{i_{b-1}}}
(T_b;\vec{x}_b)
= D_1 {x_{i_1}} D_2 {x_{i_2}} \cdots {x_{i_{b-1}}} D_b
\mlabel{eq:stdecm}
\end{eqnarray}
where $D_i=(T_i;\vec{x}_i), 1\leq i\leq b$.
For example,
$$
    \big(\ta1\ {\scalebox{1.15}{\tg42}}  \ {\scalebox{1.15}{\td31}}\, ; (v, x, w, y) \big)
    = \big(\ta1\,;\bfone \big)\, v \big({\scalebox{1.15}{\tg42}}\,;x)\, w \big({\scalebox{1.15}{\td31}}\,;y \big)
    = \ta1 \,{v}  \xldec41r \,{w} \ytd31
$$
We note that even though $X^\calf$ is not closed under concatenation of forests, it is closed under the grafting operator:
$\lc (F;\vec{x}) \rc = (\lc F\rc; \vec{x}).$

\subsection{The bijections}
We list below all the objects we have encountered so far in order to study their relations in Theorem~\mref{thm:diag}.
Let $X$ be a non-empty set and let $\Omega$ be a singleton. Thus we will drop $\Omega$ in the following notations.
\begin{enumerate}
\item $\frakM(X)$ is the \mapped monoid of unitary bracketed words on the alphabet set $X$, defined in Eq.~(\mref{eq:uniword});
\item $\frakS(X)$ is the \mapped semigroup of nonunitary bracketed words on the alphabet set $X$, defined in Eq.~(\mref{eq:nuword});
\item $\frakR(X)$ is the set of {\bf Rota-Baxter bracketed words}\mcite{E-G4,G-S} on the alphabet set $X$. Such a word is defined to be a word in $\frakM(X)$ that does not contain $\rc\,\lc$\,, such $\lc x\rc \lc y\rc$;
\item $\frakP(X)$ is the set of Motzkin paths with level steps decorated by $X$, considered in Corollary~\mref{co:Motz};
\item $\frakL(X)$ is the set of peak-free Motzkin paths with level steps decorated by $X$, considered in Corollary~\mref{co:Motz};
\item $\frakV(X)$ is the set of valley-free Motzkin paths with level steps decorated by $X$, consisting of Motzkin paths in $\frakP(X)$ with no down step followed immediately by an up step;
\item $\calf(X)$ is the set of planar rooted forests with leaves decorated by $X$, defined in Eq.~(\mref{lem:treex});
\item Define $\calf_\ell(X)$ to be the subset of $\calf(X)$ consisting of leaf decorated forests that do not have a vertex with adjacent non-leaf branches. Such a forest is called {\bf leaf-spaced}. For example, the tree
$$ \bigdectl
$$
is not leaf-spaced since the two right most branches, with leaves decorated by $g$ and $h$, are not separated by a leaf branch. But the tree
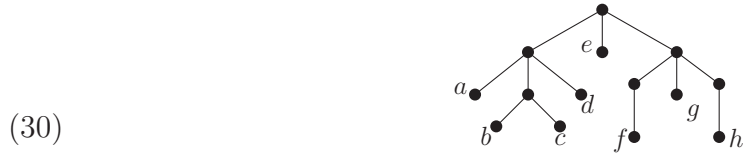
\begin{equation}
\bigdectls
\mlabel{eq:bigdectls}
\end{equation}
is leaf-spaced;
\item $X^\calf$ is the set of planar rooted tree with angles decorated by $X$, defined in Section~\mref{ss:adf};
\item $X^\calf_0$ be the subset of $X^\calf$ consisting of {\bf ladder-free forests}, namely those forests that do not have a {\bf ladder tree}, the latter being defined to be a subtree $\neq \onetree$ with only one leaf. Equivalently, a ladder-free forest is a forest $\neq \onetree$ that does not have a subtree $\tb2$\,. For example,
$\xtd31$ is ladder-free, but $\xldec41r$ is not ladder-free because of its right branch.
\end{enumerate}

\begin{theorem}
There are maps among the sets in the above list that fit into the following commutative diagram of bijection and inclusions.
$$
 
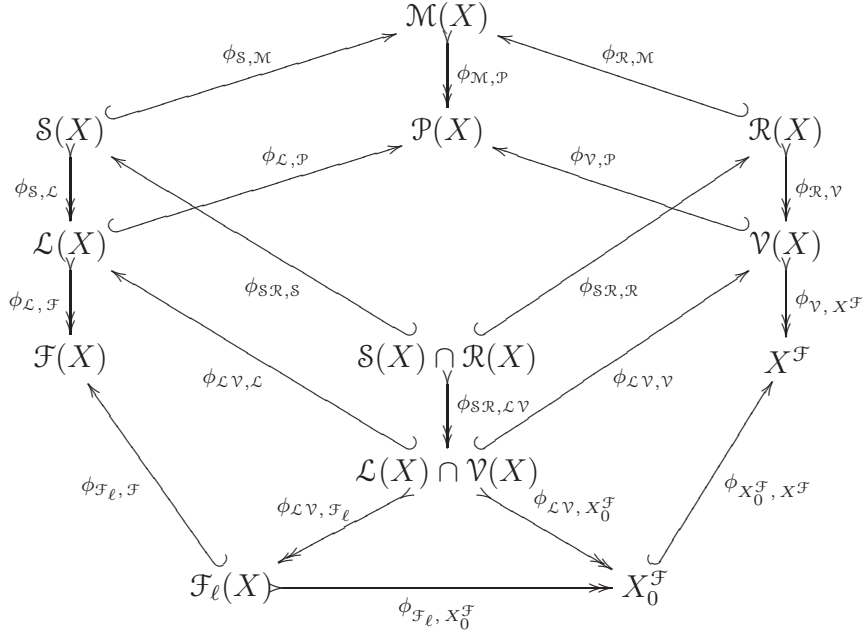
\xymatrix{
            && \frakM(X)  \ar@{>->>}[d]^{\phi_{\frakM,\frakP}} && \\
\frakS(X) \ar@{>->>}[d]_{\phi_{\frakS,\frakL}} \ar@{^{(}->}[urr]^{\phi_{\frakS,\frakM}} && \frakP(X)&&
        \frakR(X) \ar@{>->>}[d]^{\phi_{\frakR,\frakV}} \ar@{_{(}->}[ull]_{\phi_{\frakR,\frakM}}\\
\frakL(X) \ar@{>->>}[d]_{\phi_{\frakL,\,\calf}} \ar@{^{(}->}[urr]^(0.6){\phi_{\frakL,\frakP}} & && &
        \frakV(X) \ar@{>->>}[d]^{\phi_{\frakV,\,X^\calf}} \ar@{_{(}->}[ull]_(0.6){\phi_{\frakV,\frakP}} \\
\calf(X)        &&\frakS(X)\cap \frakR(X) \ar@{>->>}[d]^(.4){\phi_{\frakS\frakR,\frakL\frakV}} \ar@{_{(}->}[uull]^(.4){\phi_{\frakS\frakR,\frakS}}
                \ar@{^{(}->}[uurr]_(0.4){\phi_{\frakS\frakR,\frakR}}&& X^\calf     \\
             &&\frakL(X) \cap \frakV(X) \ar@{>->>}[dl]_{\phi_{\frakL\frakV,\,\calf_\ell}} \ar@{_{(}->}[uull]^{\phi_{\frakL\frakV,\frakL}}
                \ar@{^{(}->}[uurr]_{\phi_{\frakL\frakV,\frakV}}
                \ar@{>->>}[dr]^{\phi_{\frakL\frakV,\,X_0^\calf}}&&\\
             & \calf_\ell(X) \ar@{>->>}[rr]_{\phi_{\calf_\ell,\,X_0^\calf}} \ar@{_{(}->}[uul]^{\phi_{\calf_\ell,\,\calf}}
                 &&
              X^\calf_0 \ar@{^{(}->}[uur]_{\phi_{X_0^\calf,\,X^\calf}}&
                        }
$$
The maps will be recalled or defined in the proof. Each of the maps is compatible with the products, whenever defined, and is compatible with the distinguished operators.
\mlabel{thm:diag}
\end{theorem}
\begin{proof}
All the inclusions are clear from the definition of the sets. So we only need to verify the claimed properties for the bijective maps.

The isomorphism $\phi_{\frakM,\frakP}$ is obtained in Eq.~(\mref{eq:freeMP}).
The isomorphism $\phi_{\frakS,\frakL}$ is the restriction of $\phi_{\frakM,\frakP}$. See Theorem~\mref{thm:motzword}.(\mref{it:wordpathsg}) and its proof.
The isomorphism $\phi_{\frakL,\calf}$ is defined in Eq.~(\mref{eq:pathtree}) whose inverse is $\phi_{\calf,\frakL}$ in Eq.~(\mref{eq:freeFL}).

$\phi_{\frakR,\frakV}$ is the restriction of the \mapped monoid isomorphism $\phi_{\frakM,\frakP}$ to $\frakR(X)$. By Theorem~\mref{thm:motzword} and Remark~\mref{rk:words}, a bracketed word $W$ is in $\frakR(X)$ if and only if the corresponding Motzkin word does not have a $\rtm$ followed immediately by a $\lm$ which means $\phi_{\frakM,\frakP}(W)$ is in $\frakV(X)$.

$\phi_{\frakS\frakR,\frakL\frakV}$ is defined to be the restriction of $\phi_{\frakM,\frakP}$. Hence its bijectivity follows from the bijectivities of $\phi_{\frakS,\frakL}$ and $\phi_{\frakR,\frakV}$ which are both restrictions of $\phi_{\frakM,\frakP}$.

$\phi_{\frakL\frakV,\calf_\ell}$ is defined to be the restriction of the bijective map $\phi_{\frakL,\calf}$ to $\frakL(X)\cap \frakV(X)$.
By the explicit description of $\phi_{\calf,\frakL}=\phi_{\frakL,\calf}^{-1}$ in the proof of Theorem~\mref{thm:treefree}, a Motzkin path $\frakm\in \frakL(X)$ has a peak if and only if the corresponding leaf decorated rooted forest $\phi_{\frakL,\calf}(\frakm)=\phi_{\frakL,\calf}(\frakm) \in \calf(X)$ has a vertex with two adjacent non-leaf branches. The bijectivity of $\phi_{\frakL\frakV,\calf_\ell}$ follows.

We will define below a bijective map
\begin{equation}
 \phi_{\frakV,X^\calf} : \frakV(X) \to X^\calf
 \mlabel{eq:pathatree0}
 \end{equation}
that restricts to a bijective map
\begin{equation}
\phi_{\frakL\frakV, X_0^\calf}: \frakL(X)\cap\frakV(X)
\to X_0^\calf
\mlabel{eq:pathatree2}
\end{equation}
and is compatible with the distinguished operators.
Then by composition, we obtain a bijective map
$$\phi_{\calf_\ell,X_0^\calf}= \phi_{\frakL\frakV,X_0^\calf}\circ \phi_{\frakL\frakV,\calf_\ell}^{-1}: \calf_\ell(X) \to \frakL(X)\cap \frakV(X) \to X_0^\calf$$
that is compatible with the distinguished operators.
See also Remark~\mref{rk:leafangle} for an explicit combinatorial description of $\phi_{\calf_\ell,X_0^\calf}.$

Thus it remains to construct a bijective map
\begin{equation}
 \phi_{\frakV,X^\calf} : \frakV(X) \to X^\calf
 \mlabel{eq:pathatree1}
 \end{equation}
with the prescribed properties above.
For this we define the direct system of mutually inverse maps
$$ \phi_{\frakV,X^\calf,n}: \frakV_n(X) \to X^{\calf_n},
\quad \phi_{X^\calf,\frakV,n}: X^{\calf_n} \to \frakV_n(X)
$$
by induction on the heights $n\geq 0$ of Motzkin paths and forests. See Remark~\mref{rk:anglepath} for an explicit combinatorial description of $\phi_{X^\calf,\frakV}$.

When $n=0$, $\frakV_n(X)=\frakP_n(X)$ is the monoid of Motzkin paths of height 0. Such a path is either
$\onetree$ or
$$\frakm= \begin{array}{c}\tiny{x_{1}}\vspace{-.3cm}\\ \mb2 \end{array}
 \circ
 \begin{array}{c}\tiny{x_2}\vspace{-.3cm}\\ \mb2 \end{array}
 \circ \cdots \circ
 \begin{array}{c}\tiny{x_b}\vspace{-.3cm}\\ \mb2 \end{array}$$
Accordingly define $\phi_{\frakV,X^\calf,0}(\onetree)=(\onetree\,;\bfone)$ and
\begin{equation}
\phi_{\frakV,X^\calf,0}(\frakm)= \onetree x_1\, \onetree\,x_2\, \onetree \cdots \onetree\, x_k\, \onetree =
(\onetree;\bfone)x_1(\onetree;\bfone)x_2 (\onetree;\bfone) \cdots (\onetree;\bfone) x_b (\onetree;\bfone).
\mlabel{eq:pathatree4}
\end{equation}
This is clearly a bijection from $\frakV_0(X)$ to $X^{\calf_0}$ since any angularly decorated forest of height 0 is either $\onetree$ or $\onetree x_1\, \onetree\,x_2\, \onetree \cdots \onetree\, x_b\, \onetree$. Take $\phi_{X^\calf,\frakV,0}=\phi_{\frakV,X^\calf,0}^{-1}.$

Let $k\geq 0$. Suppose
$\phi_{\frakV,X^\calf,k}: \frakV_k(X) \to X^{\calf_k}$ and its inverse $\phi_{X^\calf,\frakV,k}: X^{\calf_k} \to \frakV_k(X)$ have been defined.
For any $\frakm\in \frakV_{k+1}(X)$, let
\begin{equation}
\frakm=\frakm_1\circ \cdots \circ \frakm_p.
\mlabel{eq:motzdec}
\end{equation}
be the decomposition of $\frakm$ into indecomposable decorated Motzkin paths $\frakm_i\neq \onetree, 1\leq i\leq p$.
Since $\frakm$ is valley free, there are no consecutive $\frakm_i$ and $\frakm_{i+1}$ that have height greater or equal to 1.
Thus we can rewrite Eq.~(\mref{eq:motzdec}) as follows.
If $\frakm_1=\xmb2$, then rewrite $\frakm_1$ as $\onetree \circ \frakm_1$. If $\frakm_p=\xmb2$, then rewrite $\frakm_p=\frakm_p\circ \onetree$. If $\frakm_i\circ \frakm_{i+1}=\xmb2 \circ \ymb2$, then rewrite
$\frakm_i \circ \onetree \circ \frakm_{i+1}$.
Note that any indecomposable Motzkin path is of the form $\onetree, \xmb2$ or an indecomposable Motzkin path of height at least one.
In this way
Eq.~(\mref{eq:motzdec}) is uniquely rewritten as
\begin{equation}
\frakm=V_1 \circ \begin{array}{c}\scriptsize{x_{i_1}}\vspace{-.3cm}\\ \mb2 \end{array}
\circ V_2 \circ \cdots \circ V_{b-1}\circ \begin{array}{c}\scriptsize{x_{i_{b-1}}}\vspace{-.3cm}\\ \mb2 \end{array} \circ V_b
\mlabel{eq:motzdec1}
\end{equation}
where each $V_j, 1\leq j\leq b,$ is either $\onetree$ or an indecomposable valley-free Motzkin path of height at least one.
Call this the {\bf standard decomposition} of $\frakm$.
One example of such a standard decomposition is
$$
\motzstand  =  \rhsmotzstand
$$
Note that an indecomposable valley-free Motzkin path $V_j$ of height at least one is of the form $\lm \oV_j \rtm$ for another Motzkin path $\oV_j\in\frakV_k(X)$.
We can then define
\begin{equation}
\phi_{\frakV,X^\calf,k+1}(\frakm) = \phi_{\frakV,X^\calf,k+1}(V_1)x_{i_1}\phi_{\frakV,X^\calf,k+1}(V_2)\cdots \phi_{\frakV,X^\calf,k+1}(V_{b-1})x_{i_{b-1}} \phi_{\frakV,X^\calf,k+1}(V_b),
\mlabel{eq:motzdec2}
\end{equation}
with
\begin{equation} \phi_{\frakV,X^\calf,k+1}(V_j)=\left \{\begin{array}{ll}
    (\onetree, \bfone), & V_j=\onetree,\\
    \lc \phi_{\frakV,X^\calf,k}(\oV_j) \rc, & V_j=\lm \oV_j\rtm.
    \end{array} \right .
\mlabel{eq:motzdec3}
\end{equation}
Thus by the induction hypothesis, the expression is well-defined and is an angularly decorated forest in its standard decomposition in Eq.~(\mref{eq:stdecm}).

Conversely, for any given angularly decorated rooted forest $D\in X^{\calf_{k+1}}$, let
$$ D= D_1 {x_{i_1}} D_2  \cdots D_{b-1}{x_{i_{b-1}}} D_b$$
be its standard decomposition in Eq.~(\mref{eq:stdecm}). Then define
\begin{equation}
\phi_{X^\calf,\frakV,k+1}(D)=\phi_{X^\calf,\frakV,k+1}(D_1)
\circ
\begin{array}{c}\scriptsize{x_{i_1}}\vspace{-.3cm}\\ \mb2 \end{array}
\circ \phi_{X^\calf,\frakV,k+1}(D_2) \circ \cdots \circ \begin{array}{c}\scriptsize{x_{i_{b-1}}}\vspace{-.3cm}\\ \mb2 \end{array} \circ \phi_{X^\calf,\frakV,k+1}(D_b)
\mlabel{eq:anglepath}
\end{equation}
where on the right hand side
\begin{equation}
 \phi_{X^\calf,\frakV,k+1}(D_j)=\left \{\begin{array}{ll}
    \onetree, & D_j=(\onetree;\bfone), \\
    \lm \phi_{X^\calf,\frakV,k}(\oD_j) \rtm, & D_j=\lc \oD_j \rc.
    \end{array} \right.
\mlabel{eq:anglepath2}
\end{equation}
Here we note that, by the definition of the standard decomposition in Eq.~(\mref{eq:stdecm}), any $D_j$ is either $(\onetree;\bfone)$ or
$(\lc \oT_j\rc; \vec{x}_j)=\lc (\oT_j;\vec{x}_j) \rc = \lc \oD_j\rc,$
where $\oD_j=(\oT_j;\vec{x}_j)$ is in $X^{\calf_k}$.
Thus $\phi_{X^\calf,\frakV,k+1}(D)$ is well-defined by the induction hypothesis.
Then it is immediately checked that $\phi_{X^\calf,\frakV,k+1}: X^{\calf_{k+1}} \to \frakV_{k+1}(X)$ is the inverse of $\phi_{\frakV,X^\calf,k+1}:\frakV_{k+1}(X)\to X^{\calf_{k+1}}$, completing the inductive construction of $\phi_{\frakV,X^\calf,n}$. Then $\phi_{\frakV,X^\calf}:=\dirlim \phi_{\frakV,X^\calf,n}$ is bijective and is compatible with the distinguished operators on $\frakV(X)$ and on $X^\calf$.

Note that
$$\phi_{X^\calf, \frakV}(\tb2\,)=\phi_{X^\calf,\frakV}(\lc \onetree \rc)=
\lm \phi_{X^\calf,\frakV}(\onetree) \rtm= \md31$$
Thus if $D\in X^\calf$ does not have a subtree $\tb2$ then $\phi_{X^\calf, \frakV}(D)$ does not have a subpath $\md31$\!, namely $\phi_{X^\calf, \frakV}(D)$ is peak free.
 This shows that $\phi_{X^\calf,\frakV}$ restricts to a bijection $\phi_{X_0^\calf, \frakL\frakV}$ in Eq.~(\mref{eq:pathatree2}). This completes the proof of the theorem.
\end{proof}

\begin{remark}
\mlabel{rk:anglepath}
{\rm ({\bf A combinatorial description of $\phi_{X^\calf,\frakV}$}.) The $X$-decorated Motzkin path $\phi_{X^\calf,\frakV}(D)$ from an angularly decorated planar rooted tree $D$ can be described as follows.
Let $D=(T;\vec{x})$ with $T$ a planar rooted tree and $\vec{x}$ the vector of angular decorations on $T$. We first list the edges of $T$ in biorder.
The {\bf edge biorder list} of $T$ is defined as follows.
\begin{enumerate}
\item
If $T$ has only one vertex, then there is nothing in the edge biorder list of $T$;
\item
If $T$ has more than one vertices, then the root vertex of $T$ has subtrees $T_1,\cdots,T_k$, $k\geq 1$, listed from left to right. Then the edge biorder list of $T$ is
\begin{itemize}
\item $U$, followed by the edge biorder list of $T_1$, followed by $D$,
\item $\cdots$,
\item $U$, followed by the edge biorder list of $T_k$, followed by $D$.
\end{itemize}
\end{enumerate}
Modifying the ``worm" illustration of the vertex preorder in~\cite[Figure 5-14]{St}, imagine that a worm begins just left of the root of the tree and crawls counterclockwise along the outside of the tree until it returns to the start point. As it crawls along, for each of the edges it passes, it records a $\lm$ if it is crawling away from the root and records a $\rtm$ if it is crawling to the root. Note that this way, each edge is recorded twice, the first by a $\lm$ and the second time by a $\rtm$.
For example, the edge biorder list of the angular decorated tree
$$ \bigdecta $$
(ignoring the decorations for now) is
$$ \lm\,\lm\,\rtm\,\lm\,\lm\,\rtm\,\lm\,\rtm\,\lm\,\rtm\, \rtm\,\lm\,\rtm\,\rtm\,\lm\,\lm\,\lm\,\rtm\,\lm\,\rtm\, \rtm\,\lm\,\lm\,\rtm\,\lm\,\rtm\,\rtm\,\rtm\,.$$
Regarding an edge biorder list as a word $\frakw$ with the alphabet set $\{\lm, \rtm\}$, counting from the left of $\frakw$, the number of occurrences of $\lm$ (resp. $\rtm$) is the number of edges that are encountered for the first (resp. the second) time. Thus $\frakw$ is a Motzkin word by its definition (Definition~\mref{de:motzword}). In fact, it is a Dyck word in the sense that it codes a Dyck path.

Next we deduce a Motzkin word with $L$-letters from this edge biorder list by replacing each occurrence of $\rtm\,\lm\,$ by an $L$. For example, the above edge biorder list is deduced to
$$ \lm\,\lm\,L\lm\,LL\rtm\,L\rtm\,L\lm\,\lm\,L\rtm\,L\lm\, L\rtm\,\rtm\,\rtm\,.$$
Note that each time the ``worm" passes an angles of the tree it records a pair $\rtm\,\lm\,$ in the edge biorder list and hence an $L$ in the corresponding Motzkin word. So the number of angles of the tree $T$ equals the number of $L$-letters in the Motzkin path. Thus we can use the entries of $\vec{x}$ in $D=(T;\vec{x})$ to decorate and replace the $L$-letters, giving rise to an $X$-decorated Motzkin word. For our example above, we have
$$ \lm\,\lm\,a\lm\,bc\rtm\,d\rtm\,e\lm\,\lm\,f\rtm\,g\lm\, h\rtm\,\rtm\,\rtm\,$$
and hence the Motzkin path 
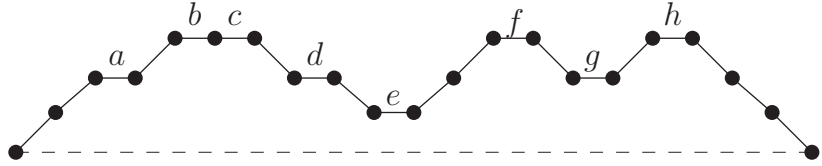
\vspace{0.3cm}
$$
\bigdecl
\vspace{0.5cm}
$$
}
\end{remark}

\begin{remark}
\mlabel{rk:leafangle}
{\rm
({\bf A combinatorial description of $\phi_{\calf_\ell,X_0^\calf}$}.) We give a direct description of the bijection $\phi_{\calf_\ell,X_0^\calf}=\phi_{\frakL\frakV,X_0^\calf}\circ \phi_{\frakL\frakV,\calf_\ell}^{-1}: \calf_\ell(X) \to X^\calf_0$. Let $L\in \calf_\ell(X)$ be a leaf-spaced forest with the leaves decorated by $X$. Construct an angularly decorated forest $D$ from $L$ as follows. For a fixed vertex of $L$, the non-leaf branches, if there is any, of this vertex divide the leaf branches of this vertex into blocks of leaves. Each such a block is one of the following forms.
\begin{enumerate}
\item
 $\bullet_{x_1} \bullet_{x_2} \cdots \bullet_{x_k}$, namely all branches of this vertex are leaf branches;
\item
 $\bullet_{x_1} \bullet_{x_2} \cdots \bullet_{x_k} \Big|$, namely there are no branches to the left of this block, but there is a non-leaf branch, denoted by $\Big|$, to its right;
\item
 $\Big|\, \bullet_{x_1} \bullet_{x_2} \cdots \bullet_{x_k} $, namely there are no branches to the right of this block, but there is a non-leaf branch, denoted by $\Big|$, to its left;
\item
 $\Big|\, \bullet_{x_1} \bullet_{x_2} \cdots \bullet_{x_k} \Big| $, namely the block is bounded from both sides by non-leaf branches.
\end{enumerate}
Then $\phi_{\frakV,X^\calf}$ does not change the non-leaf vertices, but change a leaf block of $L$ in the above four cases as follows
\begin{equation}
\phi_{\frakV,X^\calf}: \left \{
\begin{split}
\bullet_{x_1} \bullet_{x_2} \cdots \bullet_{x_k}
&
\mapsto \bullet\, x_1 \bullet x_2 \bullet \cdots \bullet x_k \bullet; \\
\bullet_{x_1} \bullet_{x_2} \cdots \bullet_{x_k}\, \Big|
& \mapsto \bullet\, x_1 \bullet x_2 \bullet \cdots \bullet x_k \Big |;\\
\Big| \,\bullet_{x_1} \bullet_{x_2} \cdots \bullet_{x_k}
& \mapsto \Big| \, x_1 \bullet x_2 \bullet \cdots \bullet x_k \bullet;\\
\Big|\, \bullet_{x_1} \bullet_{x_2} \cdots \bullet_{x_k}\, \Big|
& \mapsto \Big|\, x_1 \bullet x_2 \bullet \cdots \bullet x_k \Big |
\end{split}
\right .
\mlabel{eq:leafangle}
\end{equation}
So in the first case, an extra leaf sibling is added to create $k$ angles for the $k$ leaf decorations $x_1,\cdots,x_k$. In particular, if $L$ has a ladder subtree, then the corresponding leaf is the only (leaf or non-leaf) child of its parent, then this leaf must be split to a fork under $\phi_{\frakV,X^\calf}$. Thus $\phi_{\frakV,X^\calf}(L)$ is ladder-free.
In the second (resp. third) case, just shift the leaf decorations to the right (resp. left) angles.
In the fourth case, remove one of the leaf siblings to get the right number of angles for these $k$ leaf decorations. Note that, by the definition of $\calf_\ell(X)$, there must be a leaf between two non-leaf branches of a vertex.
Thus every angle of $\phi_{\frakV,X^\calf}(L)$ is decorated. We illustrate this by revisiting the example in Eq.~(\mref{eq:bigdectls}) and get

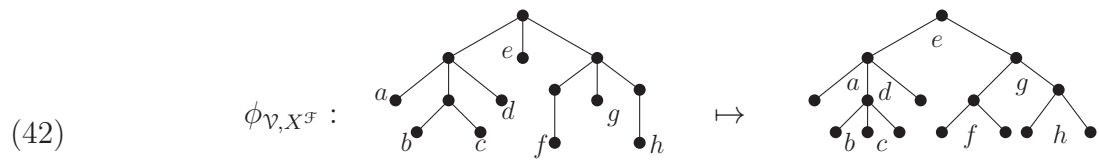
\begin{equation}
\begin{array}{c} \phi_{\frakV,X^\calf}: \\ \\ \end{array}
\qquad \quad \bigdectls \qquad
 \begin{array}{c}\mapsto \\
 \\ \end{array}  \qquad \bigdecta
\mlabel{eq:bigdecta}
\end{equation}
Note that the two ladder subtrees with leaves $g$ and $h$ are transformed to non-ladder subtrees.

To describe $\phi_{\frakV,X^\calf}^{-1}$, for each vertex in an angularly decorated tree $D=(T;\vec{x})\in X_0^\calf$, we similarly use the non-leaf branches of this vertex to divide the angles of this vertex into blocks, each in one of the four forms on the right hands of the correspondence in
Eq.~(\mref{eq:leafangle}). Then follow Eq.~(\mref{eq:leafangle}) backwards. Note that if two non-leaf branches of a vertex in an $F\in X^\calf_0$ has no leaf in between, then we are in the fourth case with $k=1$. Then by the rule, a new leaf vertex must be inserted with the same decoration that decorated the angle between the two branches. Therefore $\phi_{\frakV,X^\calf}^{-1}(F)$ must be leaf-spaced. This is the case for the angle decorated by $e$ and $g$ in Eq.~(\mref{eq:bigdecta}).
}
\end{remark}

\section{Constructions of free Rota-Baxter algebras}
\mlabel{sec:rba}

In \mcite{E-G0} (see also \mcite{A-M}), free Rota-Baxter algebras on a set $X$ is constructed using angularly decorated planar forests. We first recall this construction and then show how,
through the natural bijections in Theorem~\mref{thm:diag}, we can transport this Rota-Baxter algebra structure on angularly decorated forests to such a structure on Motzkin paths, leaf decorated planar forests and bracketed words.
\subsection{Review of the Rota-Baxter algebra structure on angularly decorated forests}
\mlabel{ss:rbaangle}
On the set of angularly decorated forests $X^\calf$ defined in Section~\mref{ss:adf}, define the free $\bfk$-module
$ \bfk\, X^\calf$ (denoted by $\ncsha(X)$ in~\mcite{E-G0}).
Note that if $D=(F;\vec{x})\in X^\calf$ is a tree (that is, if $F$ is a tree), then since $F$ is either $\onetree$ or $\lc \oF\rc$ for $\oF\in \calf$, we have either $D=(\onetree;\bfone)$ or $D=\lc \oD\rc$ where $\oD=(\oF;\vec{x})$.
The depth filtration on $\calf$ in Eq.~(\mref{eq:treex}) induces a depth filtration on $X^\calf$.
In~\mcite{E-G0}, we use this filtration to define a multiplication $\shprm$ on $\bfk\, X^\calf$ that is characterized by the following properties.
\begin{enumerate}
\item
$(\onetree;\bfone)$ is the multiplication identity;
\mlabel{it:angle1}
\item
If $D$ and $D'$ are angularly decorated trees not equal to $\onetree$, so $D=\lc \oD\rc, D'=\lc \oD'\rc$ for $\oD, \oD'\in X^\calf$,
then
\begin{equation}
D \shprm D' =
    \lc \oD\shprm D' \rc
    +\lc D \shprm \oD'\rc
    +\lambda\lc \oD \shprm \oD'\rc.
\mlabel{eq:shprt1}
\end{equation}
\mlabel{it:angle2}
\item
If $D=D_1x_{i_1}\cdots x_{i_{b-1}} D_b$ and
$D'=D'_1x'_{i'_1} \cdots x'_{i'_{b'-1}} D'_{b'}$ are the standard decomposition of the angularly decorated forests $D$ and $D'$ in Eq.~(\mref{eq:stdecm}), then
\begin{equation}
D \shprm D'= D_1x_{i_1}\cdots  D_{b-1}\,
x_{i_{b-1}}\,(D_b\shprm D'_1)\,x'_{i'_1} \,
    D'_{2}\,\cdots\,x'_{i'_{b'-1}} D_{b'}.
\mlabel{eq:shprt2}
\end{equation}
\mlabel{it:angle3}
\end{enumerate}
For example, applying Eq.~(\mref{eq:shprt1}) and Eq.~(\mref{eq:shprt2}) we have
\begin{eqnarray}
\xtd31 \shprm \ytd31 &=& \lc \onetree x \onetree\rc\ \shprm\ \lc \onetree y \onetree \rc \notag\\
&=& \lc (\onetree x \onetree)\ \shprm\ \ytd31 \rc
+ \lc \xtd31\ \shprm\ (\onetree y \onetree) \rc
+ \lambda \lc (\onetree x \onetree) \shprm (\onetree y \onetree) \rc \notag\\
&=& \lc \onetree x \ytd31 \rc + \lc \xtd31 y \onetree\rc
+ \lambda \lc \onetree x \onetree y \onetree \rc
\mlabel{eq:atreeex1}\\
&=& \xyrlong + \xyllong + \lambda\, \xyldec43
\notag
\bigskip
\end{eqnarray}
Similarly,
\begin{equation}
 \xtd31  \shprm \ \tb2
= \begin{array}{l}\\[-.7cm] \xldec41r \end{array}
+ \begin{array}{l}\\[-.7cm] \xthj44 \end{array}
+ \lambda  \xtd31 .
\mlabel{eq:atreeex2}
\end{equation}
Extending the product $\shprm$ bilinearly, we obtain
a binary operation
$$
\shprm:  \bfk\,X^\calf \otimes \bfk\,X^\calf \to \bfk\,X^\calf.
$$
For $(F;\vec{x}) \in X^F$, define
\begin{equation}
 P_X(F; \vec{x})=\lc (F;\vec{x})\rc=(\lc F \rc\, ; \vec{x})\in X^{\lc F\rc},
 \mlabel{eq:RBopm}
\end{equation}
extending to a linear operator $P_X$ on $\bfk\, X^\calf$. Let
\begin{equation}
  j_X: X \to \bfk\, X^\calf
  \mlabel{eq:jm}
\end{equation}
be the map sending $x\in X$ to $(\onetree \ \onetree;x)$. The following theorem is proved in~\mcite{E-G0}.

\begin{theorem}
    The quadruple $(\bfk\,X^\calf,\shprm,P_X,j_X)$ is the free Rota--Baxter algebra of
    weight $\lambda$ on the set $X$. More precisely, for any Rota--Baxter algebra $(R,P)$ and map $f:X\to R$, there is a unique Rota--Baxter algebra homomorphism
    $\free{f}: \bfk\,X^\calf \to R$ such that $f=\free{f}\circ j_X.$
    Similarly, The quadruple $(\bfk\,X^\calf_0,\shprm,P_X,j_X)$ is the free nonunitary Rota--Baxter algebra of
    weight $\lambda$ on the set $X$. Here $X^\calf_0$ is the set of ladder-free angularly decorated forests in Theorem~\mref{thm:diag}.
\mlabel{thm:freem}
\end{theorem}

\subsection{Rota-Baxter algebra structure on Motzkin paths}
We now transport the Rota-Baxter algebra structure from $\bfk X^\calf$ to $\bfk \frakV(X)$ through the bijection $\phi_{X^\calf, \frakV}$ in Eq.~(\mref{eq:anglepath}) and its inverse $\phi_{\frakV,X^\calf}$ in Eq.~(\mref{eq:motzdec2}).
Note that an indecomposable $X$-decorated Motzkin path is either $\onetree$ or $\mb2$ or $\lm \ofrakm \rtm$ for another $X$-decorated Motzkin path $\ofrakm$.

\begin{theorem}
The bijection $\phi_{X^\calf,\frakV}: X^\calf \to \frakV$ extends to an isomorphism
$$ \phi_{X^\calf,\frakV}: (\bfk X^\calf, \shprm, P_X)
\to (\bfk \frakV, \shprp\,, \lc\ \rc)$$
of Rota-Baxter algebras where the multiplication $\shprp$ on $\bfk \frakV$ is defined recursively with respect to the height of Motzkin paths and is characterized by the following properties.
\begin{enumerate}
\item The trivial path $\onetree$ is the multiplication identity;
\mlabel{it:path1}
\item If $\frakm$ and $\frakm'$ are indecomposable $X$-decorated Motzkin paths not equal to $\onetree$, then
\begin{equation}
\frakm \shprp \frakm' = \left \{ \begin{array}{ll}
 \frakm \circ \frakm', & \frakm=\xmb2 \mbox{ or } \frakm'=\xpmb2, \\
 \lm \ofrakm \shprp \frakm'\rtm + \lm \frakm \shpr \ofrakm' \rtm
 + \lambda \lm \ofrakm \shprp \ofrakm' \rtm ,
 & \frakm= \lm \ofrakm \rtm, \frakm'=\lm \ofrakm' \rtm;
\end{array}
\right .
\mlabel{eq:mshpr2}
\end{equation}
\mlabel{it:path2}
\item
If $\frakm=\frakm_1\circ \cdots \circ \frakm_p$ and $\frakm'=\frakm'_1\circ \cdots \circ \frakm'_{p'}$ are the decompositions of $\frakm, \frakm' \in \frakV(X)$ into indecomposable paths, then
\begin{equation}
\frakm\shprp \frakm' = \frakm_1\circ \cdots \circ (\frakm_p \shprp \frakm'_1) \circ \cdots \circ \frakm_{p'}.
\mlabel{eq:mshpr1}
\end{equation}
\mlabel{it:path3}
\end{enumerate}
Further, the Rota-Baxter algebra isomorphism $\phi_{X^\calf,\frakV}$ restricts to an isomorphism
$$\phi_{X_0^\calf,\frakL\frakV}: \bfk X_0^\calf \to \bfk (\frakL\cap\frakV)$$
of nonunitary Rota-Baxter algebras.
\mlabel{thm:anglepath}
\end{theorem}
As an illustration, the example in Eq.~(\mref{eq:atreeex1}) corresponds to
\begin{eqnarray}
\xmh43 \shprp \ymh43 &=& \lm \xmb2 \rtm \shprp \lm \ymb2 \rtm \notag \\
&=& \lm \xmb2 \shprp \ymh43 \rtm
    + \lm \xmh43 \shprp  \ymb2 \rtm
    + \lambda \lm \xmb2 \shprp \ymb2 \rtm
    \notag \\
&=& \lm \xymn55 \rtm + \lm \xymm54 \rtm
    + \lambda \lm \xymc3 \rtm
\mlabel{eq:pathprodex1}\\
&=& \uxuydd + \uuxdyd + \lambda \xymo56
\notag
\end{eqnarray}

\begin{proof}
We just need to show that under the bijection $\phi_{X^\calf,\frakV}$, the product $\shprm$ on $\bfk X^\calf$ characterized by (\mref{it:angle1}) -- (\mref{it:angle3}) in \S\,\mref{ss:rbaangle} corresponds to the product $\shprp$ characterized by (\mref{it:path1}) -- (\mref{it:path3})
in Theorem~\mref{thm:anglepath}.

First of all, since $\phi_{X^\calf,\frakV}(\onetree)=\onetree\in \frakV(X)$, $\onetree$ is the identity for the multiplication $\shprp$.
Next let $\frakm$ and $\frakm'$ be indecomposable $X$-decorated Motzkin paths that are not $\onetree$.
Then either $\frakm=\xmb2$ or $\frakm=\lm \ofrakm\rtm$. Similarly for $\frakm'$. Then by Eq.~(\mref{eq:pathatree4}) and Eq.~(\mref{eq:motzdec3}), we have
$$ \phi_{\frakV,X^\calf}(\frakm)=\left \{
\begin{array}{ll}
\onetree x \onetree, & \frakm= \xmb2, \\
\lc \phi_{\frakV,X^\calf}(\ofrakm) \rc, & \frakm=\lc \ofrakm \rc.
\end{array} \right .
\quad
\phi_{\frakV,X^\calf}(\frakm')=\left \{
\begin{array}{ll}
\onetree x' \onetree, & \frakm= \xpmb2, \\
\lc \phi_{\frakV,X^\calf}(\ofrakm') \rc, & \frakm'=\lc \ofrakm' \rc.
\end{array} \right .
$$
Denote $$D=\phi_{\frakV,X^\calf}(\frakm), D'=\phi_{\frakV,X^\calf}(\frakm'), \oD=\phi_{\frakV,X^\calf}(\ofrakm), \oD'=\phi_{\frakV,X^\calf}(\ofrakm').$$
It then follows from Eq.~(\mref{eq:shprt1}) and (\mref{eq:shprt2}) that
$$
\phi_{\frakV,X^\calf}(\frakm) \shprm \phi_{\frakV,X^\calf}(\frakm')
=\left \{ \begin{array}{ll}
\onetree x \onetree x' \onetree, & \frakm=\xmb2, \frakm'=\xpmb2,\\
\onetree x \lc \oD'\rc, & \frakm=\xmb2, \frakm'=\lm \ofrakm' \rtm, \\
\lc \oD \rc x' \onetree, & \frakm=\lm \ofrakm \rtm, \frakm'=\xpmb2, \\
\lc \oD \shprm D' \rc +\lc D \shprm \oD'\rc
+\lambda \lc \oD \shprm \oD' \rc, &
\frakm=\lm \ofrakm \rtm, \frakm'=\lm \ofrakm' \rtm.
\end{array} \right .
$$
By definition, the product $\shprp$ on $\bfk \frakV$
obtained from the product $\shprm$ on $\bfk X^\calf$ through the isomorphism $\phi_{X^\calf, \frakV}$ is
$$\frakm \shprp \frakm'= \phi_{X^\calf, \frakV}
(\phi_{\frakV, X^\calf}(\frakm) \shprm
\phi_{\frakV, X^\calf}(\frakm')).$$
Thus by Eq.~(\mref{eq:anglepath}) and (\mref{eq:anglepath2}) and the fact that $\phi_{\frakV,X^\calf}$ and $\phi_{X^\calf,\frakV}$ are the inverse of each and preserve the distinguished operators, we have
$$
\frakm\shprp \frakm'=\left \{ \begin{array}{ll}
\xmb2 \circ \xpmb2 = \frakm \circ \frakm', & \frakm=\xmb2, \frakm'=\xpmb2,\\
\xmb2 \circ \lm \ofrakm'\rtm = \frakm \circ \frakm', & \frakm=\xmb2, \frakm'=\lm \ofrakm' \rtm, \\
\lm \ofrakm \rtm \circ \xpmb2 = \frakm \circ \frakm', & \frakm=\lm \ofrakm \rtm, \frakm'=\xpmb2, \\
\lm \ofrakm \shprp \frakm' \rtm +\lm \frakm \shprp \ofrakm'\rtm
+\lambda \lm \ofrakm \shprp \ofrakm' \rtm, &
\frakm=\lm \ofrakm \rtm, \frakm'=\lm \ofrakm' \rtm.
\end{array} \right .
$$
This is Eq.~(\mref{eq:mshpr2}).

If $\frakm=\frakm_1\circ \cdots \circ \frakm_p$ and $\frakm'=\frakm'_1\circ \cdots \circ \frakm'_{p'}$ are the decompositions of $\frakm, \frakm' \in \frakV(X)$ into indecomposable paths. Let
$\frakm=V_1 \circ \begin{array}{c}\scriptsize{x_{i_1}}\vspace{-.3cm}\\ \mb2 \end{array}
\circ V_2 \circ \cdots \circ \begin{array}{c}\scriptsize{x_{i_{b-1}}}\vspace{-.3cm}\\ \mb2 \end{array} \circ V_b
$
and
$\frakm=V'_1 \circ \begin{array}{c}\scriptsize{x'_{i'_1}}\vspace{-.3cm}\\ \mb2 \end{array}
\circ V'_2 \circ \cdots \circ \begin{array}{c}\scriptsize{x'_{i'_{b'-1}}}\vspace{-.3cm}\\ \mb2 \end{array} \circ V'_{b'}
$
be their standard decompositions as in Eq.~(\mref{eq:motzdec1}). Then by Equations (\mref{eq:motzdec2}), (\mref{eq:anglepath}) and (\mref{eq:shprt2}), we have
\begin{equation}
\frakm\shprp \frakm' =
V_1 \circ \begin{array}{c}\scriptsize{x_{i_1}}\vspace{-.3cm}\\ \mb2 \end{array}
\circ V_2 \circ \cdots \circ \begin{array}{c}\scriptsize{x_{i_{b-1}}}\vspace{-.3cm}\\ \mb2 \end{array} \circ (V_b \shprp V'_1) \circ \begin{array}{c}\scriptsize{x'_{i'_1}}\vspace{-.3cm}\\ \mb2 \end{array}
\circ V'_2 \circ \cdots \circ \begin{array}{c}\scriptsize{x'_{i'_{b'-1}}}\vspace{-.3cm}\\ \mb2 \end{array} \circ V'_{b'}.
\mlabel{eq:anglepath3}
\end{equation}
By the definition of the standard decomposition
$\frakm=V_1 \circ \begin{array}{c}\scriptsize{x_{i_1}}\vspace{-.3cm}\\ \mb2 \end{array}
\circ V_2 \circ \cdots \circ \begin{array}{c}\scriptsize{x_{i_{b-1}}}\vspace{-.3cm}\\ \mb2 \end{array} \circ V_b
$
of $\frakm$, we see that if $\frakm_p=\xmb2$, then $V_b=\onetree$ and hence $V_1 \circ \begin{array}{c}\scriptsize{x_{i_1}}\vspace{-.3cm}\\ \mb2 \end{array}
\circ V_2 \circ \cdots \circ \begin{array}{c}\scriptsize{x_{i_{b-1}}}\vspace{-.3cm}\\ \mb2 \end{array} = \frakm$.
If $\frakm_p\neq \xmb2$, then $V_b=\frakm_p$ and hence
$V_1 \circ \begin{array}{c}\scriptsize{x_{i_1}}\vspace{-.3cm}\\ \mb2 \end{array}
\circ V_2 \circ \cdots \circ \begin{array}{c}\scriptsize{x_{i_{b-1}}}\vspace{-.3cm}\\ \mb2 \end{array} = \frakm_1\circ \cdots \circ \frakm_{p-1}.
$
Similarly, if $\frakm'_1=\xpmb2$, then $V'_1=\onetree$ and $\begin{array}{c}\scriptsize{x'_{i'_1}}\vspace{-.3cm}\\ \mb2 \end{array}
\circ V'_2 \cdots \circ \begin{array}{c}\scriptsize{x'_{i'_{b'-1}}}\vspace{-.3cm}\\ \mb2 \end{array} \circ V'_{b'}=\frakm'$.
If $\frakm'_1\neq \xpmb2$, then $V'_1=\frakm'_1$ and
$\begin{array}{c}\scriptsize{x'_{i'_1}}\vspace{-.3cm}\\ \mb2 \end{array}
\circ V'_2 \cdots \circ \begin{array}{c}\scriptsize{x'_{i'_{b'-1}}}\vspace{-.3cm}\\ \mb2 \end{array} \circ V'_{b'}=\frakm'_2\circ \cdots \circ \frakm'_{p'}$.
Then we see that in all cases of $\frakm_p$ and $\frakm'_1$, Eq.~(\mref{eq:anglepath3}) agrees with Eq.~(\mref{eq:mshpr1}).
This completes the proof that $\phi_{X^\calf,\frakV}$ in Eq.~(\mref{eq:anglepath}) is a Rota-Baxter algebra isomorphism.

Since $\phi_{X^\calf,\frakV}$ restricts to a bijection
$\phi_{X_0^\calf,\frakL\frakV}: X_0^\calf\to \frakL\cap \frakV$ by Theorem~\mref{thm:diag} (see Eq.~(\mref{eq:pathatree2})), the second part of the theorem follows.
\end{proof}

Furthermore the map $j_X: X\to X^\calf$ in Eq.~(\mref{eq:jm}) is translated to
\begin{equation}
j_X: X \to \frakP(X), \quad j_X(x)=\xmb2.
\mlabel{eq:jmm}
\end{equation}
Then by Theorem~\mref{thm:freem} and Theorem~\mref{thm:anglepath} we have
\begin{coro}
The quadruple $(\bfk\frakV(X), \shprp, \lm\ \rtm, j_X)$ {\rm (}resp. $(\bfk\, (\frakV(X)\cap \frakL(X)), \shprp, \lm\ \rtm, j_X)${\rm )} is the free Rota-Baxter algebra (resp. free nonunitary Rota-Baxter algebra) on  $X$.
\mlabel{co:motzfree}
\end{coro}

\subsection{Rota-Baxter algebra structure on bracketed words}
Through the natural bijection
$\phi_{\frakR,\frakV}: \frakR(X) \to \frakV(X)$ in Theorem~\mref{thm:diag} which is the restriction of
$\phi_{\frakM,\frakP}:\frakM(X) \to \frakP(X)$ in Eq.~(\mref{eq:freeMP}), the Rota-Baxter algebra structure on $\bfk \frakV(X)$ in Theorem~\mref{thm:anglepath} is transported to a Rota-Baxter algebra structure on $\bfk \frakR(X)$, giving another construction of the free Rota-Baxter algebra on $X$, in terms of bracketed words. See also~\mcite{E-G4,G-S} for variations of this construction.

A bracketed word $W\in \frakM(X)$ is called {\bf indecomposable} if either $W=\bfone$ or $W=x\in X$ or
$W=\lc \oW\rc$ where $\oW$ is another bracketed word. Clearly, $W\in \frakM(X)$ is indecomposable if and only if the Motzkin path $\phi_{\frakM,\frakP}(W)\in \frakP(X)$ is indecomposable. It then follows that any bracketed word has a unique decomposition as a product of indecomposable bracketed words.

Since $\phi_{\frakM,\frakP}$ is an isomorphism of \mapped monoids, its restriction $\phi_{\frakR,\frakV}$ and its inverse $\phi_{\frakV,\frakR}$ are compatible with the multiplications (link product on paths and concatenation product on words) and the distinguished operators (the raising and bracketing operators). Then
we transport the product $\shprp$ on $\bfk X^{\calf}$, characterized by Eq.~(\mref{eq:mshpr2}) and Eq.~(\mref{eq:mshpr1}), to a product $\shprw$ on $\bfk \frakR(X)$, characterized by the following properties.
\begin{enumerate}
\item $\bfone$ is the multiplication identity.
\item If $W$ and $W'$ are indecomposable words in $\frakR(X)$ not equal to $\bfone$, then
\begin{equation}
W \shprw W' = \left \{ \begin{array}{ll}
 W\, W' \mbox{ (word concatenation)}, & W=x \mbox{ or } W'=x', x,x'\in X, \\
 \lc \oW \shprw W'\rc + \lc W \shprw \oW' \rc
 + \lambda \lc \oW \shprw \oW' \rc ,
 & W= \lc \oW \rc, W'=\lc \oW' \rc.
\end{array}
\right .
\mlabel{eq:wshpr2}
\end{equation}
\item
If $W=W_1 \cdots  W_b$ and $W'=W'_1 \cdots  W'_{b'}$ are the decompositions of $W, W'\in \frakR(X)$ into indecomposable words, then
\begin{equation}
W\shprw W' = W_1 \cdots  (W_b \shprw W'_1)  \cdots W'_{b'}.
\mlabel{eq:wshpr1}
\end{equation}
\end{enumerate}

The example in Eq.~(\mref{eq:pathprodex1}) corresponds to
\begin{equation}
\lc x \rc \shprw \lc y \rc
= \lc x \shprw \lc y \rc \rc
    + \lc \lc x \rc \shprw  y \rc
    + \lambda \lc x \shprw y \rc
= \lc x \lc y \rc \rc
    + \lc \lc x \rc   y \rc
    + \lambda \lc x  y \rc
    \mlabel{eq:wordprodex1}
\end{equation}
Similarly,
\begin{equation}
\lc x \rc \shprw \lc \bfone \rc
= \lc x \shprw \lc \bfone \rc \rc
    + \lc \lc x \rc \shprw  \bfone \rc
    + \lambda \lc x \shprw \bfone \rc
= \lc x \lc \bfone \rc \rc
    + \lc \lc x \rc \rc
    + \lambda \lc x  \rc
   \mlabel{eq:wordprodex2}
\end{equation}

Further the map $j_X: X\to \frakR(X)$ in Eq.~(\mref{eq:jmm}) is transported to
\begin{equation}
j_X: X \to \frakS(X)\cap \frakR(X)\subseteq \frakR(X), \quad j_X(x)=x.
\mlabel{eq:jw}
\end{equation}
Then by Theorem~\mref{thm:anglepath} and Corollary~\mref{co:motzfree} we have
\begin{coro}
\begin{enumerate}
\item
The bijection $\phi_{\frakR,\frakV}: \frakR(X)\to \frakV(X)$ extends to an isomorphism
$\phi_{\frakR,\frakV}: (\bfk \frakR(X), \shprw, \lc\ \rc)\to (\bfk \frakV(X), \shprp, \lm\ \rtm)$ of Rota-Baxter algebras. This isomorphism restricts to
an isomorphism
$\phi_{\frakS\frakR,\frakL\frakV}: (\bfk (\frakS(X) \cap \frakR(X)), \shprw, \lc\ \rc)\to (\bfk (\frakL(X)\cap \frakV(X)), \shprp, \lm\ \rtm)$ of nonunitary Rota-Baxter algebras.
\item
The quadruple $(\bfk\frakR(X), \shprw, \lc\ \rc, j_X)$ {\rm (}resp. $(\bfk\,(\frakS(X)\cap\frakR(X)), \shprw, \lc\ \rc, j_X)${\rm )} is the free Rota-Baxter algebra {\rm (}resp. free nonunitary Rota-Baxter algebra{\rm )} on  $X$.
\end{enumerate}
\mlabel{co:rbwordfree}
\end{coro}

\subsection{Rota-Baxter algebra structure on leaf decorated rooted forests}
We finish this paper by obtaining a free Rota-Baxter algebra structure on the free $\bfk$-module $\bfk \calf_\ell(X)$ of leaf decorated rooted forests.

Through the bijection $\phi_{\frakL\frakV,\calf_\ell}:\frakL(X)\cap\frakV(X) \to \calf_\ell(X)$ in Theorem~\mref{thm:diag}, the free nonunitary Rota-Baxter algebra on $\bfk (\frakL(X)\cap\frakV(X))$ in Corollary~\mref{co:motzfree}
gives us a free nonunitary Rota-Baxter algebra structure on $\bfk\, \calf_\ell(X)$.
Since $\phi_{\frakL\frakV,\calf_\ell}$ sends the link product of paths to the concatenation of forests and sends the raising operator to the grafting operator,
the two properties in Eq.~(\mref{eq:mshpr2}) and Eq.~(\mref{eq:mshpr1}) translate to the following properties characterizing the multiplication $\shprl$ on leaf-spaced leaf decorated forests.
\begin{enumerate}
\item If $F$ and $F'$ are leaf decorated trees, then
\begin{equation}
F \shprl F' = \left \{ \begin{array}{ll}
 F\,  F' \mbox{ (concatenation of trees)}, & F=\bullet_x \mbox{ or } F'=\bullet_x', \\
 \lc \oF \shprl F'\rc + \lc F \shprl \oF' \rc
 + \lambda \lc \oF \shprl \oF' \rc ,
 & F= \lc \oF \rc, F'=\lc \oF' \rc.
\end{array}
\right .
\mlabel{eq:lshpr2}
\end{equation}
Here the second line makes sense since a leaf decorated tree is either of the form $\onetree_x$ for some $x\in X$, or is of the form $\lc \oF\rc$ where $\oF$ is the leaf decorated forest obtained from $F$ by removing its root. In other words, $\oF$ is the forest of the branches of the root of $F$, with the same leaf decoration as for $F$.
\item
If $F=F_1\, \cdots \, F_b$ and $F'=F'_1\, \cdots \, F'_{b'}$ are in $\calf_\ell(X)$ with their corresponding decomposition into leaf decorated trees, then
\begin{equation}
F\shprl F' = F_1\, \cdots \, (F_b \shprl F'_1) \, \cdots \, F_{b'}.
\mlabel{eq:lshpr1}
\end{equation}
\end{enumerate}
The example in Eq.~(\mref{eq:pathprodex1}) translates to
\begin{eqnarray}
\xlb2 \shprl \ylb2 &=& \lc\, \bullet_x\, \rc \shprl \lc \, \bullet_y\,  \rc \notag \\
&=& \lc \, \bullet_x\,  \shprl \ylb2 \rc
    + \lc \xlb2 \shprl  \, \bullet_y\,  \rc
    + \lambda \lc \, \bullet_x\,  \shprl \, \bullet_y\,  \rc
    \notag \\
&=& \lc \, \bullet_x\,  \ylb2 \rc + \lc \xlb2\ \, \bullet_y\,  \rc
    + \lambda \lc \, \bullet_x\, \ \, \bullet_y\,  \rc
\mlabel{eq:ltreeprodex1}\\
&=& \xylg42 + \xylf41 + \lambda \xyld31
\notag
\bigskip
\end{eqnarray}

Further the map $j_X: X\to \frakL(X)\cap \frakV(X)$ in Eq.~(\mref{eq:jmm}) is transported to
\begin{equation}
j_X: X \to \calf_\ell(X), \quad j_X(x)=\onetree_x, x\in X.
\mlabel{eq:jlt}
\end{equation}
Then by Theorem~\mref{thm:anglepath} and Corollary~\mref{co:motzfree} we have
\begin{coro}
\begin{enumerate}
\item
The bijection $\phi_{\frakL\frakV,\calf_\ell}: \frakL(X)\cap \frakV(X) \to \calf_\ell(X)$ extends to an isomorphism
$\phi_{\frakL\frakV,\calf_\ell}: (\bfk (\frakL(X)\cap \frakV(X)), \shprp, \lm\ \rtm)\to (\bfk \calf_\ell(X), \shprl, \lc\ \rc)$ of nonunitary Rota-Baxter algebras. \item
The quadruple $(\bfk \calf_\ell(X), \shprl, \lc\ \rc, j_X)$ is the free nonunitary Rota-Baxter algebra on  $X$.
\end{enumerate}
\mlabel{co:ltreefree}
\end{coro}

%%%%%%%%%%%%%%%%%%%%%%%%%%%%%%%%%%%%%%%%%%%%%%%%%%%%%%%%%%%%%%%%%%%%%%%%%%%%%

%
%\addcontentsline{toc}{section}{\numberline {}References}
%


\begin{thebibliography}{abcdsfgh}


\mbibitem{Ag3} M. Aguiar, {On the associative analog of Lie bialgebras},
               {\em J. of Algebra} {\bf{244}} (2001), 492-532.
\mbibitem{A-M} M. Aguiar and W. Moreira, {Combinatorics of the free Baxter algebra,}
    {\em Electron. J. Combin.} {\bf 13} (2006) R17. arXiv:math.CO/0510169

\mbibitem{Al} L. Alonso, Uniform generation of a Motzkin word, {\em Theoretical Computer Science} {\bf 134} (1994), 529-536.

\mbibitem{Ba} G. Baxter, {An analytic problem whose solution follows from a simple algebraic identity,}
   {\em Pacific J. Math.}, {\bf 10} (1960), 731-742.

\mbibitem{BM} S. Benchekroun and P. Moszkowski,
A new bijection between ordered trees and legal bracketings
{\em Europ. J. Combinatorics} {\bf 17} (1996), 605 – 611.

\mbibitem{Be} T. Bertrand, Solution d'un proble\`eme,
    {\em C. R. Acad. Sci. Paris}, {\bf 105} (1887), 369.

\mbibitem{Ca} P. Cartier, {On the structure of free Baxter algebras}, {\em Adv. in Math.}, {\bf{9}} (1972), 253-265.

\mbibitem{CSY} W. Y. C. Chen, L. W. Shapiro, L. L. M. Yang, Parity reversing involutions on plane trees and 2-Motzkin paths, {\em Europ. J. Combinatorics} {\bf 27} (2006), 283-289.

\bibitem{RCo} R.M. Cohn, ``Difference Algebra",
    Interscience Publishers, 1965.

\mbibitem{C-K0} A. Connes, D. Kreimer,
{Hopf algebras, renormalization and noncommutative geometry},
 {\em Comm. Math. Phys.} {\bf 199} (1998), 203-242.

\mbibitem{C-K1} A. Connes and D. Kreimer, { Renormalization in quantum field theory and
             the Riemann-Hilbert problem. I. The Hopf algebra structure of graphs
             and the main theorem.},
             {\em Comm. Math. Phys.}, {\bf 210} (2000), no. 1, 249-273.

\mbibitem{Di} R. Diestel, "Graph Theory", Third edition, Springer-Verlag, 2005. Available on-line:
http://www.math.uni-hamburg.de/home/diestel/books/graph.theory/download.html

\mbibitem{D-S} E Deutsch and L. W. Shapiro, A bijection between ordered trees and 2-Motzkin paths adn its many consequences, {\em Discrete Math.} {\bf 256} (2002), 655-670.

\mbibitem{Do-S} R. Donaghey and L. W. Shapiro, Motzkin numbers, {\em J. Combin. Theory Ser. A} {\bf 23} (1977), 291-301.

\mbibitem{EGP} K. Bbrahimi-Fard, J. M. Gracia-Bondia and F. Patras, A Lie theoretic approach to renormalization,
arXiv:hep-th/0609035.

\mbibitem{E-G4} K. Ebrahimi-Fard and L. Guo, {Rota--Baxter
    algebras and dendriform dialgebras}, arXiv: math.RA/0503647.

\mbibitem{E-G0} K. Ebrahimi-Fard and L. Guo, {Free Rota--Baxter algebras and rooted trees}, arXiv:math.RA/0510266.

\mbibitem{E-G-K2} K. Ebrahimi-Fard, L. Guo and D. Kreimer, { Integrable Renormalization II:
           the General case}, {\em Annales Henri Poincare} {\bf 6} (2005), 369-395.

\mbibitem{E-G-K3} K. Ebrahimi-Fard, L. Guo and D. Kreimer, { Spitzer's Identity and the Algebraic
          Birkhoff Decomposition in pQFT},
         {\em J. Phys. A: Math. Gen.}, {\bf 37} (2004), 11037-11052.

\mbibitem{E-G-M} K. Ebrahimi-Fard, L. Guo and Dominique Manchon, Birkhoff type decompositions and the Baker-Campbell-Hausdorff
recursion, {\em Comm. in Math. Phys.} {\bf 267} (2006) 821-845, arXiv: math-ph/0602004.

\mbibitem{E-N} K.-J. Engel and R. Nagel,
One-parameter semigroups for linear evolution equations, Graduate Texts in Mathematics, {\bf 194}, Springer-Verlag, New York, 2000.

\mbibitem{Fl} P. Flajolet,
 Mathematical methods in the analysis of algorithms and data
 structures.
 Trends in theoretical computer science (Udine, 1984),
 Principles Comput. Sci. Ser., 12, Computer Sci. Press, Rockville, MD, 1988, 225--304.

\mbibitem{Gr} R. A. Grillet, Commutative Semigroups, Springer, 2006.

\mbibitem{G-L} R. Grossman and R. G. Larson,
    Hopf-algebraic structures of families of trees,
    {\em J. Alg.} {\bf 26} (1989), 184-210.

\mbibitem{Gu2} L. Guo,
    { Baxter algebras and the umbral calculus,}
    {\em Adv. in Appl. Math.,} {\bf 27} (2001), 405-426.

\mbibitem{Gu5} L. Guo, { Baxter algebras, Stirling numbers and partitions}, {\em J. Algebra Appl.}, {\bf 4} (2005), 153-164.

\mbibitem{G-K1} L. Guo and W. Keigher, {Baxter algebras and shuffle products}, {\em Adv. Math.}, {\bf 150} (2000), 117-149.

\mbibitem{G-K2} L. Guo and W. Keigher, { On free Baxter algebras:
    completions and the internal construction,}
    {\em Adv. Math.} {\bf 151} (2000), 101--127.

\mbibitem{G-K3} L. Guo and W. Keigher, {On differential Rota-Baxter algebras}, arXiv: math.RA/0703780.

\mbibitem{G-S} L. Guo and W. Yu Sit, {Enumenation of Rota-Baxter words}, to appear in Proceedings ISSAC 2006, Genoa, Italy, ACM Press, arXiv: math.RA/0602449.

\mbibitem{G-Z} L. Guo and B. Zhang, {Renormalization of multiple zeta values}, arXiv:math.NT/0606076.

\mbibitem{HLV} K. H. Hofmann, J. D. Lawson and E. B. Vinberg, Semigroups in Algebra, Geometry and Analysis,
Walter de Gruyter, 1995.

\bibitem{Kol} E. Kolchin, {``Differential Algebra and
    Algebraic Groups.''} Academic Press, New York, 1973.

\mbibitem{Kr1} D. Kreimer, {On the Hopf algebra structure of perturbative quantum field theories},
{\em Adv. Theor. Math. Phys.}, {\bf{2}} (1998), 303-334.

\mbibitem{Kr} G. Kreweras, Sur les eventails de segments, {\em Cahiers du B.U.R.O.}, {\bf 15} (1970).

\mbibitem{Lo1} J.-L. Loday, {Dialgebras},
in Dialgebras and related operads, {\em Lecture Notes in Math.},
{\bf{1763}}, (2001), 7-66, arXiv:math.QA/0102053.

\mbibitem{L-R1} J.-L. Loday and M. Ronco, {Trialgebras and families of polytopes,}
 in ``Homotopy Theory: Relations with Algebraic Geometry, Group
 Cohomology, and Algebraic K-theory" Contemporary Mathematics, 346, (2004), 369-398.

\bibitem{Ma} S. MacLane, { ``Categories for the Working
    Mathematician,''} Springer-Verlag, New York, 1971.

\mbibitem{Ro1} G. Rota, {Baxter algebras and combinatorial identities I,} {\em Bull. Amer. Math. Soc.,} {\bf 5}, 1969, 325-329.

\mbibitem{Ro2} G. Rota, {Baxter operators, an introduction,}
  In: ``Gian-Carlo Rota on Combinatorics, Introductory papers
  and commentaries", Joseph P.S. Kung, Editor,
  Birkh\"{a}user, Boston, 1995, 504-512.

\mbibitem{ST} A. Sapounakis and P. Tsikouras,
    On $k$-colored Motzkin words, {\em Jour. Integer Sequences}, {\bf 7} (2004), Article 04.2.5.

\mbibitem{SGIF} K. P. Shum, Y. Guo, M. Ito and Y. Fong (ed.), Semigroups, the International Conference on Semigroups and its Related Topics held at Yunnan University, Kunming, August 18--23, 1995.  Springer-Verlag, Singapore, 1998.

\mbibitem{Si} M. Singer, Talk at the Second International Workshop on Differential Algebra and Related Topics, April 12-13, 2007, Rutgers University - Newark, New Jersey.

\bibitem{S-P1} M. Singer and M. van der Put, ``Galois Theory of Difference Equations", Lecture Notes in Mathematics 1666, Springer, 1997.

\bibitem{S-P} M. Singer and M. van der Put, ``Galois Theory of Linear Differential Equations", Springer, 2003.

\mbibitem{St} R. R. Stanley, Enumerative Combinatorics, Vol. 2, Cambridge University Press, 1999.

\mbibitem{We} E.~W. Weisstein, "Tree." From MathWorld, http://mathworld.wolfram.com/Tree.html



\end{thebibliography}
\end{document}